\documentclass[preprint]{imsart}
\RequirePackage[OT1]{fontenc}
\usepackage{amsthm,amsmath}
\RequirePackage[numbers]{natbib}
\RequirePackage[colorlinks,citecolor=blue,urlcolor=blue]{hyperref}
\usepackage{amsfonts} \usepackage{amsmath} \usepackage{amssymb}
\usepackage{amsthm}
\usepackage{mathabx}
\usepackage{tabularx}
\usepackage{caption}
\usepackage{rotating}
\usepackage{graphicx} \usepackage{enumerate} \usepackage{multicol}
\usepackage{mathrsfs} \usepackage{xypic}
\usepackage{verbatim,enumerate}

\usepackage{geometry,icomma,color,multirow}
\DeclareMathAlphabet{\mathpzc}{OT1}{pzc}{m}{it}

\newcounter{dummy} \numberwithin{dummy}{section}
\newtheorem{theorem}[dummy]{Theorem}
\newtheorem{corollary}[dummy]{Corollary}
\newtheorem{lemma}[dummy]{Lemma}
\newtheorem{definition}[dummy]{Definition}
\newtheorem{proposition}[dummy]{Proposition}
\theoremstyle{remark}
\newtheorem{remark}[dummy]{Remark}
\newtheorem{example}[dummy]{Example}

\DeclareFontFamily{U}{mathx}{\hyphenchar\font45}
\DeclareFontShape{U}{mathx}{m}{n}{
      <5> <6> <7> <8> <9> <10>
      <10.95> <12> <14.4> <17.28> <20.74> <24.88>
      mathx10
      }{}
\DeclareSymbolFont{mathx}{U}{mathx}{m}{n}
\DeclareFontSubstitution{U}{mathx}{m}{n}
\DeclareMathAccent{\widecheck}{0}{mathx}{"71}
\DeclareMathAccent{\wideparen}{0}{mathx}{"75}

\DeclareMathOperator{\E}{\mathbb E}
\DeclareMathOperator{\p}{\mathbb P}

\newcommand\newbullet{{\kern.8pt\displaystyle\cdot\kern.8pt}}

\numberwithin{equation}{section}

\allowdisplaybreaks

\begin{document}%\today

\title{Weak convergence on Wiener space:\\ targeting the first two chaoses}
%\author{Christian Krein}

\begin{aug}
% indicate corresponding author with \corref{}
 \author{\fnms{Christian} \snm{Krein}\ead[label=e1]{christian.krein@education.lu}}
% \thankstext{t2}{Thanks to somebody} 
 \address{Mathematics Research Unit\\ University of Luxembourg \\ Luxembourg\\ \printead{e1} }
\end{aug}

\runauthor{C. Krein}
\runtitle{Weak convergence on Wiener space: targeting the first two chaoses}

\begin{abstract}
\noindent 
We consider sequences of random variables living in a finite sum of Wiener chaoses. We find necessary and sufficient conditions for convergence in law to a target  variable   living in the sum of the first two Wiener chaoses. Our conditions hold notably for sequences of multiple Wiener integrals.
Malliavin calculus and in particular the $\Gamma$-operators are used. Our results extend previous findings  
by Azmoodeh, Peccati and Poly (2014) and are applied to central and non-central convergence situations. Our methods are applied as well  to investigate stable convergence. We finally  exclude certain classes of random variables as target variables for sequences living in a fixed Wiener chaos.

\end{abstract}

\begin{keyword}[class=MSC]
\kwd{60F05; 60G15; 60H07}
\end{keyword}
\begin{keyword}
\kwd{Brownian motion, Malliavin calculus, multiple Wiener integrals, weak convergence, limit theorems, cumulants, stable convergence}
\end{keyword}

\maketitle
\tableofcontents

\section{Introduction}
\subsection{Overview} \label{240516a}
The aim of this paper is to provide new criteria for non-central convergence in law  for sequences of polynomial functionals of a Brownian motion $W$. In particular, we consider the  convergence in law of a sequence of random variables $\left\lbrace F_n\right\rbrace_n$ to a target variable $X$, where:
\begin{itemize}
\item the random variables $F_n$   have a representation of the form $F_n=\sum_{p=1}^m I_p(f_{n,p})$ for a fixed $m\ge 2$, where $I_p(\cdot) $ is the Wiener integral of order $p$ with respect to the Brownian motion $W$;
\item the target  variable $X$ lives in the sum of the first two Wiener chaoses  associated with $W$ and  can be represented as:
\[X=I_1(f_1)+I_2(f_2)=a N + \sum_{i=1}^{k_1} b_i(R_i^2-1)+\sum_{i=1}^{k_2} [c_i(P_i^2-1)+d_iP_i],\]
 where all coefficients $b_i$, $c_j$ and $d_j$ are non-zero and $N , R_i,P_j$ are independent standard normal variables for $1\le i \le k_1$ and $1\le j \le k_2$.
We shall see that this representation covers in particular random variables of the form:
\[X=aU_0 + \sum_{i=1}^n \lambda_i (U_i^2-1),\] 
where all coefficients $\lambda_i$ are non-zero, $U_1,\ldots,U_n$ are independent standard normal variables and $U_0$ is a standard normal variable which may be correlated to $U_1,\ldots, U_n$.
\end{itemize}
Our main result (Theorem \ref{230416a}) gives  a necessary and sufficient criterion for the convergence in law to $X$: %of $  \left\lbrace F_n\right\rbrace_n$ to $X$:  

\begin{theorem} \label{230416a}
Consider $0\le k_1, k_2 < \infty$ and 
\begin{align}
X=I_1(f_1)+I_2(f_2)=a N + \sum_{i=1}^{k_1} b_i(R_i^2-1)+\sum_{i=1}^{k_2} [c_i(P_i^2-1)+d_iP_i],\label{071216a}
\end{align}
 where  $N, R_1,\ldots, R_{k_1}, P_1,\ldots,P_{k_2} \stackrel{\textrm{i.i.d.}}{\sim}\mathcal{N}(0,1)$. Suppose that at least one of the parameters $a, k_1, k_2$ is non-zero.
Consider a sequence $\left\lbrace F_n\right\rbrace_n$ of non-zero random variables such that $F_n=\sum_{i=1}^p I_i(f_{n,i})$ for $p\ge 2$ fixed and $\left\lbrace f_{n,i}\right\rbrace_n\subset H^{\odot i}$ for $1\le i \le p$. Define: \[P(x)=x^{1+1_{[a\ne 0]}} \prod_{j=1}^{k_1} (x-b_j)\prod_{j=1}^{k_2}(x-c_j)^2.\]
As $n\to \infty$, the following conditions (a) and (b) are equivalent:
\begin{enumerate}[(a)]
\item \begin{enumerate}[(1)]
\item $\displaystyle{ \kappa_r(F_n)\to \kappa_r(X),\quad\textrm{ for } r=1,\ldots,\textrm{deg}(P)}$,
\item $\displaystyle{ % \E\left[ \left| \E\left[  \sum_{r=1}^{\textrm{deg}(P)} \frac{P^{(r)}(0)}{r!2^{r-1}} M_{r-1} (F_n)\Big| F_n  \right] \right|  \right]=
\E\left[ \left| \E\left[  \sum_{r=1}^{\textrm{deg}(P)} \frac{P^{(r)}(0)}{r!2^{r-1}} \left( \Gamma_{r-1}(F_n) - \E[\Gamma_{r-1}(F_n)] \right) \Big| F_n  \right] \right|  \right] \to 0 }$,
\end{enumerate}
\item $F_n \stackrel{\textrm{Law}}{\to} X$, as $n\to \infty$.
\end{enumerate}
\end{theorem}

 In this paper, we consider functionals of a Brownian motion. By a standard isometry argument (see e.g. \citep[Section 2.2]{nualart2005}), the results immediately extend to the framework of an isonormal Gaussian process on a general real separable Hilbert space $H$.
\begin{comment}
\noindent {
\noindent\textbf{Theorem}\textit{
Consider $k_1, k_2\ge 0$  and $p\ge 2 $ fixed and: 
\[X\stackrel{\textrm{Law}}{=}a N + \sum_{i=1}^{k_1} b_i(R_i^2-1)+\sum_{i=1}^{k_2} [c_i(P_i^2-1)+d_iP_i],\quad N, R_1,\ldots, R_{k_1}, P_1,\ldots,P_{k_2} \stackrel{\textrm{i.i.d.}}{\sim}\mathcal{N}(0,1).\]
Suppose that $b_i\ne 0$ for $1\le i\le k_1$ if $k_1>0$ and $c_i\cdot d_i\ne 0$ for $1\le i \le k_2$ if $k_2>0$.
Let at least one of the parameters $a,k_1, k_2$ be non-zero.
Consider a sequence $\left\lbrace F_n\right\rbrace_n$ of non-zero random variables with $F_n=\sum_{i=1}^p I_i(f_{n,i})$ for  $f_{n,i}\in H^{\odot i}$ and $\sup_n \Vert f_{n,i}\Vert_{H^{\otimes i}}<\infty$ for $1\le i\le p$. Define: \[P(x)=x^{1+1_{[a\ne 0]}} \prod_{j=1}^{k_1} (x-b_j)\prod_{j=1}^{k_2}(x-c_j)^2.\]
As $n\to \infty$, the following (a) and (b) are equivalent:
\begin{enumerate}[(a)]
\item \begin{enumerate}[(1)]
\item $\displaystyle{ \kappa_r(F_n)\to \kappa_r(X),\quad\textrm{ for } r=1,\ldots,\textrm{deg}(P)}$,
\item $\displaystyle{  %\E\left[ \left| \E\left[  \sum_{r=1}^{\textrm{deg}(P)} \frac{P^{(r)}(0)}{r!2^{r-1}} M_{r-1} (F_n)\Big| F_n  \right] \right|  \right]=
\E\left[ \left| \E\left[  \sum_{r=1}^{\textrm{deg}(P)} \frac{P^{(r)}(0)}{r!2^{r-1}} \left( \Gamma_{r-1}(F_n) - \E[\Gamma_{r-1}(F_n)] \right) \Big| F_n  \right] \right|  \right] \to 0 }$,
\end{enumerate}
\item $F_n \stackrel{\textrm{Law}}{\to} X$.
\end{enumerate}
}
\end{comment}

The notations used in this theorem are introduced in Section \ref{prelim}. In particular, the representation of $X$ in Eq.~\eqref{071216a} is detailed in Eq.~\eqref{210416a} and \eqref{070416aaa}, $\kappa_r(Y)$ is the $r$-th cumulant of a random variable $Y$ and the sequence $\left\lbrace \Gamma_i(F_n)\right\rbrace_i$ is defined recursively using Malliavin operators.
In Section  \ref{mainresults} the representation  of the target random variable    in Eq.~\eqref{071216a} is derived.
\\Our main result unifies, generalizes and extends previous findings. More precisely, Theorem \ref{230416a}:
 \begin{itemize}
 \item further extends to a non-central setting  the seminal paper \cite{nualart2005} and in particular \cite[Theorem 5.3.1]{Peccati} by dealing with central limit theorems for sequences living in a finite sum of Wiener chaoses,
\item extends \cite[Theorem 3.4]{ECP2023} by considering a sequence $\left\lbrace F_n \right\rbrace_n$ of random variables which are no longer restricted to the second Wiener chaos but live in a finite sum of Wiener chaoses,
 \item extends \cite[Theorem 1.2]{nourdin2009b}, \cite[Theorem 3.2]{azmoodeh} and \cite[Proposition 1.7]{2016arXiv161202279D} by considering  target variables involving linear combinations of independent $\chi^2$ distributed random variables and adding a possibly correlated normal variable,
\item improves \cite[Theorem 3.2]{azmoodeh} by replacing $L^2$-convergence with $L^1$-convergence and finding thus a necessary and sufficient criterion for convergence in law for a  large class of target variables, in particular for linear combination of independent central $\chi^2$ distributed target variables.
 \end{itemize}
 In addition to these applications which are discussed in Section \ref{applications}, Theorem \ref{230416a} is used to investigate stable convergence in Section \ref{stable}.

\subsection{History and motivation}\label{history}
The study of convergence in law for sequences of multiple Wiener integrals, by variational techniques, has been the object of an intense study in recent years. 
 The starting point of this line of research is \cite{nualart2005}. In this reference and later in \cite{Nualart2008614}, the authors gave necessary and sufficient criteria
for the convergence in law of a sequence of multiple Wiener integrals $I_p(f_{n,p})$  to a standard normal variable $N$: 
\textit{If the functions $f_{n,p}$ are symmetric in the $p\ge 2$ variables with $\E[I_p(f_{n,p})^2]\to 1$, as $n\to \infty$, then   the following conditions are equivalent:
\begin{enumerate}[(i)]
\item $\E[I_p(f_{n,p})^4]\to 3$, as $n\to \infty$,
\item $ \Vert f_{n,p}  \otimes_l f_{n,p} \Vert_{H^{\otimes (2p-2l)}}\to 0  $,
 as $n\to \infty$, for every $l=1,\ldots, p-1$,
 \item $\Vert DI_p(f_{n,p})\Vert_{H}^2 \stackrel{L^2(\Omega)}{\to} p$, as $n\to \infty$,
 \item $I_p(f_{n,p})\stackrel{\textrm{Law}}{\to}N$, as $n\to \infty$.
 \end{enumerate}}
Notice that $D$ is the standard Malliavin derivative operator and $f_{n,p}\otimes_l f_{n,p}$ is the contraction of order $l$, see Section \ref{malliavin}. The equivalence of  (i), (ii) and (iv) has been found in \cite{nualart2005}, the equivalence of either one of these conditions  with (iii) has been proved later in  \cite{Nualart2008614}.
Considering the proof of \cite[Theorem 4]{Nualart2008614}, it is easy to see that condition (iii) above can be replaced by:
\textit{\begin{enumerate}[(i)]
\item[(iii')] $\E\left[ \Vert DI_p(f_{n,p})\Vert_{H}^2 \,|\, I_p(f_{n,p})\right] \stackrel{L^1(\Omega)}{\to} p$, as $n\to \infty$.
\end{enumerate}}
This is remarkable since conditions of this form play a crucial role in \citep{azmoodeh} and in the main result of the present paper. 
 Since this characterisation  has been published,  limit theorems have been extended   beyond standard normal target variables. In \cite{nourdin2009b}, the authors establish, for $p\ge 2$ even,   necessary and sufficient conditions for a sequence $\left\lbrace I_p(f_{n,p})\right\rbrace_n$  of multiple Wiener integrals to converge in law to a Gamma random variable.  The   conditions found by Nourdin and Peccati use  contractions, convergence of the first moments of $I_p(f_{n,p})$ and Malliavin derivatives.
In particular, the results of \cite{nourdin2009b} cover the convergence in law of a sequence of multiple Wiener integrals to a random variable $X_1\stackrel{\textrm{Law}}{=} \sum_{i=1}^k(N_i^2-1)$ with $N_i$ independent standard normal variables and $1\le k<\infty$. $X_1$ has  a centered $\chi^2$ law with $k$ degrees of freedom. The authors also prove that the convergence is stable in this case. More results about stable convergence can be found in \cite{MR2425357} and in the more recent work \cite{2013arXiv1305.3899N}.

For the case $p=2$, linear combinations of independent centered $\chi^2$ distributed random variables are  important since it is known that every element $X_2$ of the second Wiener chaos has a representation of the form $X_2\stackrel{\textrm{Law}}{=}\sum_{i=1}^n \alpha_i (N_i^2-1)$, where $1\le n \le \infty$  and $\left\lbrace N_i\right\rbrace_i$ is a sequence of independent standard normal variables,  see \cite[Theorem 6.1]{janson1997gaussian}. It is proved in \cite{ECP2023}, that every sequence $\left\lbrace I_2(f_{n,p})\right\rbrace_n$ which converges in law has a limit of the form $\alpha_0 N_0+ \sum_{i=1}^n \alpha_i (N_i^2-1)$, where $N_i$ are independent standard normal variables. %In particular for $n<\infty$, Nourdin and Poly provide necessary and sufficient conditions for the convergence in law to a random variable of 
%this form using contractions instead of cumulants.
 In \cite[Theorem 3.4]{ECP2023} the authors use cumulants and a polynomial $Q$ to characterise this convergence in law if $n<\infty$.
 
The idea of using polynomials to find necessary and sufficient conditions for the convergence in law has proved to be useful. 
In \citep{azmoodeh}, the authors consider the  more general problem of finding necessary and sufficient conditions for a sequence $\left\lbrace F_n\right\rbrace_n$ to converge to a random variable with a representation of the form
\begin{align}
X_3\stackrel{\textrm{Law}}{=} \sum_{i=1}^k \alpha_i (N_i^2-1),\quad N_1,\ldots,N_k\stackrel{\textrm{i.i.d.}}{\sim}\mathcal{N}(0,1), \label{231116a}
\end{align}
where $F_n$ are random variables living in a (fixed) finite sum of Wiener chaoses, see Section \ref{wiener}. For $k<\infty$, their main finding, \cite[Theorem 3.2]{azmoodeh} provides a necessary and a sufficient condition for $F_n\stackrel{\textrm{Law}}{\to} X_3$, as $n\to \infty$, in terms of Malliavin operators $\Gamma_i$, defined in Section \ref{cumulants}.

{
Linear combinations of independent centered $\chi^2$  distributed random variables as in Eq.~\eqref{231116a} are of great interest because of their role within the second Wiener chaos. %, see \cite[Chapter 6]{janson1997gaussian}. 
This class of random variables is important in stochastic geometry as well. In \cite{2015arXiv150800353M}, the authors consider the two-dimensional torus and prove the weak convergence of the normalized nodal length of the so-called \lq arithmetic random waves\rq, to a target variable $\mathcal{M}_\eta$ defined by:
\[
\mathcal{M}_\eta := \frac{-1-\eta}{2\sqrt{1+\eta^2}} (X_1^2-1) + \frac{-1+\eta}{2\sqrt{1+\eta^2}}  (X_2^2-1),\quad \eta\in [0,1],
\]
where $X_1, X_2$ are independent standard normal variables. An important element of the proof is the fact that the Wiener chaos expansion of the normalised nodal length is dominated by its fourth order chaos component.
}

Another line of research, which is closely connected to the previous results, investigates the convergence of sequences living in a finite sum of Wiener chaoses by using distances between probability measures. In this context, the distance between two laws $F$ and $G$ is defined as:
\begin{align*}
d_{\mathscr{H}}(F,G)=\sup_{h\in \mathscr{H}}| \E[h(F)] - \E[h(G)] |,
\end{align*}
where $\mathscr{H}$ is a class of functions. Different classes $\mathscr{H}$ lead to distances such as the Wasserstein, total variation and Kolmogorov distance, see \cite[Appendix C]{Peccati} for details. Upper bounds for the total variation and smooth distances are proved in \cite{MR2520122} and \cite{MR3003367}. For most of these results, the distance of a distribution to  a centered normal distribution or a centered $\chi^2$ distribution with $\nu$ degrees of freedom is considered. Recently, in \cite{2016arXiv161202279D} a new estimate is proved for the Wasserstein distance of a distribution to a centered $\chi^2$ distribution with $\nu$ degrees of freedom. In particular, the authors find a new necessary and sufficient criterion for sequences living in a fixed Wiener chaos to converge in law to $X_1$ defined above. 

{
In \cite{2016arXiv160103301A} the authors consider target variables of the  form given in Eq.~\eqref{231116a}:
\begin{align*}
X_3=\sum_{i=1}^k \alpha_i (N_i^2-1), 
\end{align*}
where the coefficients are not necessarily pairwise distinct and $N_1,\ldots, N_k$ are independent standard normal variables. The authors discuss Stein's method for this class of target variables and apply a new and original Fourier-based  approach to derive a Stein-type characterisation. The polynomials used in \cite{ECP2023, azmoodeh} and $\Gamma$-operators, see \cite{azmoodeh, Peccati}, are combined with the integration by parts formula of Malliavin calculus to derive a Stein operator which allows to characterise target variables as in Eq.~\eqref{231116a}. The authors consider a linear combination of $\Gamma$-operators which shall be generalized in the present paper and  the 2-Wasserstein distance. 
In general, the 2-Wasserstein distance between the laws of random vectors $U$ and $V$
is defined as follows:
\begin{align*}
d_{W_2}(U, V):= \left( \inf\E\left[ \Vert X-Y\Vert_d^2 \right]  \right)^{\frac{1}{2}},
\end{align*} 
where the infimum is taken over all joint distributions of $X$ and $Y$ with respective marginals $U$ and $V$,  and $\Vert \cdot \Vert_d$ stands for the Euclidean norm on $\mathbb{R}^d$, 
see \cite[Definition 1.1]{2017arXiv170401376A}. It is  shown that:
\begin{align*}
d_{W_2}(F_n, X_3)\le C \left( \sqrt{\Delta(F_n)}+\sum_{r=2}^{k+2}|\kappa_r(F_n)-\kappa_r(X_3)|\right),
\end{align*} 
where $C$ is independent of $n$, the quantity $\Delta(F_n)$ can be expressed in terms of cumulants $\kappa_r$ and polynomials, the sequence $\left\lbrace F_n\right\rbrace_n$ must satisfy several conditions which hold in particular for sequences living in the second Wiener chaos. It is proved in particular that:
 \begin{align}
d_{W_2}(F_n, X_3)\le C   \sqrt{\Delta(F_n)},\label{131018b}
\end{align} 
if $\dim_{\mathbb{Q}} \textrm{span }\left( \alpha_1^2, \ldots, \alpha_k^2\right)=k$, see Eq.~\eqref{231116a}. This shows that $\Delta(F_n)\to 0 $, as $n\to \infty$, is sufficient for convergence in the 2-Wasserstein metric which implies convergence in law. In \cite{2017arXiv170401376A}, it is shown that the convergence of the cumulants can not be omitted in the general case. In addition to the  upper bound found in \cite{2016arXiv160103301A}, a lower bound for the 2-Wasserstein metric is derived in \cite{2017arXiv170401376A}, namely:
\begin{align*}
d_{W_2}(F_n, X_3)\ge C' \sqrt{\Delta(F_n)},
\end{align*}
where $C'>0$ is independent of $n$ and the sequence $\left\lbrace F_n\right\rbrace_n$ lives in the second Wiener chaos. In other words, for sequences living in the second Wiener chaos, weak convergence to $X_3$ is equivalent to $\Delta(F_n)\to 0$, as $n\to \infty$,  if $\dim_{\mathbb{Q}} \textrm{span }\left( \alpha_1^2, \ldots, \alpha_k^2\right)=k$.
The so-called Stein-Tikhomirov method is considered in \cite{2016arXiv160506819A}. This method can be seen as a combination of Stein's method with other methods to measure the rate of convergence of a sequence of random variables. The authors consider in particular the Stein-type characterisation found in \cite[Theorem 2.1]{2016arXiv160103301A} and apply their version of the Stein-Tikhomirov method. The same linear combination of $\Gamma$-operators as in \cite{2016arXiv160103301A} is used and the so-called transfer-principle allows to find upper bounds on smooth Wasserstein distances using upper bounds on the difference of the characteristic functions of the approximating sequence and the target variable. For sequences living in the sum of the first $p$ Wiener chaoses and constants $C, \Theta > 0$ depending only on $p$, the following bound is proved:
\begin{align*}
d_{W_2}(F_n, X_3)\le C \, \Delta_n \, |\log (\Delta_n)|^\Theta,
\end{align*}
where $\Delta_n$ is expressed in terms of cumulants and $\Gamma$-operators. In particular, if $k\ge 3$ in Eq.~\eqref{231116a}, we have for the Kolmogorov distance that $d_{\textrm{Kol}}(F_n, X_3)\le B \, \sqrt{\Delta_n}$. Finally the authors find bounds for $\Delta_n$ if $\left\lbrace F_n\right\rbrace_n$ lives inside a fixed Wiener chaos and $k=1$ in Eq.~\eqref{231116a}.
Even though the aforementioned papers present important new results for target variables living in the second Wiener chaos, none of them considers target variables living in the sum of the first two Wiener chaoses with possibly correlated first and second order components.
}

The results  of \cite[Theorem 3.2]{azmoodeh} are the starting point of the present work.
As anticipated, we shall consider a sequence $\left\lbrace F_n \right\rbrace_n$ of random variables living in a finite (fixed) sum of Wiener chaoses and provide necessary and sufficient conditions for $F_n \stackrel{\textrm{Law}}{\to} X_4$, as $n\to \infty$, where $k<\infty$ and $X$ has the following, more general form:
\begin{align}
X_4= \sum_{i=1}^k  \left[ { \alpha_i (N_i^2-1)+\beta_i N_i}\right],\quad N_1,\ldots,N_k\stackrel{i.i.d.}{\sim}\mathcal{N}(0,1). \label{250416a}
\end{align}
The representation in Eq.~\eqref{250416a}
is   equivalent to one in Eq.~\eqref{071216a}, where we have dropped all vanishing coefficients and regrouped  independent normal variables. 
Both representations \eqref{071216a} and \eqref{250416a} are useful for the discussion to follow.
Random variables as in Eq.~\eqref{250416a} are important since every random variable living in the sum of the first two Wiener chaoses has a representation of this form, with $k\le \infty$, see \cite[Theorem 6.2]{janson1997gaussian}. Our conditions  make, as in \cite{azmoodeh}, use of the operators $\Gamma_i$. Clearly such a result can be seen as extension of \cite[Theorem 3.2]{azmoodeh}.
For sequences of random variables living in a finite sum of Wiener chaoses, we shall apply our methods and results to derive necessary and sufficient criteria  to prove stable convergence to target variables with representations as in Eq.~\eqref{250416a}

\subsection{Results and plan}
The paper is organized as follows:
\begin{enumerate}[-]
\item In Section \ref{prelim}, we introduce the necessary notations and give a brief introduction to Malliavin calculus. The basic elements of this theory shall be needed in the forthcoming proofs.
\item In Section \ref{mainresults}, we prove our characterisation in Theorems \ref{030416aa} and   \ref{030416aab}. The main Theorem \ref{230416a} is then a direct consequence of these theorems.
\item In Section \ref{applications}, we apply Theorem \ref{230416a} to several situations, such as the convergence in law to a normal variable or a centered $\chi^2$ distributed random variable with $k_1$ degrees of freedom. We shall also recover the results of \cite{azmoodeh}. In Theorem \ref{230416aa}, we give   sufficient conditions, based only on cumulants and contractions. %, for convergence in law to a random variable of the form given in Eq.~\eqref{210416ab} respectively in Eq.~\eqref{070416aaa}.
{ We conclude this section by giving a criterion which excludes certain classes of target variables for sequences living in a chaos of odd order.}
\item In Section \ref{stable}, we give criteria which can be used to determine whether a sequence converges stably. 
%\item In Section \ref{negative}, we prove that certain classes of random variables cannot appear as target variables if our sequence lives in a fixed Wiener chaos of odd order.
\end{enumerate}

\section{Preliminaries} \label{prelim}

\subsection{Multiple Wiener integrals}\label{wiener}
The reader is referred to \citep{Peccati}, \cite{MR2200233} or \cite{MR2460554} for a   detailed introduction to  multiple Wiener integrals.
Consider the real Lebesgue space $H= L^2([0,T], \lambda_T) $, where $\lambda_T $ is the Lebesgue measure on $[0, T]$.
The real separable Hilbert space $H $ is endowed with the standard scalar product
$
\langle h, g\rangle_H:= \int_{0}^{T}h  g d\lambda
$
for all $h, g\in H$. We write $H^{\otimes p}$ for $L^2([0, T]^p ,  \lambda_T^p)$, where $\lambda_T^p:=\lambda^p | [0, T]^p$, and define $H^{\odot p}$ as the subspace of $H^{\otimes p}$ containing exactly the functions which are symmetric on a set of Lebesgue measure $T^p$.
Consider a complete probability space $\left( \Omega, \p, \mathcal{F}\right)$ and a standard Brownian motion $(W_t)_{t\in [0, T]}$ with respect to $\p$ and the filtration $(\mathcal{F}_t)_{t\in [0, T]}$. 
Define for every $h\in H$:
\[
W(h)=I_1(h)=\int_0^T h(t)dW_t,
\]
then  $W(h)\in L^2 (\Omega):= L^2 \left( \Omega, \p\right)$, in other words $W(h)$ is square-integrable.
 We have for $h, g\in H$: 
\[
\E[W(h)\,W(g)]= \E[I_1(h)\, I_1(g)]= \langle h, g \rangle_H.
\]
More generally the $q$-th Wiener chaos is defined as closed linear subspace of $L^2(\Omega)$ which is generated by the random variables of the form $H_q(W(h))$ where $h\in H$ with $\Vert  h \Vert_H=1 $ and $H_q$ is the $q$-th Hermite polynomial. The elements of the $q$-th Wiener chaos can be represented as multiple Wiener integrals. For every  $f_p\in H^{\odot p}$:
\begin{align*}
 I_p(f_p)&:=\int_{[0, T]^p}f_p(t_1,  \ldots,  t_p)dW_{t_1}  \ldots   dW_{t_p} \\
 &:=p! \int_0^T \left(\int_0^{t_p} \ldots  \int_0^{t_3}\left(\int_0^{t_2} f_p(t_1, \ldots, t_p)dW_{t_1}\right) dW_{t_2} \ldots  dW_{t_{p-1}} \right) dW_{t_p}. 
\end{align*}
It is well known that every $F\in L^2(\Omega)$ has a representation of the form $F=
\sum_{p=0}^\infty I_p(f_p)$,
where $I_0(f_0)=f_0=\E[F]$ and the right-hand side  converges in $L^2(\Omega)$. We have moreover for $p, q\ge 1$:
\begin{align}
\E[I_p(f_p)\,I_q(g_q)]= 1_{[p=q]} \,p!\, \langle f_p , g_q \rangle_{H^{\otimes p}}=1_{[p=q]}\int_{[0, T]^p}f_p\, g_q\, d\lambda_T^p. \label{260416x}
\end{align}
For two function $f_p\in H^{\odot p}$ and $g_q\in H^{\odot q}$, the contraction of $r$ indices is defined for $1\le r\le p\wedge q$ by:
\begin{align*}
(f_p\otimes_r g_q )(t_1, \ldots, t_{p+q-2r})&= \int_{[0, T]^r}f_p(t_1, \ldots, t_{p-r}, s)g_q(t_{p-r+1}, \ldots, t_{p+q-2r}, s) \,d\lambda^r (s).
\end{align*}
We have $f_p\otimes_r g_q \in H^{p+q-2r}$. The symmetrization of $f_p\otimes_r g_q$ is $f_p\tilde{\otimes}_r g_q$.
We shall also need the multiplication formula for multiple Wiener integrals. For $f_p\in H^{\odot p} $ and $g_q\in H^{\odot q}$, we have:
\[
I_p(f_p) I_q(g_q)= \sum_{r=0}^{p\wedge q} r! \begin{pmatrix}
p\\r
\end{pmatrix}\, \begin{pmatrix}
q\\r
\end{pmatrix} I_{p+q-2r}(f_p\tilde{\otimes}_{r} g_q).
\]

\subsection{Malliavin calculus}\label{malliavin}
The reader is referred to \citep{Peccati}, \cite{MR2200233} or \citep{azmoodeh} for a   detailed introduction to  Malliavin calculus.
Let $k\ge 1 $ and $F$ be  a random variable with $F=f\left( W(h_1), \ldots
, W(h_k)\right)$ where 
 $f$ is an infinitely differentiable rapidly decreasing function
on $\mathbb{R}^k$ and
$h_1, \ldots, h_k\in H$.
Then $F$ is called a \textit{smooth random variable} and $\mathcal{S}$ is the set containing exactly the smooth random variables.
The Malliavin derivative of $F$ is defined by:
\begin{align}
DF:=&\sum_{i=1}^k {h}_i(t)\, \partial_i f (W(h_1), \ldots
, W(h_k)) \nonumber\\
  &=\sum_{i=1}^k {h}_i(t)\, \partial_i f
\left(\int_0^T {h}_i(s)dW_s, \ldots ,  \int_0^T 
{h}_k(s)dW_s\right). \label{191214}
\end{align}
We have that ${D}_t$ is closable from $L^2(\Omega, \p)$ to
$L^2(\Omega\times[0, T], \p\otimes\lambda_T)$, that is (see \cite
{MR2460554} or \cite[Proposition 1.2.1]{MR2200233}): \textit{If a
sequence $\left\lbrace H_n \right\rbrace_{n\in\mathbb{N}} \subset L^2\left(
\Omega, \p\right)$ converges to $0$, that is $\E\left[ H_n^2
\right]\to 0$ as $n\to\infty$,
and
${D}_t H_n$ converges in $L^2(\Omega\times[0, T], \p\otimes\lambda
_T)$ as $n\to\infty$, then $\lim_{n\to\infty}{D}_t H_n = 0 $.} We
write Dom ${D}$ for the closed domain of ${D}$.
Moreover the Malliavin derivative has a closable adjoint $\delta$
(under $\p$). The operator $\delta$ is called the \textit
{divergence operator} or, in the white noise case, the Skorohod
integral. The domain of $\delta$ is denoted by Dom $\delta$, it is
the set of square-integrable random variables $v\in L^2(\Omega\times
[0,T], \p\otimes\lambda_T)$ with:
\[
\E\left[ \langle{D}F, v\rangle_{H} \right]
\le c_v\, \sqrt{\E[F^2]},
\]
for a constant $c_v$ (depending on $v$) and all $F\in \mathbb{D}^{1,2}$ where
$\mathbb{D}^{1,2}$ is the closure of the class of smooth random variables with
respect to the norm:
\[
\Vert F \Vert_{{1,2}}:= \left( \E\left[ F^2 \right] + \E
\left[ \Vert{D} F\Vert_{H}^2 \right] \right)^{1/2}.
\]
With the scalar product
$
\langle F, G\rangle_{1, 2}=\E[FG]+ \E[\langle{D}F, {D}G\rangle
_{H}],
$
$\mathbb{D}^{1, 2}$ is a Hilbert space. If $v\in$ Dom $\delta$, then $\delta(v)$ is the element of
$L^2(\Omega, \p)$ characterised by
\begin{align}
\E[F\delta(v)]=\E[\langle v, {D}F\rangle_{H}]. \label{int_parts}
\end{align}
This relation is often called the \textit{integration by parts
formula}. We have the more general rule (see \cite[Proposition
2.5.4]{Peccati}) for $F\in \mathbb{D}^{1, 2}$, $v\in$ Dom
$\delta$ such that $Fv\in$ Dom $\delta$:
\begin{align}
\delta(Fv)=F\delta(v)- \langle{D}F, v\rangle_{H}
. \label{int_parts_0}
\end{align}
For multiple Wiener integrals and $f\in H^{\odot k}$, we have (see for instance \cite[p.35]{MR2200233}):
\begin{align}
&~{D}_{x_1}{D}_{x_2}\ldots{D}_{x_l} \int_0^T \ldots\int_0^T f 
(y_1, \ldots, y_k)dW_{y_1}\ldots dW_{y_k}\label{24082014a} \\
&=
\frac{k!}{(k-l)!}
\, \int_0^T \ldots\int_0^T f (y_1, \ldots, y_{k-l}, x_1, \ldots 
, x_l)dW_{y_1}\ldots dW_{y_{k-l}} .\nonumber
\end{align}
Formula \eqref{int_parts} can be generalized for the
multiple divergence (see \cite[p.33]{Peccati}): If $v\in$ Dom $\delta
^l$ and $F\in \mathbb{D}^{l, 2}$:
\begin{align}
\E\left[ F \delta^l (v) \right]=\E[\langle{D}^l F , v\rangle
_{H^{\otimes l}}], \label{031115c}
\end{align}
see \cite{Peccati} or \cite{MR2200233} for details. %For the ease of
%notations, we have used the short-form $\lambda^n_T:=\lambda^n|[0,T]^n$.
For $m\ge 1$, $p\ge 1 $ and the $m$-th Malliavin derivative $D^m F$, we can define  $\mathbb{D}^{m, p}$ as the closure of the class of smooth random variables with respect to the norm $\Vert \cdot \Vert_{m, p}$ defined by:
\[
\Vert F \Vert_{m, p} = \left( \E[|F|^p]  +\sum_{i=1}^m \E\left[\Vert D^i F \Vert^p_{H^{\otimes i}} \right] \right)^{1/p}
\]
We define $\mathbb{D}^\infty:= \cap_{m=1}^\infty \cap_{p=1}^\infty \mathbb{D}^{m,p}$.

\subsection{Cumulants and $\Gamma$-operators}\label{cumulants}
The reader is referred to \citep{azmoodeh} or \cite{Peccati} for a   detailed introduction to  cumulants, Malliavin operators and $\Gamma$-operators in particular.
%\textit{See  for details.} 
The $r$-th cumulant of a random variable $F$ exists if the $r$-th moment of $F$ exists and is defined as $\kappa_r(F):= (-i)^r \frac{d^r}{dt^r}\log \E[\exp(itF)]_{|t=0}$.    
The operator $L$, defined as $L=-\sum_{q=0}^\infty q J_q$, is the infinitesimal generator of the Ornstein-Uhlenbeck semi-group where $J_q$ is the orthogonal projection operator on the $q$-th Wiener chaos. The domain of $L$ is $\mathbb{D}^{2, 2}$. $L$ admits a pseudo-inverse $L^{-1}$ and for any $F\in L^2(\Omega)$, we have $L^{-1}F=-\sum_{q=1}^\infty \frac{1}{q} J_q(F)$.
For $F\in \mathbb{D}^\infty$, the sequence of random variables $\left\lbrace\Gamma_i(F)\right\rbrace_i\subset \mathbb{D}^\infty$ is recursively defined as follows:
\[
\Gamma_i(F)=\langle DF, -DL^{-1}\Gamma_{i-1}(F)\rangle_H,\quad \textrm{for } i\ge 1,
\]
and $\Gamma_0(F)=F$. For $F\in \mathbb{D}^\infty$ and $r\ge 0$, we define:
$
M_r(F)=\Gamma_r(F)-\E[\Gamma_r(F)]$.

\subsection{Stable convergence}\label{stableconvergence}
The concept of stable convergence is used in Section \ref{stable}. The reader can find  an extensive discussion of this topic in \cite{MR1943877} or \cite{MR3362567}, the basic facts are resumed in \cite{2013arXiv1305.3899N}.
Consider a sequence $\left\lbrace F_n\right\rbrace_n$ of real random variables on the complete probability space $\left(\Omega, \mathcal{F}, \p \right)$, see Section \ref{wiener}. Let $F$ be a real random variable defined on some extended probability space $\left( \Omega', \mathcal{F}', \p'\right)$. We say that $F_n$ converges stably to $F$, written $F_n \stackrel{\textrm{st.}}{\to}F$, as $n\to \infty$, if:
\begin{align*}
\lim_n \E\left[ Z \exp\left( i \lambda F_n \right) \right] = \E'\left[ Z\exp\left( i \lambda F \right) \right],
\end{align*}
for every $\lambda\in \mathbb{R}$ and every bounded $\mathcal{F}$-measurable random variable $Z$.
Obviously, stable convergence implies convergence in law, whereas the converse  does not hold in general.
We notice that the $\p$-completion of the $\sigma$-field generated by the set $\left\lbrace I_1(f) : f\in H \textrm{ with } \Vert f \Vert_H=1 \right\rbrace$ is $\mathcal{F}$. We have thus the following useful characterisation of stable convergence:
\begin{center}
$ F_n\stackrel{\textrm{st.}}{\to} F$ if and only if  $\left( F_n, I_1(f) \right)^\top \stackrel{\textrm{Law}}{\to} \left( F, I_1(f) \right)^\top$ for every $f\in H$ with $\Vert f \Vert_H=1$.
\end{center}

\section{Main results}\label{mainresults}
We start with an example in order to motivate the reader.

\begin{example}
Consider sequences $\left\lbrace h_{n,2}\right\rbrace_n,\left\lbrace f_{n,2}^{(i)}\right\rbrace_n, \left\lbrace g_{n,2}^{(j)}\right\rbrace_n \subset H^{\odot 2}$ for $1\le i\le k_1$ and $1\le j\le k_2$ for $k_1>0$ and $k_2>0$.  {Suppose that the sequences $\left\lbrace I_2(f_{n,2}^{(i)})\right\rbrace_n$,  $\left\lbrace I_2(g_{n,2}^{(j)})\right\rbrace_n$ and $\left\lbrace I_2(h_{n,2})\right\rbrace_n$ converge in law to a standard normal variable  and that $
 \langle k_{n,2} , k'_{n,2}\rangle_{H\otimes H}\to 0$, 
 %~,~ \textrm{ for } \left\lbrace k_{2,n} \right\rbrace_n  , \left\lbrace k'_{2,n} \right\rbrace_n \in \left\lbrace  \left\lbrace h_{n,2} \right\rbrace_n , \left\lbrace f_{n,2}^{(i)} \right\rbrace_n ,\left\lbrace g_{n,2}^{(j)} \right\rbrace_n, 1\le i \le k_1, 1\le j \le k_2\right\rbrace , 
%\]
 for any distinct sequences  $\left\lbrace k_{2,n} \right\rbrace_n  $ and $  \left\lbrace k'_{2,n} \right\rbrace_n$ chosen among $\left\lbrace h_{n,2}\right\rbrace_n,\left\lbrace f_{n,2}^{(i)}\right\rbrace_n, \left\lbrace g_{n,2}^{(j)}\right\rbrace_n $ for $1\le i\le k_1$ and $1\le j\le k_2$.
}
%
%
%
%Suppose that: \[\sup_n \Vert f_{n,2}^{(i)}\Vert_{H\otimes H}<\infty~,~\sup_n \Vert g_{n,2}^{(j)}\Vert_{H\otimes H}<\infty~,~\sup_n \Vert h_{n,2}\Vert_{H\otimes H}<\infty,\]
% for $1\le i\le k_1$ and $1\le j \le k_2$ and, as $n\to \infty$:
%\[
%I_2(f_{n,2}^{(i)})\stackrel{\textrm{Law}}{\to}U_i ~,~ I_2(g_{n,2}^{(j)})\stackrel{\textrm{Law}}{\to}V_j ~, ~I_2(h_{n,2})\stackrel{\textrm{Law}}{\to}W, \quad U_i,V_j,W \sim \mathcal{N}(0,1),
%\]
%with $
% \langle k_{n,2} , k'_{n,2}\rangle_{H\otimes H}\to 0$, 
% %~,~ \textrm{ for } \left\lbrace k_{2,n} \right\rbrace_n  , \left\lbrace k'_{2,n} \right\rbrace_n \in \left\lbrace  \left\lbrace h_{n,2} \right\rbrace_n , \left\lbrace f_{n,2}^{(i)} \right\rbrace_n ,\left\lbrace g_{n,2}^{(j)} \right\rbrace_n, 1\le i \le k_1, 1\le j \le k_2\right\rbrace , 
%\]
% for any distinct sequences  $\left\lbrace k_{2,n} \right\rbrace_n  $ and $  \left\lbrace k'_{2,n} \right\rbrace_n$ chosen among $\left\lbrace h_{n,2}\right\rbrace_n,\left\lbrace f_{n,2}^{(i)}\right\rbrace_n, \left\lbrace g_{n,2}^{(j)}\right\rbrace_n $ for $1\le i\le k_1$ and $1\le j\le k_2$.
We have then, see \cite[Theorem 1]{MR2126978}, as $n\to \infty$: 
\begin{align*}
(I_2(h_{n,2}), I_2(f_{n,2}^{(1)}), \ldots , I_2(f_{n,2}^{(k_1)}) , I_2(g_{n,2}^{(1)}), \ldots , I_2(g_{n,2}^{(k_2)}))^\top \stackrel{\textrm{Law}}{\to} Z^\top,
\end{align*}  
where $Z^\top:=(N, R_1, \ldots, R_{k_1}, P_1,\ldots, P_{k_2})^\top$ has a $k_1+k_2+1$-dimensional standard normal distribution. The continuous mapping theorem yields that, as $n\to \infty$:
\begin{align}
&I_2(h_{n,2})+ \sum_{i=1}^{k_1} b_i \left( I_2(f_{n,2}^{(i)})^2-1 \right) + \sum_{j=1}^{k_2} \left[ c_j \left( I_2(g_{n,2}^{(j)})^2-1 \right) + d_j I_2(g_{n,2}^{(j)})\right]\label{241116a}\\
&\stackrel{\textrm{Law}}{\to} aN+ \sum_{i=1}^{k_1} b_i(R_i^2-1) + \sum_{j=1}^{k_2} \left[ c_j (P_j^2-1)+ d_jP_j \right], \label{180816aa}
\end{align}
where $N, R_1, \ldots,  R_{k_1}, P_1, \ldots, P_{k_2}\stackrel{\textrm{i.i.d.}}{\sim}\mathcal{N}(0,1)$.
It is easy to see that the expression in Eq.~\eqref{241116a} has a representation of the form $I_4(\varphi_{n, 4})+I_2(\varphi_{n,2})+\varphi_{n, 0}$ for $\varphi_{n, l}\in H^{\odot l}$, $l=0, 2, 4$, and $\lim_n \varphi_{n, 0}=0$. We find with Slutsky's theorem that  $I_4(\varphi_{n, 4})+I_2(\varphi_{n, 2})\stackrel{\textrm{Law}}{\to}aN+ \sum_{i=1}^{k_1} b_i(R_i^2-1) + \sum_{j=1}^{k_2} \left[ c_j (P_j^2-1)+ d_jP_j \right]$, as $n\to \infty$.
\end{example}

\begin{remark}\label{250416b}
In the present paper we give  necessary and sufficient conditions for weak convergence towards target variables as in Eq.~\eqref{180816aa}.
We illustrate now why this class of target variables is important.

Consider a  random variable $X$ living in the sum of the first two Wiener chaoses: 
\begin{align}
X:=I_1(f_1)+ I_2(f_2),\label{210416a}
\end{align}
 with $f_1\in H$ and $f_2\in H^{\odot 2}$. %In other words, the random variable $X$ lives in the sum of the first two chaoses. 
It is known, see  \cite[Theorem 6.1]{janson1997gaussian} or \cite[Proposition 2.7.13]{Peccati}, that $I_2(f_2)=\sum_{k=1}^\infty \alpha_k (I_1(h_k)^2-1)$ for an orthonormal system $\left\lbrace h_i | i\in\mathbb{N}\right\rbrace \subset H$. Suppose from now on that $\alpha_k\ne 0$ for $1\le k \le N$ and $\alpha_k=0$ for every $k\ge N$. Consider the projection of $f_1$ on $\textrm{span}(h_1, \ldots, h_N)$, then we find  $f_1=\beta_1 h_1+\ldots + \beta_N h_N+ \beta_0 h_0$, where $\Vert h_0\Vert_H=1$ and $h_0\perp \textrm{span }(h_1, \ldots, h_N)$. Hence:
\begin{align}
X&= \beta_0 I_1(h_0) +  \sum_{i=1}^N \left[ \beta_i I_1(h_1) + \alpha_i(I_1(h_i)^2-1)\right]. \label{141018a}
\end{align}
Since $\langle h_i , h_j\rangle_H=\delta_{i,j}$, we have that $\left\lbrace I_1(h_0),\ldots, I_1(h_N)\right\rbrace$ is a  set of independent standard normal variables and some of the coefficients $\beta_0, \ldots, \beta_N$ may be equal to 0. The representation of $X$ found in Eq.~\eqref{141018a} is thus equivalent to the representation in Eq.~\eqref{180816aa}:
\begin{align}
X&= a N + \sum_{i=1}^{k_1} b_i(R_i^2-1)+\sum_{i=1}^{k_2} [c_i(P_i^2-1)+d_iP_i],%\quad b_i\ne 0,\, c_j d_j\ne 0 \textrm{ for } 1\le i \le k_1~,~1\le j \le k_2, 
\label{070416aaa}
\end{align}
where $b_i\ne 0,\, c_j d_j\ne 0$  for  $1\le i \le k_1~,~1\le j \le k_2$ and
 $N, R_i, P_j$ are independent standard normal variables living in the first Wiener chaos.
 A similar argument applied together with the multiplication formula for Wiener integrals shows that $X=I_1(f_1)+I_2(f_2)$ for: 
\begin{align}
f_1=  ah_0 +\sum_{i=1}^{k_2}d_i h'_i  ~,~ f_2= \sum_{i=1}^{k_1} b_i h_i \tilde{\otimes} h_i + \sum_{i=1}^{k_2} c_i h'_i \tilde{\otimes} h'_i,\label{231018}
\end{align}
for a set of orthonormal functions $h_0, h_1, \ldots, h_{k_1},  h_1', \ldots, h_{k_2}'$.

\color{black}

If $k_1=0$, we set $b_j=0$ for every $j$ and the empty sum in the representations above is removed. Notice that an empty product equals 1.  We proceed similarly if $k_2=0$. If on the other hand $k_1\ne 0$, we suppose that $b_j\ne 0$ for every $j=1,\ldots , k_1$. We proceed similarly if $k_2\ne 0$.

\end{remark}

The following Lemma shows that random variables with a representation as in Eq.~\eqref{070416aaa} extend the class of random variables with a representation as in Eq.~\eqref{231116a} by adding a (possibly correlated) normal variable.

\begin{lemma} \label{230516c}
Consider the following families of random variables:
\begin{enumerate}[(A)]
\item $X\stackrel{\textrm{Law}}{=} aN  + \sum_{i=1}^{k_1}b_i (R_i^2-1)+ \sum_{i=1}^{k_2}\left[ c_i \left( P_i^2-1 \right)+d_i P_i \right]$,
where $N, R_1, \ldots, R_{k_1}$, $P_1, \ldots, P_{k_2}\stackrel{\textrm{i.i.d.}}{\sim} \mathcal{N}(0,1)$ and $b_i\ne 0$ for $1\le i \le k_1$ if $k_1>0$ and $c_i d_i\ne 0$ for $1\le i\le k_2$ if $k_2>0$.
\item $Y\stackrel{\textrm{Law}}{=}aU_0 + \sum_{i=1}^n \lambda_i (U_i^2-1)$ where $\lambda_i\ne 0$ for $1\le i \le n$ if $n>0$ and  $(U_0, U_1, \ldots  , U_n)^\top$
 is a centered normal vector such that $U_1,\ldots,  U_n\stackrel{\textrm{i.i.d.}}{\sim}\mathcal{N}(0,1)$.%\begin{align*}
%(U_0, U_1, \ldots, U_n)^\top \sim \mathcal{N}(0_{(n+1)\times 1}, \Sigma)~,~\Sigma=
% \left( 
%\begin{array}{c@{}|c@{}}
% \begin{array}{c}
%      1 \\
%  \end{array} & \begin{array}{ccc}
%             \sigma_1  & \cdots & \sigma_n \\
%  \end{array}    \\ 
%  \hline
%    \begin{array}{c}
%                         \sigma_1\\ 
%                       \vdots   \\
%                        \sigma_n \\
%                      \end{array}    & \mathbf{I_{n\times n}}
%\end{array}\right),
%\end{align*}
%and $\Sigma$ is positive definite.
\end{enumerate}
Then class (A) coincides with class (B). In other words every random variable in (A) has a representation as in (B) and vice versa.
%\begin{enumerate}[(1)]
%\item The random variable $X$ in (A) has a representation as in (B) if $a\ne 0$ or %%$k_2=0$. 
%\item The random variable $Y$ in (B) has a representation as in (A).
%\end{enumerate}
\end{lemma}

%\begin{remark}
%In other words, Lemma \ref{230516c} states that   class (A) is larger than   class (B). The random variables of class (B) may however be of practical interest. For the sake of generality, we shall consider in this paper target random variables of class (A).
%\end{remark}

\begin{proof}

\begin{enumerate}[(1)]
\item Consider $X\stackrel{\textrm{Law}}{=}aN +\sum_{i=1}^{k_1} b_i (R_i^2-1) +\sum_{i=1}^{k_2}[c_i(P_i^2-1)+d_iP_i]$ as in (A), then, if $a^2+\sum_{i=1}^{k_2}d_i^2\ne 0$:
  \begin{align*}
  X&\stackrel{\textrm{Law}}{=}\left(aN + \sum_{i=1}^{k_2}d_i P_i\right) +\sum_{i=1}^{k_1} b_i (R_i^2-1) +\sum_{i=1}^{k_2} c_i(P_i^2-1) \\
  &\stackrel{\textrm{Law}}{=}\sqrt{a^2+\sum_{i=1}^{k_2}d_i^2}\,\frac{aN + \sum_{i=1}^{k_2}d_i P_i}{\sqrt{a^2+\sum_{i=1}^{k_2}d_i^2}}+\sum_{i=1}^{k_1} b_i (R_i^2-1) +\sum_{i=1}^{k_2} c_i(P_i^2-1).
  \end{align*}
Define:\[
U_i = \begin{cases}  \frac{aN + \sum_{i=1}^{k_2}d_i P_i}{\sqrt{a^2+\sum_{i=1}^{k_2}d_i^2}}& \textnormal{\textit{for }} i=0,\\ 
R_i & \textnormal{\textit{for }} 1\le i \le k_1,\\ 
 P_{i-k_1}&\textnormal{\textit{for} } k_1+1 \le i \le k_1+k_2, 
\end{cases}
\] 
and drop the corresponding terms if $k_1=0$ or $k_2=0$. After renaming the coefficients, we find the following representation with $n:=k_1+k_2$:
\[
X\stackrel{\textrm{Law}}{=}\sqrt{a^2+\sum_{i=1}^{k_2}d_i^2} U_0 + \sum_{i=1}^n \lambda_i(U_i^2-1),
\]
and $(U_0, U_1,\ldots, U_n)^\top $ is clearly centered and normal since every linear combination $\sum_{i=1}^n{\alpha_i U_i}$ is normal. This last property follows directly from $N, R_1, \ldots , R_{k_1}$, $P_1, \ldots, P_{k_2}\stackrel{\textrm{i.i.d.}}{\sim}\mathcal{N}(0,1)$. The definition of $U_1, \ldots , U_n$ yields that $U_1, \ldots, U_n \stackrel{\textrm{i.i.d.}}{\sim} \mathcal{N}(0, 1)$. If $a^2+\sum_{i=1}^{k_2}d_i^2=0$, then $a=d_1=\ldots=d_{k_2}=k_2=0$. The equivalence of both representations is trivial in this case.

\item Consider $Y\stackrel{\textrm{Law}}{=}aU_0 + \sum_{i=1}^n \lambda_i( U_i^2-1)$ as in (B). Since the case $a=0$ is clear, we suppose that $a\ne 0 $ and $U_0\sim\mathcal{N}(0,1)$. Let $U:=(U_0, U_1, \ldots, U_n)^\top\sim \mathcal{N}(0_{(n+1)\times 1},  \Sigma)$ for a positive semi-definite matrix $\Sigma$. Since $U_1, \ldots , U_n\stackrel{\textrm{i.i.d.}}{\sim}\mathcal{N}(0,1)$, we have:
\[\Sigma=
 \left( 
\begin{array}{c@{}|c@{}}
 \begin{array}{c}
      1 \\
  \end{array} & \begin{array}{ccc}
             \sigma_1  & \cdots & \sigma_n \\
  \end{array}    \\ 
  \hline
    \begin{array}{c}
                         \sigma_1\\ 
                       \vdots   \\
                        \sigma_n \\
                      \end{array}    & \mathbf{I_{n\times n}}
\end{array}\right).\]
We suppose first that $\det \Sigma=1-\sum_{i=1}^n\sigma_i^2>0 $ then $\Sigma=B \, B^\top $ for $B$ defined by:
\[
B:=  \left( 
\begin{array}{c@{}|c@{}}
 \begin{array}{c}
      \sqrt{\det \Sigma} \\

  \end{array} & \begin{array}{ccc}
             \sigma_1  & \cdots & \sigma_n \\
  \end{array}    \\ 
  \hline
    \begin{array}{c}
             \mathbf{0_{n \times 1}}
                      \end{array}    & \mathbf{I_{n\times n}}
\end{array}\right)\in \mathbb{R}^{(n+1)\times(n+1)}.
\]
Consider a $n+1$-dimensional standard normal vector
 $V^\top= (V_0, V_1, \ldots, V_{n})^\top$, then $U\stackrel{\textrm{Law}}{=}BV$. Hence:
 \begin{align}
 Y &\stackrel{\textrm{Law}}{=} a \left( \sqrt{\det \Sigma} V_0 + \sigma_1 V_1 + \ldots + \sigma_n V_n \right) + \sum_{i=1}^n \lambda_i (V_i^2-1) \nonumber \\
&\stackrel{\textrm{Law}}{=}a \sqrt{\det \Sigma} V_0 + \sum_{i=1}^n  \left[ \lambda_i(V_i^2-1) + a\sigma_i V_i \right].  \label{230516a}
 \end{align}
Noticing that some of the covariances $\sigma_i$ may be zero, Eq.~\eqref{230516a} yields the representation (A) after renaming the independent standard normal random variables and the coefficients. 

Consider now the case $\det \Sigma=0$, a standard normal vector $(V_1, \ldots, V_n)^\top$ and define $B\in \mathbb{R}^{(n+1)\times n}$:
\[
B:=  \left( 
\begin{array}{c@{}}
 \begin{array}{ccc}
             \sigma_1  & \cdots & \sigma_n \\
  \end{array}    \\ 
  \hline
 \mathbf{I_{n\times n}}
\end{array}\right)\in \mathbb{R}^{(n+1)\times n}.
\]
Then $B\, B^\top = \Sigma$ and $U\stackrel{\textrm{Law}}{=}BV$, hence:
\begin{align*}
%a  U_0 + \sum_{i=1}^n \lambda_i (U_i^2-1)
Y &\stackrel{\textrm{Law}}{=}a\,\left( \sigma_1 V_1 + \ldots + \sigma_n V_n \right) + \sum_{i=1}^n \lambda_i (V_i^2-1)
\stackrel{\textrm{Law}}{=} \sum_{i=1}^n \left[ \lambda_i (V_i^2-1)+ a\sigma_i V_i\right].
\end{align*}
The statement follows now as above.
\end{enumerate}\end{proof}

\begin{remark}
The random variable $X$ in (A) may lead to a degenerate normal vector $(U_0, U_1, \ldots,  U_n)^\top $ in (B). In particular $a=k_1=0$, $k_2=1$ leads to $X\stackrel{\textrm{Law}}{=}d_1P_1+ c_1(P_1^2-1)$ and the corresponding normal vector $(U_0, U_1)^\top$ is degenerate. It can be easily deduced from the proof of Lemma \ref{230516c}  that $(U_0, U_1, \ldots, U_n)^\top$ is non-degenerate if $a\ne 0 $ or $k_2=0$. In order to simplify our calculations, we shall consider in this paper target variables of class (A).
\end{remark}

\color{black}

We shall need the characteristic function and the cumulants of $X$, defined in  Eq.~\eqref{210416a} and \eqref{070416aaa}.

\begin{lemma}\label{030416d}
Consider $k_1, k_2\ge 0$ and 
\[X=I_1(f_1)+I_2(f_2)=a N + \sum_{i=1}^{k_1} b_i(R_i^2-1)+\sum_{i=1}^{k_2} [c_i(P_i^2-1)+d_iP_i],\] as defined in 
Eq.~\eqref{210416a} and \eqref{070416aaa}  where    $N, R_1, \ldots, R_{k_1}, P_1, \ldots, P_{k_2} \stackrel{\textrm{i.i.d.}}{\sim}\mathcal{N}(0, 1)$.
We have  with $\Delta_l:=4c_l^2+d_l^2$ for the characteristic function $\varphi_X$ of $X$: 
\begin{align*}
\varphi_X(x)  &= \exp\left( -\frac{a^2x^2}{2}-ix \sum_{j=1}^{k_!}b_j +\sum_{j=1}^{k_2}\frac{x^2\Delta_j+2ic_j x}{4ixc_j-2} \right)\\
&\rule{5mm}{0mm} \times \prod _{j=1}^{k_1} (1-2ixb_j)^{-1/2} \, \prod _{j=1}^{k_2} (1-2ixc_j)^{-1/2},
\end{align*}
where, as usual, an empty product equals 1 and an empty sum equals 0.
Moreover $\varphi_X$ is the unique solution of the initial value problem $y(0)=1$ and:
\begin{align}
&y'(x)\prod_{j=1}^{k_1}(1-2ixb_j)\, \prod_{j=1}^{k_2} (1-2ixc_j)^2 \nonumber \\
&= y(x)\left( -xa^2 - i\sum_{j=0}^{k_1}b_j \right)\prod_{j=1}^{k_1}(1-2ixb_j)\, \prod_{j=1}^{k_2} (1-2ixc_j)^2\nonumber \\
&\rule{5mm}{0mm}+ y(x) \prod_{j=1}^{k_1} (1-2ixb_j)  \sum_{l=1}^{k_2} \left(\prod_{j\ne l} (1-2ixc_j)^2 \right) \left( 2xc_l^2-x\Delta_l(1-2ixc_l) -x^2 i c_l\Delta_l  \right) \nonumber \\
&\rule{5mm}{0mm} + y(x)\prod_{j=1}^{k_2} (1-2ixc_j)^2\sum_{l=1}^{k_1} ib_l \prod_{j\ne l} (1-2ixb_j).\label{010416b}
\end{align}  
\end{lemma}

\begin{proof}
Using the characteristic function of the non-central $\chi^2$ distribution and the representation:
\[
c_j(P_j^2-1)+d_jP_j=c_j\left(P_j+\frac{d_j}{2c_j}\right)^2  - \frac{d_j^2+4c_j^2}{4c_j},
\]
we find for the characteristic functions:
\begin{align*}
\varphi_{b_j(R_j^2-1)}(x)&= (1-2ixb_j)^{-1/2} \, \exp\left( 
-ix b_j
\right),
\\
\varphi_{
c_j(P_j^2-1)+d_jP_j}(x)&= (1-2ixc_j)^{-1/2} \, \exp\left( 
\frac{x^2 \Delta_j + 2ic_j x}{4ixc_j-2}
\right).
\end{align*}
Hence, with the independence of the standard normal random variables:
\begin{align*}
\varphi_X(x)&=\varphi_{aN}(x)\, \prod_{j=1}^{k_1} \varphi_{b_j( R_j^2-1)}(x)\,
\prod_{j=1}^{k_2} \varphi_{c_j(P_j^2-1)+d_jP_j} (x)\\
&= \exp\left( -  \frac{a^2 x^2}{2} \right) \, 
\prod _{j=1}^{k_1} (1-2ixb_j)^{-1/2} \exp\left( -ix\sum_{j=1}^{k_1}b_j \right)
\,\prod _{j=1}^{k_2} (1-2ixc_j)^{-1/2} \\
&\rule{5mm}{0mm} \times \exp\left( \sum_{j=1}^{k_2} \frac{x^2 \Delta_j + 2ic_jx}{4ixc_j-2} \right)\\
&= \exp\left( -\frac{a^2x^2}{2}-ix \sum_{j=1}^{k_1} b_j +\sum_{j=1}^{k_2}\frac{x^2\Delta_j+2ic_j x}{4ixc_j-2} \right)\\
&\rule{5mm}{0mm} \times \prod _{j=1}^{k_1} (1-2ixb_j)^{-1/2} \, \prod _{j=1}^{k_2} (1-2ixc_j)^{-1/2}.
\end{align*}
Notice that
$
\frac{d}{dx} \frac{1}{\sqrt{1-2ix\alpha}} = \frac{1}{\sqrt{1-2ix\alpha}}\, \frac{i\alpha}{1-2ix\alpha},
$
for every real constant $\alpha$, thus $\varphi'_X(x)$ equals: 
\begin{align*}
%&\varphi_X'(x)=
&\varphi_X(x ) \left( -xa^2-i\sum_{j=1}^{k_1}b_j +\sum_{j=1}^{k_2}\frac{2x\Delta_j+2ic_j}{4ixc_j-2} -\sum_{j=1}^{k_2} \frac{(x^2\Delta_j+2ic_jx)\, 4ic_j}{(4ixc_j-2)^2} \right. \\
&\rule{5mm}{0mm} \left.+\sum_{j=1}^{k_1} \frac{ib_j}{1-2ib_j} + \sum_{j=1}^{k_2}\frac{ic_j}{1-2ixc_j}\right) \\
%&\rule{5mm}{0mm}+ \varphi_X(x)\, \left( \sum_{j=1}^{k_1} \frac{ib_j}{1-2ib_j} + \sum_{j=1}%^{k_2}\frac{ic_j}{1-2ixc_j} \right)\\
&= \varphi_X(x)\left( -xa^2 - i\sum_{j=1}^{k_1}b_j  - x\sum_{j=1}^{k_2} \frac{\Delta_j}{1-2ixc_j} - x^2 \sum_{j=1}^{k_2} \frac{ic_j\Delta_j}{(1-2ixc_j)^2}\right. \\
&\rule{5mm}{0mm} \left. + 2x \sum_{j=1}^{k_2} \frac{c_j^2}{(1-2ixc_j)^2} +\sum_{j=1}^{k_1}\frac{ib_j}{1-2ixb_j} \right).
\end{align*}
Multiplying by $\prod_{j=1}^{k_1}(1-2ixb_j)\, \prod_{j=1}^{k_2} (1-2ixc_j)^2\ne 0$ yields:
\begin{align*}
&\varphi_X'(x)\prod_{j=1}^{k_1}(1-2ixb_j)\, \prod_{j=1}^{k_2} (1-2ixc_j)^2 \\
&= \varphi_X(x)\left( -xa^2 - i\sum_{j=0}^{k_1}b_j \right)\prod_{j=1}^{k_1}(1-2ixb_j)\, \prod_{j=1}^{k_2} (1-2ixc_j)^2\\
&\rule{5mm}{0mm}+ \varphi_X(x) \prod_{j=1}^{k_1} (1-2ixb_j)  \sum_{l=1}^{k_2} \left(\prod_{j\ne l} (1-2ixc_j)^2 \right)\left(2 x c_l^2 -x\Delta_l(1-2ixc_l) -x^2 i c_l\Delta_l  \right) \\
&\rule{5mm}{0mm} + \varphi_X(x)\prod_{j=1}^{k_2} (1-2ix c_j)^2 \sum_{l=1}^{k_1} ib_l \prod_{j\ne l} (1-2ixb_j).
\end{align*} 
Thus $\varphi_X$ is  a solution of the initial value problem.   { The uniqueness of the solution follows  with the Cauchy-Lipschitz theorem since 
\begin{align*}
&x\mapsto -xa^2 - i  \sum_{j=1}^{k_1}b_j  - x\sum_{j=1}^{k_2} \frac{\Delta_j}{1-2ixc_j} - x^2 \sum_{j=1}^{k_2} \frac{ic_j\Delta_j}{(1-2ixc_j)^2} \\
&\rule{5mm}{0mm}+ 2x \sum_{j=1}^{k_2} \frac{c_j^2}{(1-2ixc_j)^2}  + \sum_{j=1}^{k_1}\frac{ib_j}{1-2ixb_j}
\end{align*}
is continuous and bounded on every (real) interval.}% It is enough to apply the Picard-Lindel{\"o}f theorem.
\end{proof}

\begin{remark}
Notice that it may be possible to simplify the differential equation if not all coefficients are pairwise different. In Theorem \ref{230416a}, \ref{030416aa} and \ref{030416aab}, this simplification may yield a polynomial of smaller degree. For a special case, this problem is discussed in Theorem \ref{240416c} and Remark \ref{141018f}. For the rest of this section we shall allow that not all coefficients are pairwise different. In the case of pairwise different coefficients, the differential equation cannot be simplified and  the same observation holds for the polynomials in the previously cited theorems.
\end{remark}

\begin{lemma}\label{140616g}
Consider $k_1, k_2\ge 0$ and 
\[X=I_1(f_1)+I_2(f_2)=a N + \sum_{i=1}^{k_1} b_i(R_i^2-1)+\sum_{i=1}^{k_2} [c_i(P_i^2-1)+d_iP_i],\] as defined in Eq.~\eqref{210416a} and \eqref{070416aaa} where   $N, R_1, \ldots, R_{k_1}, P_1,\ldots, P_{k_2} \stackrel{\textrm{i.i.d.}}{\sim}\mathcal{N}(0, 1)$. 
Then, for $r\ge 2$:
\begin{align*}
\kappa_r(X)= 
a^2 1_{[r=2]} +
\sum_{j=1}^{k_1} 2^{r-1}(r-1)! b_j^r
+
\sum_{j=1}^{k_2} [2^{r-1}(r-1)! c_j^r + 2^{r-3}r! c_j^{r-2}d_j^2].
\end{align*}
\end{lemma}
\begin{proof} 
We have $
\kappa_2(N)=a^2$, $\kappa_r(N)=0$ for $r>2$ and $\kappa_r(R_j^2-1)= 2^{r-1}(r-1)!$. We notice: \[
c_j(P_j^2-1)+d_jP_j= c_j\left( P_j+ \frac{d_j}{2c_j} \right)^2 -\frac{d_j^2}{4c_j}-c_j,
\]
hence for $r\ge 2$:
\[
\kappa_r(c_j(P_j^2-1)+d_jP_j)=c_j^r \kappa_r\left[ \left(P_j+\frac{d_j}{2c_j}\right)^2\right].
\]
Using the formula for the cumulants of the non-central $\chi^2$ distribution, we have for $S\sim \mathcal{N}(\mu, 1)$:
\[ \kappa_r(S^2)=2^{r-1} (r-1)! (1+r\mu^2),
\]
thus:
\[
\kappa_r[c_j(P_j^2-1)+d_jP_j]= 2^{r-1} (r-1)! c_j^r \left[ 1+ \left(\frac{d_j}{2c_j}\right)^2 r \right]= 2^{r-1}(r-1)! c_j^r + 2^{r-3}r! c_j^{r-2}d_j^2. 
\]
The result follows now with the independence of  the random variables.
\end{proof}

\noindent We can now prove the first part of our main  result:  a sufficient criterion for the convergence in law to $X$.

\begin{theorem}\label{030416aa}

Consider $k_1, k_2\ge 0$ and 
\[X=I_1(f_1)+I_2(f_2)=a N + \sum_{i=1}^{k_1} b_i(R_i^2-1)+\sum_{i=1}^{k_2} [c_i(P_i^2-1)+d_iP_i],\]
as defined in Eq.~\eqref{210416a} and \eqref{070416aaa} 
 where  $N, R_1, \ldots, R_{k_1}, P_1,\ldots, P_{k_2} \stackrel{\textrm{i.i.d.}}{\sim}\mathcal{N}(0, 1)$. Suppose that at least one of the parameters $a, k_1, k_2$ is non-zero.
Consider a sequence $\left\lbrace F_n\right\rbrace_n$ of non-zero random variables with $F_n=\sum_{i=1}^p I_i(f_{n, i})$ for $p\ge 2$ fixed and $\left\lbrace f_{n, i}\right\rbrace_n\subset H^{\odot i}$ for $1\le i \le p$. Define: \[P(x)=x^{1+1_{[a\ne 0]}} \prod_{j=1}^{k_1} (x-b_j)\prod_{j=1}^{k_2}(x-c_j)^2.\]
If the following conditions hold, as $n\to \infty$:
\begin{enumerate}[(1)]
\item $\displaystyle{ \kappa_r(F_n)\to \kappa_r(X),\quad\textrm{ for } r=1,\ldots,\textrm{deg}(P)}$,
\item $\displaystyle{ % \E\left[ \left| \E%\left[  \sum_{r=1}^{\textrm{deg}(P)} %\frac{P^{(r)}(0)}{r!2^{r-1}} M_{r-1} (F_n)%\Big| F_n  \right] \right|  \right]=
\E\left[ \left| \E\left[  \sum_{r=1}^{\textrm{deg}(P)} \frac{P^{(r)}(0)}{r!2^{r-1}} \left( \Gamma_{r-1}(F_n) - \E[\Gamma_{r-1}(F_n)] \right) \Big| F_n  \right] \right|  \right] \to 0 }$,
\end{enumerate}
then $F_n \stackrel{\textrm{Law}}{\to} X$, as $n\to \infty$.
\end{theorem}

\begin{proof} 
Notice that an empty product equals 1 and an empty sum equals 0. We prove the result for $a\ne 0$, the other case can be treated similarly.
We shall use and extend an idea of Nourdin and Peccati, see \cite[Paragraph 3.5]{nourdin2009b}. Since $\sup_n\kappa_2(F_n)<\infty$, we have with Chebychev's inequality that the  sequence $\left\lbrace F_n\right\rbrace_n$ is tight. We shall use the following corollary of Prokhorov's Theorem:
\textit{Consider a tight sequence $\left\lbrace F_n\right\rbrace_n$  of random variables. If every subsequence $\left\lbrace F_{n_k}\right\rbrace_k$ which converges in law has the same limit $Y$, then the initial sequence $\left\lbrace F_n\right\rbrace_n$ converges in law to $Y$.} We consider thus a subsequence  $\left\lbrace F_{n_k}\right\rbrace_k$ which converges in law to some random variable $Y$. 
We notice that
$\lim_k\kappa_2(F_{n_k})=\lim_n \kappa_2(F_n)=\kappa_2(X)$, hence $\sup_k \E[F_{n_k}^2]<\infty$.  Since $\left\lbrace F_n\right\rbrace_n$ lives in a fixed finite sum of Wiener chaoses, the hypercontractivity property  implies $\sup_k \E[|F_{n_k}|^r]<\infty$ for every $r\ge 2$. With  $F_{n_k}\stackrel{\textrm{Law}}{\to}Y$, as $k\to \infty$, we have $\E[Y^2]=\lim_k\E[F_{n_k}^2]=\lim_n\E[F_n^2]=\kappa_2(X)\ne 0$. Hence $Y$ is a non-zero random variable and we have $\lim_k \kappa_r(F_{n_k})=\kappa_r(Y)$ for every $r$. On the other hand we have $\lim_n\kappa_r (F_n)=\kappa_r(X) $  for $r=1,\ldots,\textrm{deg}(P)$, thus:
\[\kappa_r(X)=\kappa_r(Y),\quad \textrm{for }  r=1, \ldots,    \textrm{deg}(P).\] 
To simplify the notations and to avoid complicated indices we shall write from now on $\left\lbrace F_n\right\rbrace_n$ for the subsequence.  By the previous corollary, the proof is complete if we can prove that $\varphi_{F_n}$ converges to  $\varphi_X$. We shall prove this by showing that $\varphi_Y=\lim_n \varphi_{F_n}$ solves the initial value problem of Lemma \ref{030416d}. This implies that $\varphi_{Y}= \varphi_X$ and hence $Y\stackrel{\textrm{Law}}{=}X$. The proof is divided in 5 steps:
\begin{itemize}
\item Step 1: We show that $\varphi'_Y(0)=0$ and find an alternative representation for $\E[\exp(ixF_n)\Gamma_r(F_n)]$.
\item Step 2: We calculate: 
\begin{align}
\E\left[ \sum_{r=1}^{\textrm{deg}(P)} \frac{P^{(r)}(0)}{r!2^{r-1}}\exp(ixF_n)\Gamma_{r-1}(F_n) \right],\label{241116c}
\end{align}
and: 
\begin{align}
 \sum_{r=1}^{\textrm{deg}(P)} \frac{P^{(r)}(0)}{r!2^{r-1}}\E[\exp(ixF_n)]\E[\Gamma_{r-1}(F_n)].\label{241116d}
\end{align}
\item Step 3: We calculate $\varphi_Y'(x)\prod_{j=1}^{k_1} (1-2ixb_j)\prod_{j=1}^{k_2} (1-2ixc_j)^2$.
\item Step 4: We find an expression for: 
\begin{align*}
&\varphi_Y(x)\left( -xa^2 - i \sum_{j=0}^{k_1}b_j  \right)\prod_{j=1}^{k_1}(1-2ixb_j)\, \prod_{j=1}^{k_2} (1-2ixc_j)^2\\
&\rule{5mm}{0mm}+ \varphi_Y(x) \prod_{j=1}^{k_1} (1-2ixb_j)  \sum_{l=1}^{k_2} \prod_{j\ne l} (1-2ixc_j)^2 \left(2 x c_l^2 -x\Delta_l(1-2ixc_l) -x^2 i c_l\Delta_l   \right) \\
&\rule{5mm}{0mm} + \varphi_Y(x)\prod_{j=1}^{k_2}(1-2ixc_j)^2 \sum_{l=1}^{k_1} ib_l \prod_{j\ne l} (1-2ixb_j).
\end{align*} 
\item Step 5: The proof is completed by showing that the expressions found in the last two steps are equal. Then $\varphi_Y$ is the unique solution to the initial value problem in Lemma \ref{030416d}, $Y$ and $X$ are thus equal in distribution and $F_n\stackrel{\textrm{Law}}{\to}X$, as $n\to \infty$.
\end{itemize}
For the ease of notation, we define
$
k:=2k_2+k_1~,~ G_1(x):=\prod_{j=1}^{k_1}(1-2ixb_j)$ and $G_2(x):=\prod_{j=1}^{k_2}(1-2ixc_j)^2.
$

\textit{Step 1.} 
The random variable $Y$ is non-zero  and  has moments of every order.
We notice that for $x=0$, the differential equation of Lemma \ref{030416d} holds for $\varphi_Y$ if we can prove $\varphi_Y'(0)=0$. We notice that $\E[|F_n|]<\infty$, hence $\varphi_{F_n}'(0)=i\E[F_n]=0$. On the other hand $\lim_n \varphi_{F_n}'(x)=\varphi_Y'(x)$ for every $x\in \mathbb{R}$ since $\sup_n \E[F_n^2]<\infty$ and $iF_n\exp(ixF_n)\stackrel{\textrm{Law}}{\to} iY\exp(ixY)$ by the continuous mapping theorem. We have thus $\varphi_Y'(0)=\lim_n \varphi_{F_n}'(0)=0$. We suppose now that $x\ne 0$ and  calculate $\E[\exp(ixF_n) \Gamma_r F_n]$ for $r\in \mathbb{N}$.
For 
$r\in \left\lbrace0, 1\right\rbrace$:\begin{align*}
&\E[\exp(ixF_n) \Gamma_0 (F_n)]= -i\E[\exp(ixF_n)iF_n] = -i\varphi_{F_n}'(x),\\
&\E[\exp(ixF_n)\Gamma_1(F_n)]= \E[\exp(ixF_n)\langle DF_n,-DL^{-1}F_n\rangle_H]\\
&= \frac{1}{ix}\E[\langle D\exp(ixF_n),-DL^{-1}F_n\rangle_H]= \frac{1}{ix} \E[\exp(ixF_n)(-\delta DL^{-1} F_n)] \\
&= \frac{1}{ix}\E[\exp(ixF_n)F_n]= -\frac{i}{ix} \varphi_{F_n}'(x),
\end{align*}
and for $r>1$:
\begin{align*}
&\E[\exp(ixF_n)\Gamma_r{(F_n)}]= \E[\exp(ixF_n)\langle DF_n,-DL^{-1}\Gamma_{r-1}(F_n)\rangle_H]\\
&= \frac{1}{ix}\E[\langle D\exp(ixF_n), -DL^{-1}\Gamma_{r-1}(F_n)\rangle_H]= \frac{1}{ix}\E[ \exp(ixF_n) [-\delta D L^{-1} \Gamma_{r-1} (F_n)] ]\\
&=\frac{1}{ix}\E[\exp(ixF_n) [\Gamma_{r-1}(F_n) - \E[\Gamma_{r-1}(F_n)]] ]\\
&= \frac{1}{ix} \E[\exp(ixF_n) \Gamma_{r-1}(F_n)]- \frac{1}{ix}\E[\exp(ixF_n)] \E[\Gamma_{r-1}(F_n)]. 
\end{align*}
Iteration yields:
\begin{align}
\E[\exp(ixF_n)\Gamma_r(F_n)]&= -{i}{(ix)^{-r}} \varphi_{F_n}'(x)-\varphi_{F_n}(x)\sum_{j=1}^{r-1} (ix)^{-j}\E[\Gamma_{r-j}(F_n)] .\label{030416f}
\end{align}
\textit{Step 2.} We calculate now the sum in \eqref{241116c}. We have:
\begin{align*}
 &-i\varphi_{F_n}'(x) \sum_{r=1}^{k+2} \frac{P^{(r)}(0)}{r!2^{r-1}} (ix)^{-r+1}= -i\varphi_{F_n}'(x) (2ix) \sum_{r=1}^{k+2} \frac{P^{(r)}(0)}{r!} (2ix)^{-r}\\
&= -i\varphi_{F_n}'(x)(2ix) P\left( \frac{1}{2ix} \right)\\
&= -i\varphi_{F_n}'(x) (2ix) \left( \frac{1}{2ix} \right)^{2} \prod_{j=1}^{k_1} \left( \frac{1}{2ix}-b_j \right) \prod_{j=1}^{k_2} \left( \frac{1}{2ix}-c_j \right)^2\\ 
&= -i\varphi_{F_n}'(x)(2ix)^{-1-k}G_1(x)G_2(x).
\end{align*}
With Eq.~\eqref{030416f} we have thus:
\begin{align}
&\E\left[\sum_{r=1}^{k+2}\frac{P^{(r)}(0)}{r!2^{r-1}} \exp(ixF_n)\Gamma_{r-1}(F_n)\right] = -i\varphi_{F_n}'(x) \sum_{r=1}^{k+2} \frac{P^{(r)}(0)}{r!2^{r-1}} (ix)^{-r+1}+ \varphi_{F_n}(x) \nonumber \\
&\rule{5mm}{0mm}\times\E\left[\sum_{r=3}^{k+2}\frac{P^{(r)}(0)}{r!2^{r-1}}  
\left( \frac{- \E[\Gamma_{r-2}(F_n)]}{(ix)^{1}} - \frac{ \E[\Gamma_{r-3}(F_n)]}{(ix)^{2}} -\ldots - \frac{\E[\Gamma_1(F_n)]}{(ix)^{(r-2)}} \right)
\right]\nonumber \\
&= - i\varphi_{F_n}'(x) \, (2ix)^{-k-1} G_1(x)G_2(x)\nonumber  \\
&\rule{5mm}{0mm}-\frac{ \varphi_{F_n}(x)}{(ix)^{1}}\left[ \frac{P^{(3)}(0)}{3!2^2} \E[\Gamma_1(F_n)]+\frac{P^{(4)}(0)}{4!2^3} \E[\Gamma_2(F_n)] +\ldots \right.\nonumber\\
&\rule{5mm}{0mm} \left.+ \frac{P^{(k+2)}(0)}{(k+2)!2^{k+1}} \E[\Gamma_{k}(F_n)]\right]\nonumber \\
&\rule{5mm}{0mm}- \frac{\varphi_{F_n}(x)}{(ix)^{2}}\left[ \frac{P^{(4)}(0)}{4!2^3} \E[\Gamma_1(F_n)]+\frac{P^{(5)}(0)}{5!2^4} \E[\Gamma_2(F_n)] +\ldots\right.\nonumber \\
&\rule{5mm}{0mm} \left. + \frac{P^{(k+2)}(0)}{(k+2)!2^{k+1}} \E[\Gamma_{k-1}(F_n)]\right]-\ldots - \frac{ \varphi_{F_n}(x)}{(ix)^{k}} \, \frac{P^{(k+2)}(0)}{(k+2)!2^{k+1}} \E[\Gamma_1(F_n)].  \label{040416a}
\end{align}
We calculate the sum in \eqref{241116d} using 
 $\E[\Gamma_r(F_n)]=\kappa_{r+1}(F_n)/r!$ and $\E[\Gamma_0(F_n)]=\E[F_n]=0$:
\begin{align}
\sum_{r=1}^{k+2} \frac{P^{(r)}(0)}{r!2^{r-1}}\E[\exp(ixF_n)]\E[\Gamma_{r-1}(F_n)]&= %\sum_{r=2}^{k+2} \frac{P^{(r)}(0)}{r!2^{r-1}}\E[\exp(ixF_n)]\E[\Gamma_{r-1}(F_n)]\nonumber\\
 \sum_{r=2}^{k+2} \frac{P^{(r)}(0)}{r!2^{r-1}}\frac{\kappa_r(F_n)}{(r-1)!}\varphi_{F_n}(x)
 .\label{030416c}
\end{align}
\textit{Step 3.} 
We have thus for the limit of the  expression on the right-hand side of Eq.~\eqref{030416c}, as $n\to \infty$:
\begin{align}
&\varphi_Y(x) \sum_{r=2}^{k+2} \frac{P^{(r)}(0)}{r!2^{r-1}}
\frac{\kappa_r(X)}{(r-1)!}=  \varphi_Y(x)\frac{P''(0)}{2! 2}\,a^2\nonumber + \varphi_Y(x)\\
 &\rule{5mm}{0mm} \times \sum_{r=2}^{k+2} \frac{P^{(r)}(0)}{r!2^{r-1}} \left[ \sum_{j=1}^{k_1} \frac{2^{r-1}(r-1)! b_j^r}{(r-1)!}+
 \sum_{j=1}^{k_2}\left( \frac{2^{r-1}(r-1)! c_j^{r}+ 2^{r-3}r! c_j^{r-2}d_j^2}{(r-1)!} \right)\right] \nonumber\\
 &= \varphi_Y(x)\left[\sum_{j=1}^{k_1}\sum_{r=1}^{k+2} \frac{P^{(r)}(0)}{r!} b_j^r +   \sum_{j=1}^{k_2}  \sum_{r=1}^{k+2} \frac{P^{(r)}(0)}{r!}c_j^r + \sum_{j=1}^{k_2}
 \frac{d_j^2}{4c_j} \sum_{r=1}^{k+2} \frac{P^{(r)}(0)}{r!}r \, c_j^{r-1} \right.\nonumber \\
 &\rule{5mm}{0mm} \left.+ \frac{(-1)^{k_1}a^2}{2}   \prod_{j=1}^{k_2}c_j^2 \prod_{j=1}^{k_1} b_j\right]\nonumber \\
&= \varphi_Y(x) \left[ \sum_{j=1}^{k_1}P(b_j) +  \sum_{j=1}^{k_2} P(c_j) +   \sum_{j=1}^{k_2} c_j^{-1}d_j^2 P'(c_j)/4 + \frac{1}{2}  (-1)^{k_1} a^2 \prod_{j=1}^{k_2}c_j^2 \prod_{j=1}^{k_1} b_j \right]\nonumber\\
&= \frac{1}{2}\varphi_Y(x) (-1)^{k_1} a^2 \prod_{j=1}^{k_2}c_j^2 \prod_{j=1}^{k_1} b_j .\label{040416b}
  \end{align} 
We have with $M_{r-1}:= \Gamma_{r-1}(F_n)- \E[\Gamma_{r-1}(F_n)]$, as $n\to \infty$:
\begin{align*}
\left|\E\left[ \exp(ixF_n)\sum_{r=1}^{k+2}\frac{P^{(r)}(0)}{r!2^{r-1}} M_{r-1}(F_n) \right]\right| %&=\left|\E\left[ \exp(ixF_n)\E\left[ \sum_{r=1}^{k+2}\frac{P^{(r)}(0)}{r!2^{r-1}} M_{r-1}(F_n) \Big| F_n \right] \right]\right|\\
&\le\E\left[ \left|\E\left[ \sum_{r=1}^{k+2}\frac{P^{(r)}(0)}{r!2^{r-1}} M_{r-1}(F_n) \Big| F_n \right]\right| \right]\to 0,
\end{align*}
hence:
\begin{align*}
&\lim_n \E\left[ \exp(ixF_n)\sum_{r=1}^{k+2}\frac{P^{(r)}(0)}{r!2^{r-1}} \Gamma_{r-1}(F_n) \right]\\
 &= \lim_n \E\left[ \exp(ixF_n)\sum_{r=1}^{k+2}\frac{P^{(r)}(0)}{r!2^{r-1}} M_{r-1}(F_n) \right] +  \lim_n \varphi_{F_n}(x) \sum_{r=1}^{k+2}\frac{P^{(r)}(0)}{r!2^{r-1}} \E[\Gamma_{r-1}(F_n) ] \\
 &= 0+  \lim_n \varphi_{F_n}(x) \sum_{r=1}^{k+2}\frac{P^{(r)}(0)}{r!2^{r-1}} \E[\Gamma_{r-1}(F_n) ] ,
\end{align*}
using Eq.~\eqref{040416a} and Eq.~\eqref{040416b}, we have with $\lim_n \E[\Gamma_r(F_n)]=\kappa_{r+1}(X)/r!$ for $r=1,\ldots,\textrm{deg}(P)-1$:
\begin{align}
&-i\varphi_{Y}'(x) (2ix)^{-k-1} G_1(x) G_2(x)
-\sum_{m=1}^{k}  \varphi_Y(x) (ix)^{-m} \sum_{r=m}^{k} \frac{P^{(r+2)}(0)}{(r+2)!2^{r+1}} \frac{\kappa_{r-m+2}(X)}{(r-m+1)!}\nonumber\\
&= \frac{1}{2} \varphi_Y(x) (-1)^{k_1} a^2 \prod_{j=1}^{k_2}c_j^2 \prod_{j=1}^{k_1} b_j .\label{040416c}
\end{align}
We have used that $\lim_n \varphi_{F_n}'(x)=\lim_n \varphi_Y'(x)$ pointwise. This can be seen using the continuous mapping theorem and $\sup_n \E[F_n^2]<\infty$. Multiplying Eq.~\eqref{040416c} by $i (2ix)^{k+1}$ yields:
\begin{align}
\varphi_Y'(x) G_1(x)G_2(x) &= \sum_{m=1}^{k}  (2ix)^{k+1} i\varphi_Y(x) (ix)^{-m} \sum_{r=m}^{k} \frac{P^{(r+2)}(0)}{(r+2)!2^{r+1}} \frac{\kappa_{r-m+2}(X)}{(r-m+1)!}\nonumber\\
&\rule{5mm}{0mm}+ i\varphi_Y(x) (2ix)^{k+1} \frac{1}{2} (-1)^{k_1} a^2 \prod_{j=1}^{k_2}c_j^2 \prod_{j=1}^{k_1} b_j .\label{170816a}
\end{align}
\textit{Step 4.}
We have: 
\begin{align}
&\lim_n \varphi_{F_n}(x)\Bigl[\Bigl( -xa^2 - i\sum_{j=1}^{k_1}b_j \Bigr)\prod_{j=1}^{k_1}(1-2ixb_j)\, \prod_{j=1}^{k_2} (1-2ixc_j)^2\Bigr. \nonumber \\
&\rule{5mm}{0mm}+ \Bigl.\prod_{j=1}^{k_2} (1-2ixc_j)^2\sum_{l=1}^{k_1} ib_l \prod_{j\ne l} (1-2ixb_j)\Bigr.\nonumber\\
&\rule{5mm}{0mm}+ \Bigl.   \prod_{j=1}^{k_1} (1-2ixb_j)  \sum_{l=1}^{k_2}  \left(\prod_{j\ne l} (1-2ixc_j)^2 \right)\left[2 x c_l^2 -x\Delta_l(1-2ixc_l) -x^2 i c_l\Delta_l \right] \Bigr] \nonumber\\
&=i\varphi_Y(x) \left(ixa^2-\sum_{j=1}^{k_1}b_j\right) G_1(x)G_2(x)+ i\varphi_Y(x) G_2(x) \sum_{l=1}^{k_1} b_l \prod_{j\ne l} (1-2ixb_j)\nonumber\\
&\rule{5mm}{0mm}+ i\varphi_Y(x) G_1(x) \sum_{l=1}^{k_2} \left(\prod_{j\ne l} (1-2ixc_j)^2 \right) \left[ ix \Delta_l (1-2ixc_l) + (ix)^2 c_l\Delta_l  - 2 ix c_l^2 \right] \nonumber\\
&= i\varphi_Y(x) ix a^2 G_1(x)G_2(x) +i\varphi_Y(x)G_2(x) \sum_{l=1}^{k_1} [-b_l(1-2ixb_l)+b_l] \prod_{j\ne l } (1-2ixb_j)\nonumber\\
&\rule{5mm}{0mm}+ i\varphi_Y(x) G_1(x) \sum_{l=1}^{k_2} \left(\prod_{j\ne l} (1-2ixc_j)^2 \right) \left[ ix \Delta_l (1-2ixc_l) + (ix)^2 c_l\Delta_l  - 2 ix c_l^2 \right] \nonumber\\
&= i\varphi_Y(x) ix a^2 G_1(x)G_2(x) + i\varphi_Y(x) G_2(x) 2ix \sum_{l=1}^k b_l^2 \prod_{j\ne l}(1-2ixb_j)+ i\varphi_Y(x) G_1(x)\nonumber\\
&\rule{5mm}{0mm}\times  \sum_{l=1}^{k_2} \left(\prod_{j\ne l} (1-2ixc_j)^2 \right) \left[ ix \Delta_l (1-2ixc_l) + (ix)^2 c_l\Delta_l  - 2ix c_l^2 \right]. \label{020416a}
\end{align}
\textit{Step 5.}
We have that $Y\stackrel{\textrm{Law}}{=}X$ if and only if the right-hand side of Eq.~\eqref{020416a} equals
$\varphi_Y'(x) G_1(x)G_2(x) $ or, using the previous results and Eq.~\eqref{170816a} in particular, if the following equality holds:
 \begin{align}
&\sum_{l=1}^{k} (2ix)^{k+1} i\varphi_Y(x) (ix)^{-l} \sum_{r=l}^{k} \frac{P^{(r+2)}(0)}{(r+2)!2^{r+1}}  \frac{\kappa_{r-l+2}(X)}{(r-l+1)!} \nonumber\\
&\rule{5mm}{0mm}+i\varphi_Y(x)(2ix)^{k+1}\frac{ (-1)^{k_1}a^2}{2}\prod_{j=1}^{k_2} c_j^2\, \prod_{j=1}^{k_1} b_j \nonumber \\
&= i\varphi_Y(x)ix a^2 G_1(x)G_2(x) + i\varphi_Y(x) 2ix G_2(x) \sum_{l=1}^{k_1} b_l^2 \prod_{j\ne l}(1-2ixb_j)\nonumber\\
&\rule{5mm}{0mm}+ i\varphi_Y(x) G_1(x) \sum_{l=1}^{k_2} \left(\prod_{j\ne l} (1-2ixc_j)^2 \right) \left[ ix \Delta_l (1-2ixc_l) + (ix)^2 c_l\Delta_l  - 2ix c_l^2 \right]. \label{040416dd}
 \end{align}
If $\varphi_Y(x)\ne 0$, we can divide by $i\varphi_Y(x)$ and compare the coefficients of $x, x^2,  x^{3}, \ldots$ on the left- and right-hand side of Eq.~\eqref{040416dd}. For this final part of the proof, see Appendix.  
 Considering the previous remarks, this concludes the proof.
\end{proof}

\begin{remark} \label{230416b}
\begin{enumerate}[(1)]
\item  {Notice that the proof of Eq.~\eqref{040416dd} for the general case is rather lengthy and technical. This is basically due to the differential equation derived in Lemma  \ref{030416d} from which follows  
Eq.~\eqref{040416dd}. For special cases, such as the case considered in \citep{azmoodeh}, the differential equation simplifies considerably and therefore a relatively simple recurrence for the moments of the target variable can be proved. In \cite{azmoodeh}, the authors have used this recurrence to prove their main result. In the general case however, this recurrence is hard to handle, therefore we have chosen to use differential equations rather than recurrence relations. \\
We illustrate now how the proof of  
 the crucial equation can be simplified for the class of target variables considered in \cite{azmoodeh}. Notice that for $a=c_i=d_i=0$, $k=k_1$ and $X\stackrel{Law}{=} \sum_{i=1}^k b_i (N_i^2-1)$ with $N_1, \ldots, N_k\stackrel{\textrm{i.i.d.}}{\sim}\mathcal{N}(0, 1)$,  %Eq.~\eqref{040416dd} becomes with $k=k_1$:
we have to check the following equation:
\begin{align}
&\sum_{l=1}^{k-1} \frac{(2ix)^k}{ (ix)^{l}} \sum_{r=l}^{k-1} \frac{P^{(r+2)}(0)}{(r+2)!2^{r+1}} \frac{\kappa_{r-l+2}(X)}{(r-l+1)!}-(-1)^k (2ix)^k
\sum_{i=1}^k b_i \prod_{j=1}^k b_j\nonumber\\
&=2ix \sum_{m=1}^k b_m^2 \prod_{j\ne m} (1-2ixb_j), \label{281216a}
\end{align} 
where $P(x):=x\prod_{i=1}^k (x-b_i)$. We use the following relation which can be proved by induction over $l$:
\begin{align}
\sum_{r=l}^{k-1} \frac{P^{(r+2)}(0)}{(r+2)! } \sum_{j=1}^k b_j^{r+2-l} = (-1)^{k-l-1} \sum_{j=1}^k b_j^2 T_{k-1-l}^{(j)},
\end{align}
where $1\le l \le k-1$ and:
\[T_m^{(j)}:= \sum_{i_1<\ldots <i_m \atop i_1, \ldots, i_m\ne j} b_{i_1}\times \ldots \times b_{i_m},\quad 1 \le m \le k-2 
\]
and $T_0^{(j)}:=1$, see Definition \ref{040316d} for details. We have thus for the left-hand side of Eq.~\eqref{281216a}:
\begin{align*}
&\sum_{l=1}^{k-1}(2ix)^k (ix)^{-l} \sum_{r=l}^{k-1} \frac{P^{(r+2)}(0)}{(r+2)!2^{r+1}} \frac{\kappa_{r-l+2}(X)}{(r-l+1)!}-(-1)^k (2ix)^k
\sum_{i=1}^k b_i \prod_{j=1}^k b_j \\
&= \sum_{l=1}^{k-1} (2ix)^k (ix)^{-l} \sum_{r=l}^{k-1} \frac{P^{(r+2)}(0)}{(r+2)!2^{r+1}}\sum_{j=1}^k \frac{2^{r-l+1}(r-l+1)!b_j^{r-l+2}}{(r-l+1)!} \\
&\rule{5mm}{0mm}-  (-1)^k (2ix)^k
\sum_{i=1}^k b_i \prod_{j=1}^k b_j \\
&= \sum_{l=1}^{k-1} (2ix)^{k-l} (-1)^{k-l-1} \sum_{j=1}^k b_j^2 T_{k-1-l}^{(j)} - (-1)^k (2ix)^k
\sum_{i=1}^k b_i \prod_{j=1}^k b_j,
\end{align*}
and for the right-hand side of Eq.~\eqref{281216a}:
\begin{align*}
&2ix \sum_{m=1}^k b_m^2 \sum_{j=1}^{k-1} \sum_{i_1 < \ldots < i_j \atop i_1,\ldots, i_j\ne m} (-2ixb_{i_1}) \times \ldots \times (-2ixb_{i_j}) + 2ix \sum_{m=1}^k b_m^2 \\
&= \sum_{j=0}^{k-1} (2ix)^{1+j} (-1)^j \sum_{m=1}^k b_m^2 T_j^{(m)} =\sum_{l=0}^{k-1} (2ix)^{k-l} (-1)^{k-l-1} \sum_{m=1}^k b_m^2 T_{k-l-1}^{(m)} \\ 
&=\sum_{l=1}^{k-1} (2ix)^{k-l} (-1)^{k-l-1} \sum_{m=1}^k b_m^2 T_{k-l-1}^{(m)} + (2ix)^k (-1)^{k-1} \sum_{m=1}^k b_m \, b_mT_{k-1}^{(m)}\\ 
&=\sum_{l=1}^{k-1} (2ix)^{k-l} (-1)^{k-l-1} \sum_{m=1}^k b_m^2 T_{k-l-1}^{(m)} + (2ix)^k (-1)^{k-1} \sum_{m=1}^k b_m  \prod_{j=1}^k b_j.
\end{align*}
This proves that Eq.~\eqref{281216a} holds. \\
In the general case, % of Eq.~\eqref{040416dd},   
 the calculation of $\sum_{r=l}^{\textrm{deg}(P)-2} \frac{P^{(r+2)}(0)}{(r+2)!2^{r+1}} \frac{\kappa_{r-l+2}(X)}{(r-l+1)!}$ is lengthy. The following relations are needed and can be proved by induction over $l$, for the ease of notation define $k':=2k_2+k_1+1_{[a\ne 0]}$:
\begin{align}
&\sum_{r=l}^{k'-1}\frac{P^{(r+2)}(0)}{(r+2)!} b_j^{r+2-l} = (-1)^{k'-l-1} \sum_{j=1}^{k_1}b_j^2 \sum_{i_1+i_2+i_3\atop =k'-1-l} T_{i_1}^{(j)}S_{i_2} S_{i_3}, \nonumber \\
&\sum_{r=l}^{k'-1}\frac{P^{(r+2)}(0)}{(r+2)!}\sum_{j=1}^{k_2}c_j^{r+2-l} = (-1)^{k'-l-1} \nonumber\\
&\rule{5mm}{0mm}\times\sum_{j=1}^{k_2}\left( c_j^2 \sum_{i_1+i_2+i_3\atop = k'-1-l}T_{i_1}S_{i_2}^{(j)}S_{i_3}^{(j)} +c_j^3 \sum_{i_1+i_2+i_3\atop =  k'-2-l}T_{i_1}S_{i_2}^{(j)}S_{i_3}^{(j)}\right), \label{281216b}\end{align}
and:
\begin{align}
&\sum_{r=l}^{k'-1} \frac{P^{(r+2)}(0)}{(r+2)!} \sum_{j=1}^{k_2}(r-l+2)c_j^{r-l}d_j^2 \nonumber\\
&= (-1)^{k'-l-1} \sum_{j=1}^{k_2}\left( 2d_j^2 \sum_{i_1+i_2+i_3 \atop = k'-l-1} T_{i_1} S_{i_2}^{(j)}S_{i_3}^{(j)} + cjd_j^2 \sum_{i_1+i_2+i_3 \atop = k'-l-2} T_{i_1} S_{i_2}^{(j)}S_{i_3}^{(j)}\right)\label{281216c},
\end{align}
where  $S_i$ and $S_i^{(j)}$ are defined similarly to $T_i^{(j)}$ with the coefficients $b_1, \ldots, b_{k_1}$ replaced by $c_1, \ldots, c_{k_2}$, see Definition \ref{040316d} for details.  \\
In particular,  equations \eqref{281216b} and \eqref{281216c} imply that the proof of Eq.~\eqref{040416dd} is technical.  In Proposition \ref{291216d} we present a proof of the latter equation which does not require Eq. \eqref{281216b} and \eqref{281216c}. The coefficients of the polynomials on both sides of Eq.~\eqref{040416dd} are  compared directly. This proof still remains complicated and lengthy. As mentioned above, this is due to the form of the differential equation for the characteristic function of the target variable in the general case. }
\item Theorem \ref{030416aa}   extends \cite[Theorem 3.2,(ii) $\to$ (i)]{azmoodeh} since it holds for a more general set of target random variables $X$ and we have $L^1$-convergence instead of $L^2$-convergence.
\item We notice that for the proof it is essential that, as $n\to \infty$: 
\begin{align*}
& \E\left[  \exp(ixF_n)  \sum_{r=1}^{\textrm{deg}(P)} \frac{P^{(r)}(0)}{r!2^{r-1}}  M_{r-1}(F_n)      \right]\\
&=  \E\left[  \exp(ixF_n) \E\left[ \sum_{r=1}^{\textrm{deg}(P)} \frac{P^{(r)}(0)}{r!2^{r-1}} M_{r-1}(F_n)   \Big|F_n  \right]  \right]  \to 0 .
\end{align*}
Since $|\exp(ixF_n)|=1$, the triangle inequality shows that it is sufficient to have:
\[
\E\left[ \left| \E\left[ \sum_{r=1}^{\textrm{deg}(P)} \frac{P^{(r)}(0)}{r!2^{r-1}} \left( \Gamma_{r-1}(F_n) - \E[\Gamma_{r-1}(F_n)] \right)   \Big|F_n  \right] \right| \right]  \to 0,
\]
as $n\to \infty$. This $L^1$-convergence  is weaker than  $L^2$-convergence which results typically from the   Cauchy-Schwarz inequality, used to control  unbounded factors.

\item Different random variables $X$ may lead to the same polynomial $P$. Let $N$ be  a standard normal variable,  $X_1\stackrel{\textrm{Law}}{=}N^2-1+d_1N$ and $X_2\stackrel{\textrm{Law}}{=}N^2-1+d_2 N$ with $d_1\ne d_2$  non-zero, then Theorem \ref{030416aa}
yields the polynomial $P(x)=x(x-1)^2$. A similar observation holds for 
\cite[Theorem 3.4]{ECP2023} where the constant $\mu_0$ only appears in the limit of the second cumulant.
In our case,    the condition (1) of  Theorem \ref{030416aa}
discerns both cases. For special cases it may be useful to consider a polynomial which is different from the \lq standard polynomial\rq ~defined in Theorem \ref{030416aa}. This observation is founded in the initial value problem established in  Lemma \ref{030416d}. It can be shown that the differential equation in Lemma \ref{030416d} can be simplified if   the $b_i$ or the $(c_j, d_j)$ are not pairwise different. For instance, if $X_1\stackrel{\textrm{Law}}{=}(N_1^2-1)+ (N_2^2-1)$ and $N_1, N_2\stackrel{\textrm{i.i.d}}{\sim}\mathcal{N}(0, 1)$, Lemma \ref{030416d} yields a differential equation which can be simplified to yield \cite[Eq. (1.9)]{nourdin2009b}.

\end{enumerate}  
\end{remark}

We proceed now to the proof of the converse of Theorem \ref{030416aa}. We shall need the following Lemma which generalizes \cite[Eq.~(3.6)]{azmoodeh}.

\begin{lemma}\label{200416a}
Let $P(x):=x^{1+1_{[a\ne 0]}} \prod_{i=1}^{k_1} (x-b_i) \prod_{i=1}^{k_2}(x-c_i)^2$ and $X=I_1(f_1) + I_2(f_2)$, see Eq.~\eqref{210416a}, then:
\begin{align*}
\p\left( \sum_{r=1}^{\textrm{deg}(P)} \frac{P^{(r)}(0)}{r!2^{r-1}} \left(  
\Gamma_{r-1}(X) - \E[\Gamma_{r-1}(X)]
\right)=0\right)=1.
\end{align*}
\end{lemma}
\begin{proof} In this proof all the equalities between random variables hold with probability 1.  {We consider iterated contractions recursively defined as follows for $f_2\in H^{\odot 2}$:
\[
f_2 \tilde{\otimes}_{1}^{(1)}f_2 := f_2 ~;~ f_2 \tilde{\otimes}_{1}^{(p)}f_2 := \left( f_2 \tilde{\otimes}_{1}^{(p-1)}f_2\right) \tilde{\otimes}_{1} f_2,
\]
for $p\ge 2$.} 
We use the representation for $X$ given in Eq.~\eqref{210416a}  {and Eq.~\eqref{231018}}.
\begin{enumerate}[(1)]
\item We prove by induction over $r\ge 1$:
\begin{align}
\Gamma_r(X)-\E[\Gamma_r(X)]&= 2^r I_2(f_2 \tilde{\otimes}_1^{(r+1)} f_2) +2^{r-1}\,3I_1((\ldots ( f_2 \tilde{\otimes}_1 f_1 ) \tilde{\otimes}_1 f_2 ) \ldots )\tilde{\otimes}_1 f_{2})\nonumber \\
&\rule{5mm}{0mm}
+ 2^{r-1}   \sum_{j=3}^{r+1} \sum_{g_j=f_1 \atop g_i=f_2,\textrm{ for }i\ne  j}I_1((\ldots ( g_1 \tilde{\otimes}_1 g_2 ) \tilde{\otimes}_1 g_3 ) \ldots )\tilde{\otimes}_1 g_{r+1}). \label{170416c}
\end{align}
Notice that all iterated contractions on the right hand side of Eq.~\eqref{170416c} run over $r+1$ functions.\\
For $r=1$, we have with the stochastic Fubini theorem and the multiplication formula for multiple Wiener integrals:
\begin{align*}
\Gamma_1(X)&= \langle DX , -DL^{-1} X\rangle_H = \langle 2 I_1(f_2(t,\cdot)) + f_1(t), I_1(f_2(t,\cdot))+f_1(t) \rangle_H\\
&= 2\int_0^T I_2(f_2(t,\cdot)\tilde{\otimes} f_2(t,\cdot)) + f_2(t,\cdot )\tilde{\otimes}_1 f_2(t,\cdot) dt + 3 \int_0^T I_1(f_2(t,\cdot)\tilde{\otimes} f_1(t)) dt  \\
& \rule{5mm}{0mm}+ \int_{0}^T f_1(t)^2dt \\
&=2I_2(f_2 \tilde{\otimes}_1 f_2) + 3 I_1(f_2 \tilde{\otimes}_1 f_1) + 2\Vert  f_2 \Vert_{H\otimes H}^2 + \Vert f_1 \Vert_H^2.
\end{align*}
The claim follows since the expectation of every multiple Wiener integral is 0. Suppose now that the claim holds for some $r\ge 1$, then:
\begin{align*}
&\Gamma_{r+1}(X)= \langle DX,-DL^{-1}\Gamma_r(X)\rangle_H\\
&= \langle 2 I_1(f_2(t,\cdot)) + f_1(t) , 2^r I_1((f_2 \tilde{\otimes}_1^{(r+1)} f_2)(t,\cdot))\\
&\rule{5mm}{0mm} + 2^{r-1} 3 ((\ldots ( f_2 \tilde{\otimes}_1 f_1 ) \tilde{\otimes}_1 f_2 ) \ldots \tilde{\otimes}_1 f_{2})(t) )\\
&\rule{5mm}{0mm} + 2^{r-1}   \sum_{j=3}^{r+1} \sum_{g_j=f_1 \atop g_i=f_2,\textrm{ for }i\ne  j}((\ldots ( g_1 \tilde{\otimes}_1 g_2 ) \tilde{\otimes}_1 g_3 ) \ldots \tilde{\otimes}_1 g_{r+1} )(t)\rangle_H \\
&=2^{r+1} I_2(f_2\tilde{\otimes}_1^{(r+2)}f_2) + 2^{r+1} \langle f_2\tilde{\otimes}_1^{(r+1)} f_2, f_2\rangle_H   \\
&\rule{5mm}{0mm}+ 2^r 3 I_1((\ldots ( f_2 \tilde{\otimes}_1 f_1 ) \tilde{\otimes}_1 f_2 ) \ldots )\tilde{\otimes}_1 f_{2} )\tilde{\otimes}_1 f_2)\\
&\rule{5mm}{0mm} + 2^{r}   \sum_{j=3}^{r+1} \sum_{g_j=f_1 \atop g_i=f_2,\textrm{ for }i\ne  j}I_1((\ldots ( g_1 \tilde{\otimes}_1 g_2 ) \tilde{\otimes}_1 g_3 ) \ldots ) \tilde{\otimes}_1 g_{r+2}) \\
&\rule{05mm}{0mm}+2^r I_1( (f_2 \tilde{\otimes}_1^{(r+1)} f_2) \tilde{\otimes}_1 f_1 )+2^{r-1}
\\
&\rule{5mm}{0mm}\times \biggl(3 \langle (\ldots ( f_2 \tilde{\otimes}_1 f_1 ) \tilde{\otimes}_1 f_2 ) \ldots ) \tilde{\otimes}_1 f_{2}   , f_1 \rangle_H  \biggr.\\
&\rule{5mm}{0mm}+ \biggl.\sum_{j=3}^{r+1} \sum_{g_j=f_1 \atop g_i=f_2,\textrm{ for }i\ne  j}\langle(\ldots ( g_1 \tilde{\otimes}_1 g_2 ) \tilde{\otimes}_1 g_3 ) \ldots ) \tilde{\otimes}_1 g_{r+1} , f_1\rangle_H\biggr).
\end{align*}
Hence Eq.~\eqref{170416c} follows for $r+1$.

\item For $r\ge 2$, we have the following equalities:
\begin{align}
f_2 \tilde{\otimes}_1^{(r)} f_2 = \sum_{i=1}^{k_1} b_i^{r} h_i\tilde{\otimes}h_i + \sum_{r}^{k_2} c_i^{r} h'_i\tilde{\otimes} h'_i, \label{170416a}\\
(\ldots ( g_1 \tilde{\otimes}_1 g_2) \tilde{\otimes}_1 g_3)\ldots ) \tilde{\otimes}_1 g_r = \sum_{i=1}^{k_2}  c_i^{r-1}d_i h'_i, \label{170416b}
\end{align}
where the last equality holds if exactly one of the functions $g_1, \ldots, g_r$ equals $f_1$, all remaining functions $g_i$ being equal to $f_2$.
We notice that $(h_i \tilde{\otimes} h_i) \tilde{\otimes}_1 (h'_j \tilde{\otimes } h'_j)= 0$ and  $(h'_i \tilde{\otimes} h'_i) \tilde{\otimes}_1 (h'_j \tilde{\otimes } h'_j)= 1_{[i=j]} h'_i \tilde{\otimes} h'_i$ and $(h_i \tilde{\otimes} h_i) \tilde{\otimes}_1 (h_j \tilde{\otimes } h_j)= 1_{[i=j]} h_i \tilde{\otimes} h_i$. Hence $f_2 \tilde{\otimes}_1 f_2 =  \sum_{i=1}^{k_1} b_i^{2} h_i\tilde{\otimes}h_i + \sum_{i=1}^{k_2} c_i^2 h'_i\tilde{\otimes} h'_i$ and,   generally:
\begin{align*}
f_2 \tilde{\otimes}_1^{(r)} f_2 &= \sum_{i=1}^{k_1} b_i^{r} h_i\tilde{\otimes}h_i + \sum_{i=1}^{k_2} c_i^{r} h'_i\tilde{\otimes} h'_i.
\end{align*}

To prove Eq.~\eqref{170416b}, we suppose first that $g_2=f_1$ and $g_1= g_3 = \ldots = g_r=f_2$. Then:
\begin{align*}
g_1 \tilde{\otimes }_1 g_2 = f_2 \tilde{\otimes }_1 f_1 &= \left(
\sum_{i=1}^{k_1} b_i h_i \tilde{\otimes} h_i + \sum_{i=1}^{k_2} c_i h'_i \tilde{\otimes} h'_i
 \right) \tilde{\otimes}_1  \left( 
ah_0 +\sum_{i=1}^{k_2}d_i h'_i 
 \right)=\sum_{i=1}^{k_2} c_i d_i h'_i.
\end{align*}  
Since $g_1=g_3= \ldots=g_r=f_2$, it is now easy to see that:
\begin{align*}
(\ldots ( g_1 \tilde{\otimes}_1 g_2) \tilde{\otimes}_1 g_3)\ldots ) \tilde{\otimes}_1 g_r = \sum_{i=1}^{k_2}  c_i^{r-1}d_i h'_i.
\end{align*}
If, on the other hand $g_1=\ldots=g_l=f_2$ for $l>1$, we can use Eq.~\eqref{170416a} to see that 
\begin{align*}
(\ldots ( g_1 \tilde{\otimes}_1 g_2) \tilde{\otimes}_1 g_3)\ldots ) \tilde{\otimes}_1 g_l  
=f_2 \tilde{\otimes}_1^{(l)} f_2 =
\sum_{i=1}^{k_1} b_i^{l} h_i\tilde{\otimes}h_i + \sum_{i=1}^{k_2} c_i^{l} h'_i\tilde{\otimes} h'_i,
\end{align*}
and we can proceed as above to see that Eq.~\eqref{170416b} holds.

\item Considering Eq.~\eqref{170416c}, it is easy to see that 
\[
\sum_{r=1}^{\textrm{deg}(P)} \frac{P^{(r)}(0)}{r!2^{r-1}} \left(  
\Gamma_{r-1}(X) - \E[\Gamma_{r-1}(X)]
\right)
\]
lives in the sum of the first two Wiener chaoses. We consider the projection of the random variable above on the first respectively on the second Wiener chaos.
We have with Eq.~\eqref{170416c}:
\begin{align*}
&J_1\left(\sum_{r=1}^{\textrm{deg}(P)} \frac{P^{(r)}(0)}{r!2^{r-1}} \left(  
\Gamma_{r-1}(X) - \E[\Gamma_{r-1}(X)]
\right)
\right) = \frac{P'(0)}{1!2^0} J_1(X) \\
&\rule{5mm}{0mm}+\sum_{r=2}^{\textrm{deg}(P)} \frac{P^{(r)}(0)}{r! 2^{r-1}} 2^{r-2}\, \left( 
3I_1((\ldots ( f_2 \tilde{\otimes}_1 f_1 ) \tilde{\otimes}_1 f_2 ) \ldots ) \tilde{\otimes}_1 f_{2})\right. \\
&\rule{5mm}{0mm} + \left.    \sum_{j=3}^{r} \sum_{g_j=f_1 \atop g_i=f_2,\textrm{ for }i\ne  j}I_1((\ldots ( g_1 \tilde{\otimes}_1 g_2 ) \tilde{\otimes}_1 g_3 ) \ldots ) \tilde{\otimes}_1 g_{r})
\right) \\
&= \frac{P'(0)}{1!2^0}\, I_1(f_1) + \sum_{i=1}^{k_2} \sum_{r=2}^{\textrm{deg}(P)}\frac{P^{(r)}(0)}{r!2^{r-1}} I_1\left(  2^{r-2} \left( 3h'_i + (r-2)h'_i \right)d_ic_i^{r-1}\right)\\
&= \frac{P'(0)}{1!2^0}\, I_1(f_1) + \frac{1}{2}\sum_{i=1}^{k_2} \sum_{r=2}^{\textrm{deg}(P)}\frac{P^{(r)}(0)}{r!} I_1\left(  h'_i\right)(r+1)d_ic_i^{r-1}.
\end{align*}
If $a=0$, we have:
\begin{align}
\frac{P'(0)}{1!2^0} \, I_1(f_1) = \frac{1}{2} \sum_{i=1}^{k_2}\frac{P'(0)}{1!2^0} I_1(h'_i)\, 2d_i. \label{170416d}
\end{align}
If $a\ne 0$, we have $P'(0)=0$, hence Eq.~\eqref{170416d} holds in both cases and:
\begin{align*}
&J_1\left(\sum_{r=1}^{\textrm{deg}(P)} \frac{P^{(r)}(0)}{r!2^{r-1}} \left(  
\Gamma_{r-1}(X) - \E[\Gamma_{r-1}(X)]
\right)
\right)\\
&=\frac{1}{2}\sum_{i=1}^{k_2} \sum_{r=1}^{\textrm{deg}(P)}\frac{P^{(r)}(0)}{r!} I_1\left(  h'_i\right)(r+1)d_ic_i^{r-1}\\
&= \frac{1}{2}\sum_{i=1}^{k_2}\left[ 
d_i \sum_{r=1}^{\textrm{deg}(P)} \frac{P^{(r)}(0)}{r!} r c_i^{r-1}
\right] I_1(h'_i) 
+ \frac{1}{2} \sum_{i=1}^{k_2} \left[ \frac{d_i}{c_i} \sum_{r=1}^{\textrm{deg}(P)} \frac{P^{(r)}(0)}{r!} c_i^r   \right] I_1(h'_i)\\
&= \frac{1}{2} \sum_{i=1}^{k_2} d_i P'(c_i) I_1(h'_i) + \frac{1}{2} \sum_{i=1}^{k_2} \frac{d_i}{c_i}P(c_i) I_1(h'_i)=0.
\end{align*}
We have for the projection on the second Wiener chaos with Eq.~\eqref{170416a}:
\begin{align*}
&J_2\left(\sum_{r=1}^{\textrm{deg}(P)} \frac{P^{(r)}(0)}{r!2^{r-1}} \left(  
\Gamma_{r-1}(X) - \E[\Gamma_{r-1}(X)]
\right)
\right)\\
&=\sum_{r=1}^{\textrm{deg}(P)} \frac{P^{(r)}(0)}{r!2^{r-1}} 2^{r-1}I_2\left(  f_2 \tilde{\otimes}_1^{(r)}f_2\right)\\
&= \sum_{i=1}^{k_1} \sum_{r=1}^{\textrm{deg}(P)} \frac{P^{(r)}(0)}{r!} b_i^{r}I_2(h_i \tilde{\otimes} h_i)+ \sum_{i=1}^{k_2} \sum_{r=1}^{\textrm{deg}(P)} \frac{P^{(r)}(0)}{r!} c_i^{r}I_2(h'_i \tilde{\otimes} h'_i) \\
&= \sum_{i=1}^{k_1} P(b_i)I_2(h_i \tilde{\otimes}h_i) + \sum_{i=1}^{k_2} P(c_i) I_2(h'_i \tilde{\otimes} h'_i)=0.
\end{align*}

\item %Since $\sum_{r=1}^{\textrm{deg}(P)} \frac{P^{(r)}(0)}{r!2^{r-1}} \left(  
%\Gamma_{r-1}(X) - \E[\Gamma_{r-1}(X)]
%\right)$ equals
%\[
%J_1\left(\sum_{r=1}^{\textrm{deg}(P)} \frac{P^{(r)}(0)}{r!2^{r-1}} \left(  
%\Gamma_{r-1}(X) - \E[\Gamma_{r-1}(X)]
%\right)
%\right)+J_2\left(\sum_{r=1}^{\textrm{deg}(P)} %\frac{P^{(r)}(0)}{r!2^{r-1}} \left(  
%\Gamma_{r-1}(X) - \E[\Gamma_{r-1}(X)]
%\right)
%\right),
%\]
%the claim of the Lemma follows now directly.
Since $Z:=\sum_{r=1}^{\textrm{deg}(P)} \frac{P^{(r)}(0)}{r!2^{r-1}} \left(  
\Gamma_{r-1}(X) - \E[\Gamma_{r-1}(X)]
\right)$ equals $J_1(Z)+ J_2(Z)$,
the claim of the Lemma follows now directly.
\end{enumerate}
\end{proof}

\begin{theorem}\label{030416aab}
Consider $k_1,  k_2\ge 0$ and 
\[X=I_1(f_1)+I_2(f_2)=a N + \sum_{i=1}^{k_1} b_i(R_i^2-1)+\sum_{i=1}^{k_2} [c_i(P_i^2-1)+d_iP_i],\]
as defined in Eq.\eqref{210416a} and Eq.~\eqref{070416aaa} 
 where  $N, R_1, \ldots, R_{k_1}, P_1,\ldots, P_{k_2} \stackrel{\textrm{i.i.d.}}{\sim}\mathcal{N}(0,1)$. Suppose that at least one of the parameters $a, k_1, k_2$ is non-zero.
Consider a sequence $\left\lbrace F_n\right\rbrace_n$ of non-zero random variables
with $F_n=\sum_{i=1}^p I_i(f_{n, i})$ for $p\ge 2$ fixed and $\left\lbrace f_{n, i}\right\rbrace_n\subset H^{\odot i}$ for $1\le i \le p$. Define: \[P(x)=x^{1+1_{[a\ne 0]}} \prod_{j=1}^{k_1} (x-b_j)\prod_{j=1}^{k_2}(x-c_j)^2.\]
If $F_n \stackrel{\textrm{Law}}{\to} X$, as $n\to \infty$, then the following limits hold, as $n\to \infty$:
\begin{enumerate}[(1)]
\item $\displaystyle{ \kappa_r(F_n)\to \kappa_r(X),\quad\textrm{ for } r=1,\ldots,\textrm{deg}(P)}$,
\item $\displaystyle{ % \E\left[ \left| \E\left[  %\sum_{r=1}^{\textrm{deg}(P)} \frac{P^{(r)}(0)}{r!%2^{r-1}} M_{r-1} (F_n)\Big| F_n  \right] \right|  %\right]=
\E\left[ \left| \E\left[  \sum_{r=1}^{\textrm{deg}(P)} \frac{P^{(r)}(0)}{r!2^{r-1}} \left( \Gamma_{r-1}(F_n) - \E[\Gamma_{r-1}(F_n)] \right) \Big| F_n  \right] \right|  \right] \to 0 }$.
\end{enumerate}
\end{theorem}

\begin{proof}
The proof of this theorem is identical to \citep[Theorem 3.2, Proof of (i)$\to$ (iii) ]{azmoodeh}. It is enough to replace \citep[Lemma 3.1]{azmoodeh} by Lemma \ref{200416a}.
\end{proof}

\paragraph*{Proof of Theorem \ref{230416a}}
The main result of this paper  is now  a direct consequence of Theorem \ref{030416aa} and Theorem \ref{030416aab}.\qed \\

\begin{remark} \label{150316a}
\begin{enumerate}[(1)]
\item We notice that, for any sequences of  integrable random variables $\left\lbrace F_n\right\rbrace_n$ and $\left\lbrace G_n\right\rbrace_n$, we have that $\E[|F_n|]\to 0$ implies $\E[|\E[F_n|G_n]|]\to 0$, as $n\to \infty$:
\begin{align*}
\E[|\E[F_n|G_n]|] &\le \E [  \E[|F_n|~|G_n] ]   = \E[|F_n|] \to 0,\quad \textrm{as } n \to \infty.
\end{align*}
%We have used the conditional version of Jensen's inequality.
Hence, with the notations of Theorem \ref{230416a}, a set of sufficient conditions for $F_n\stackrel{\textrm{Law}}{\to} X$, as $n\to \infty$, is:
\begin{enumerate}[(1)]
\item[(1')]  $\displaystyle{\kappa_r(F_n)\to \kappa_r(X),\quad  \textrm{for }  r=1,\ldots, \textrm{deg}(P)}$,
\item[(2')] $\displaystyle{\E\left[ \left|    \sum_{r=1}^{\textrm{deg}(P)} \frac{P^{(r)}(0)}{r!2^{r-1}} \left( 
\Gamma_{r-1}(F_n) -\E[\Gamma_{r-1}(F_n)] 
\right)   \right|^\alpha \right]\to 0}$ ,
\end{enumerate}
where $\alpha\ge 1$. If $\alpha=2$,% yields a more restrictive condition (2') but 
we can use Eq.~\eqref{260416x} to calculate the expectation in (2').

\item For $\alpha=2$, conditions (1') and (2') can be expressed in terms of conditions for contractions. However the resulting conditions are usually complicated as indicates   the following example:\\

\indent \textit{Consider $p\ge 2$ even, $a\ne 0$, $b\ne 0$ and a sequence of functions $\left\lbrace f_{n, p}\right\rbrace\subset H^{\odot p}$ such that, as $n\to \infty$:
\begin{align}
p! \Vert f_{n, p}\Vert_{H^{\otimes p}}^2&\to a^2+2b^2,\label{241215gg}\\
\langle f_{n,  p} \tilde{\otimes}_{p/2} f_{n, p} , f_{n, p} \rangle_{H^{\otimes p}} &\to 8b^3 \left[ p! (p/2)! \begin{pmatrix}
p \\ p/2
\end{pmatrix}^2\right]^{-1} \label{241215g}.
 \end{align}
Define  sets $A_l$ and $B_l$ for every even integer $l$   with $2\le l\le 3p-4$:
\begin{align*}
A_l &:= \left\lbrace (s,t)\in \mathbb{N}^2 : 1\le s\le p-1 ~,~ 1\le t \le p\wedge(2p-2s) ~,~ 3p-2(s+t)=l \right\rbrace, \\
B_l &:= \left\lbrace s\in \mathbb{N} : 1\le s\le p-1 ~,~ 2p-2s=l \right\rbrace,
\end{align*}
and:
\begin{align*}
c(p, s, t)&:=p (s-1)! (t-1)! \begin{pmatrix}
p-1 \\ s-1
\end{pmatrix}^2 \begin{pmatrix}
p-1 \\ t-1
 \end{pmatrix} \begin{pmatrix}
 2p-2s-1 \\ t-1
 \end{pmatrix} , \\
 k(p,  s)&:=(s-1)! \begin{pmatrix}
 p-1 \\ s-1
 \end{pmatrix}^2.
\end{align*} 
Define for every even integer $l$  a condition $C_l$ by:
{\fontsize{9.0}{0}
\begin{align*}
\left\Vert 
\sum_{(s, t)\in A_l} 
c(p, s, t)
(f_{n, p} \tilde{\otimes}_s f_{n, p} ) \tilde{\otimes }_t f_{n, p} -2b \sum_{s\in B_l} 
k(p, s)
f_{n, p} \tilde{\otimes}_s f_{n, p}  
\right\Vert_{H^{\otimes l}}\to 0.
\end{align*}}
If conditions $C_l$ hold for every even integer $l$  with $2\le l \le 3p-4$, we have, as $n\to \infty$:
\[
I_p(f_{n, p})\stackrel{\textrm{Law}}{\to} aN+b(\xi^2-1),\quad N, \xi\stackrel{\textrm{i.i.d.}}{\sim} \mathcal{N}(0,1).
\]
}

The proof of this result is omitted, 
another example for such conditions can be found in \cite[Theorem 4.1]{azmoodeh}.   
\end{enumerate}
\end{remark}

\section{Applications}\label{applications}
 {
We notice that our results about convergence in law of sequences living in a finite sum of Wiener chaoses can be extended to match convergence in total variation, see for instance \cite[Appendix C]{Peccati}. Indeed a direct application of our results together with \cite[Lemma 3.3]{azmoodeh} or \cite[Theorem 3.1]{MR3003367} proves that we can replace the convergence in law of sequences living in a finite sum of Wiener chaoses by convergence in total variation.}

\subsection{Recovering  classical criteria}

As a direct consequence of Theorem \ref{230416a}, we get the following Corollary \ref{010516b} which  extends \cite[Theorem 3.2]{azmoodeh}. We have proved that (2) is necessary and sufficient for convergence in law whereas the authors in the cited reference need $L^2$-convergence of the conditional expectation to prove $F_n \stackrel{\textrm{Law}}{\to}X$, as $n\to \infty$. As pointed out in Remark \ref{230416b} this results from the Cauchy-Schwarz inequality. Moreover, Corollary \ref{010516b} extends the first part of \cite[Proposition 1.7]{2016arXiv161202279D} to more general linear combinations of independent central $\chi^2$ distributed random variables. As anticipated, if $b_1=\ldots=b_k=1$, the polynomial $P$ can be simplified.

\begin{corollary} \label{010516b}
Consider $X= I_2(f_2)=\sum_{i=1}^{k} b_i(R_i^2-1)$ for $R_1, \ldots, R_{k_1}\stackrel{\textrm{i.i.d.}}{\sim}\mathcal{N}(0, 1)$ and $b_i\ne 0 $ for $i=1, \ldots, k$, see Eq.~\eqref{210416a} and \eqref{070416aaa}. Define $P(x)=x\prod_{j=1}^{k}(x-b_i)$.
Let $\left\lbrace F_n\right\rbrace_n$ be a sequence  of  non-zero random variables such that each $F_n$ lives in a finite sum of Wiener  chaoses, i.e. $F_n= \sum_{i=1}^m I_i(f_{n, i})$ for $n\ge 1$ and $m\ge 2$ fixed, $\left\lbrace f_{n, i}\right\rbrace_n\subset H^{\odot i}$ for $1\le i \le m$. The following two asymptotic relations (1) and (2) are equivalent, as $n\to \infty$:
\begin{enumerate}[(1)]
\item  $F_n\stackrel{\textrm{Law}}{\to} X$;
\item \begin{enumerate}[(a)]
\item $\kappa_r(F_n)\to \kappa_r(X)$, for $r=2, \ldots, k+1$,
\item  $\displaystyle{\E\left[  \left| \E\left[ \sum_{r=1}^{k+1} \frac{P^{(r)}(0)}{r!2^{r-1}} \left( \Gamma_{r-1}(F_n) - \E[\Gamma_{r-1}(F_n)] \right)   \Big|F_n  \right]   \right|\right]  \to 0}$.
\end{enumerate}
\end{enumerate}
\end{corollary}

\begin{proof}
Use Theorem \ref{230416a} with $a=k_2=0$.
\end{proof}

We come now to the  seminal paper \cite{nualart2005} of Nualart and Peccati in which the authors have characterised the convergence in law to a standard normal random variable. Since this first paper, the conditions given in \cite{nualart2005} have been extended, see for instance \cite{nourdin2009b}. 
Taking $k_1=k_2=0$ and $a=1$, we have with Theorem \ref{230416a} the following characterisation which corresponds to condition $(iii')$ cited in Section \ref{history}:
\begin{corollary}
Consider $X= aI_1(h_0) $ for $I_1(h_0) \sim \mathcal{N}(0, 1)$ and $a\ne 0$. Let $\left\lbrace F_n\right\rbrace_n$ be a sequence  of non-zero random variables such that each $F_n$ lives in a finite sum of Wiener chaoses, i.e. $F_n= \sum_{i=1}^m I_i(f_{n, i})$ for $n\ge 1$ and $m\ge 2$ fixed, $\left\lbrace f_{n, i}\right\rbrace_n\subset H^{\odot i}$ for $1\le i\le m$. The following two asymptotic relations (1) and (2) are equivalent, as $n\to \infty$:
\begin{enumerate}[(1)]
\item  $F_n\stackrel{\textrm{Law}}{\to} X$,
\item \begin{enumerate}[(a)]
\item $\kappa_2(F_n)\to a^2$,
\item  $\displaystyle{\E\left[  \left| \E\left[ 
\langle DF_n ,-DL^{-1}F_n \rangle_H - a^2
  \Big|F_n  \right]   \right|\right]  \to 0}$.
\end{enumerate}
\end{enumerate}
\end{corollary}

\begin{proof}
The equivalence of (1) and (2) follows directly from Theorem \ref{230416a} with $P(x)=x^2$, since:
\[
\sum_{r=1}^2 \frac{P^{(r)}(0)}{r!2^{r-1}} \left( \Gamma_{r-1}(F_n)- \E[\Gamma_{r-1}(F_n)] \right) = \frac{1}{2}  \, \left( \Gamma_1(F_n) - \kappa_2(F_n) \right).
\]
If (2) holds, we have clearly $\lim_n \kappa_2(F_n)=a^2$. If (1) holds, the latter limit follows from \cite[Lemma 3.3]{azmoodeh}.
\end{proof}

\begin{remark}
The equivalence of (1) and (2) in the last corollary can be extended to  sequences $\left\lbrace F_n \right\rbrace_n$ living in $\mathbb{D}^{1,2}$ if some assumptions on the boundedness of the $\Gamma_1$-operator is added. Indeed we can prove that (1) implies (2) if $\sup_n \E[\Gamma_1(F_n)^2]<\infty $. %The latter condition holds in particular if $\sup_n \E[\Vert DF_n\Vert^4]<\infty $. 
It is unsure whether the equivalence of (1) and (2) holds without any additional assumptions if $\left\lbrace F_n \right\rbrace_n\subset\mathbb{D}^{1,2}$.
\end{remark}
\color{black}

We use now Remark \ref{150316a} to recover more criteria for the convergence in law, see Corollary \ref{170816ab}.
Notice that (1) is \cite[Theorem 5.3.1]{Peccati}  for the case of a sequence $\left\lbrace F_n\right\rbrace_n$ living in a fixed sum of Wiener chaoses, whereas (2) is the sufficient part of the criterion given in \cite[Eq. (1.3)]{nourdin2009b} with $L^2$-convergence replaced by $L^1$-convergence. Since all $L^p$-norms inside a fixed sum of Wiener chaoses are equivalent, the criterion given in (2) below is thus necessary and sufficient.

\begin{corollary} \label{170816ab}
Consider $X= aI_1(h_0) $ for $I_1(h_0) \sim \mathcal{N}(0, 1)$ and $a\ne 0$. Let $\left\lbrace F_n\right\rbrace_n$ be a sequence of non-zero random variables with $\kappa_2(F_n)\to a^2$, as $n\to \infty$.
\begin{enumerate}[(1)]
\item If $F_n= \sum_{i=1}^m I_i(f_{n, i})$ for $n\ge 1$, $m\ge 2$ fixed with $\left\lbrace f_{n, i}\right\rbrace_n\subset H^{\odot i}$ for $1\le i\le m$ and, as $n\to \infty$:
\[
\E\left[ \left| \langle DF_n ,-DL^{-1}F_n\rangle_H -a^2 \right|\right] \to 0,
\] 
then $F_n \stackrel{\textrm{Law}}{\to} X$, as $n\to \infty$.
\item If $F_n= I_p(f_{n, p})$ with $\left\lbrace f_{n, p}\right\rbrace_n\subset H^{\odot p}$ and, as $n\to \infty$:
\[
\E \left[ \left|1/p \Vert D F_n \Vert_{H}^2-a^2 \right|\right] \to 0 ,
\] 
then $F_n \stackrel{\textrm{Law}}{\to} X$, as $n\to \infty$.
\end{enumerate}
\end{corollary}

\subsection{New applications}
For the rest of the section we suppose that $p$ is  even. Notice that if $p$ is odd, we cannot have:
\[I_p(f_{n, p}) \stackrel{\textrm{Law}}{\to}aN + \sum_{i=1}^{k_1} b_i (R_i^2-1)
,\] as $n\to \infty$ for $k_1\ge 1$, $\sup_n \Vert f_{n, p}\Vert_{H^{\otimes p}}<\infty$ and  $\sum_{i=1}^k b_i^3 \ne 0$. This can be checked using the third cumulant and Eq.~\eqref{260416x}, see \cite[Remark 1.3]{nourdin2009b}.
More generally, if $p\ge 3$ is odd, we can exclude a large set of possible target random variables by using the fact that all odd-order cumulants are zero,
see Remark  \ref{141018f} for details.

The following Lemma shall be needed to prove a sufficient criterion based on the convergence of some contractions and cumulants.

\begin{lemma}\label{070316a}
Consider $p_1$, $p_2\ge 1$, $f\in H^{\odot p_1}$ and $g\in H^{\odot p_2}$. We have for $0\le r \le p_1 \wedge p_2$:
\begin{align*}
 \Vert f \tilde{\otimes}_r g \Vert_{p_1+p_2-2r}^2 &\le  \Vert f {\otimes}_r g \Vert_{p_1+p_2-2r}^2 
= \left| \langle   f \otimes_{p_1-r} f , g\otimes_{p_2-r} g \rangle_{H^{\otimes (2r)}}  \right| \\
&\le  \Vert f \otimes_{p_1-r} f \Vert_{H^{\otimes (2r)}} \,  \Vert g \otimes_{p_2-r} g \Vert_{H^{\otimes (2r)}} \le \Vert f \otimes_{p_1-r} f \Vert_{H^{\otimes (2r)}} \,  \Vert g  \Vert_{H^{\otimes p_2}}^2 \\
&\le   \Vert f  \Vert_{H^{\otimes p_1}}^2 \,  \Vert g  \Vert_{H^{\otimes p_2}}^2.
\end{align*}
\end{lemma}
\begin{proof}
The first inequality is a standard result. The first equality follows from \cite[Eq. (2.5)]{rosinski1}. The next inequality follows from the Cauchy-Schwarz inequality. The last inequalities follow from the following general result:
\[
\Vert f \otimes_r g \Vert_{H^{\otimes (p_1+p_2-2r)}}\le \Vert f \Vert_{H^{\otimes p_1}}\Vert g \Vert_{H^{\otimes p_2}},
\]
which is, again, a consequence of the Cauchy-Schwarz inequality.
\end{proof}

\begin{theorem}  \label{230416aa}
Consider $k_1, k_2\ge 0$ and 
\[X=I_1(f_1)+I_2(f_2)=a N + \sum_{i=1}^{k_1} b_i(R_i^2-1)+\sum_{i=1}^{k_2} [c_i(P_i^2-1)+d_iP_i],\]
as defined in Eq.~\eqref{210416a} and Eq.~\eqref{070416aaa} 
 where  $N, R_1, \ldots, R_{k_1}, P_1,\ldots, P_{k_2} {\sim}\mathcal{N}(0, 1)$. Suppose that at least one of the parameters $a, k_1, k_2$ is non-zero.
Consider a sequence $\left\lbrace F_n\right\rbrace_n$ of non-zero random variables with $F_n= I_p(f_{n, p})$ for $p\ge 4$ fixed, even, and $\left\lbrace f_{n, p}\right\rbrace_n\subset H^{\odot p}$. Define: \[P(x)=x^{1+1_{[a\ne 0]}} \prod_{j=1}^{k_1} (x-b_j)\prod_{j=1}^{k_2}(x-c_j)^2~; 
c_p:= 2 (p/2)! \begin{pmatrix}
p-1 \\ p/2-1
\end{pmatrix}^2 ~;%~f_{n, p}\tilde{\otimes}_{p/2}^{(1)} f_{n, p}:=f_{n, p},
\]
\[
f_{n, p}\tilde{\otimes}_{p/2}^{(1)} f_{n, p}:=f_{n, p}~;~ f_{n, p} \tilde{\otimes}_{p/2}^{(l+1)}f_{n, p} := \left( f_{n, p}\tilde{\otimes}_{p/2}^{(l)} f_{n, p}\right) \tilde{\otimes}_{p/2} f_{n, p},\quad  \textrm{ for } l\in \mathbb{N}_{>0}.
\]
%and generally
%$f_{n, p} \tilde{\otimes}_{p/2}^{(l+1)}f_{n, p} := \left( f_{n, p}\tilde{\otimes}_{p/2}^{(l)} f_{n, p}\right) \tilde{\otimes}_{p/2} f_{n, p}$.
As $n\to \infty$,  we have $I_p(f_{n, p})\stackrel{\textrm{Law}}{\to}X$ if the following hold:
\begin{align}
&\kappa_r(I_p(f_{n, p}))\to \kappa_r(X),\quad \textrm{for } r=1,\ldots, \textrm{deg } (P),\label{230516b} \\
&\Vert f_{n, p} \tilde{\otimes}_l f_{n, p} \Vert_{H^{\otimes (2p-2l)}} \to 0,\quad \textrm{ for every } 1\le l \le p-1 \textrm{ with } l\ne p/2,  \label{070416a}\\
&\left\Vert \sum_{r=1}^{\textrm{deg}P} \frac{P^{(r)}(0)}{r!2^{r-1}}  f_{n, p} \tilde{\otimes}_{p/2}^{(r)} f_{n, p} c_p^{r} \right\Vert_{H^{\otimes p}}\to 0. \label{050316m}
\end{align}
\end{theorem}

\begin{proof}
\begin{enumerate}[(1)]
\item We notice that, for $p\ge 4$ even, condition \eqref{070416a}  is equivalent to:
\begin{align*}
\Vert f_{n, p} {\otimes}_l f_{n, p} \Vert_{H^{\otimes (2p-2l)}} \to 0,\quad \textrm{ for every } 1\le l \le p-1 \textrm{ with } l\ne p/2,
\end{align*}
see \cite[Proposition 3.1]{nourdin2009b}. 
If $p=2$, we have necessarily $k_2=0$ and the convergence in law of sequences $\left\lbrace I_2(f_{n, 2})\right\rbrace_n$ is completely characterised by necessary and sufficient criteria in \citep{ECP2023}, see Remark \ref{070416dd}.
%Notice that both conditions are void if $p=2$. 
\item We first prove that, except for the contractions $f_{n, p}\tilde{\otimes}_{p/2}^{(r)} f_{n, p}$ for $r=2, \ldots, \textrm{deg}(P)$,
 all the contractions, appearing in the representation of  $M_{r-1}=\Gamma_{r-1}(I_p(f_{n, p}))-\E[\Gamma_{r-1}(I_p(f_{n, p}))]$ for $r=2,\ldots,  \textrm{deg}(P)$, converge to zero (in the corresponding Hilbert-space norm).
With \citep[Proposition 2.1]{azmoodeh}, we have for $i\ge 1$ that $\Gamma_{i}(I_p(f_{n, p}))-\E[\Gamma_i(I_p(f_{n, p}))]$ equals:
\begin{align*}
&\sum_{(r_1,\ldots,r_i)\in S_i} c_p(r_1, \ldots, r_i)\\
&\rule{5mm}{0mm} \times I_{(i+1)p-2r_1-\ldots-2r_i}((\ldots( f_{n,p}\tilde{\otimes}_{r_1}f_{n,p} )\tilde{\otimes}_{r_2}f_{n,p})\ldots f_{n,p} ) \tilde{ \otimes}_{r_i}f_{n,p}),
\end{align*} 
where $S_i$ is the set of elements $(r_1, \ldots, r_i)$ such that:
\begin{align*}
1\le r_1 \le p~,~ \ldots~,~ 1\le r_i  \le (ip-2r_1-\ldots-2r_{i-1})\wedge p;\\
r_1<p~,~\ldots~,~ r_1+\ldots+r_{i-1}<ip/2~,~(i+1)p-2r_1-\ldots-2r_i\ne 0,
\end{align*} 
and $c_p(r_1, \ldots, r_i)$ is defined recursively, see \citep[Proposition 2.1]{azmoodeh}. For $(r_1, \ldots, r_i)\in S_i\setminus \left\lbrace (p/2, \ldots, p/2) \right\rbrace$, we prove that, as $n\to \infty$: 
\begin{align}
\Vert
(\ldots( f_{n, p}\tilde{\otimes}_{r_1}f_{n, p} )\tilde{\otimes}_{r_2}f_{n, p})\ldots f_{n, p} ) \tilde{ \otimes}_{r_i}f_{n, p} \Vert_{H^{\otimes [(i+1)p-2r_1-\ldots-2r_i]}} \to 0. \label{070416b}
\end{align}
If $r_1\ne p/2$, we have $\Vert f_{n, p} \tilde{\otimes }_{r_1} f_{n, p}\Vert_{H^{\otimes (2p-2r_1)}} \to 0$, as  $n\to \infty$,
 since %$r_1\ne p/2$ and 
 $r_1\ne p$. If $i\ge 2$ and  $r_1=p/2$, let $r_1=\ldots=r_j=p/2$ and $r_{j+1}\ne p/2$ with $1\le j\le i-1$. 
We have $r_{j+1}\ne p$, since otherwise the definition of $S_i$ yields the following contradictions:
\[
\frac{jp}{2}+p= \sum_{k=1}^{j+1} r_k
%r_1+\ldots + r_j + r_{j+1}
< \frac{(j+2)p}{2},\quad \textrm{ if } j+1\le i-1,
\]
and:
\[
0\ne (i+1)p-2\sum_{k=1}^i r_k = (i+1)p-2(i-1 ) p/2 -2p =0,\quad \textrm{ if } j+1=i.
\]
Thus, with Lemma \ref{070316a}:
\begin{align*}
&\Vert
(\ldots( f_{n, p}\tilde{\otimes}_{r_1}f_{n, p} )\tilde{\otimes}_{r_2}f_{n, p})\ldots f_{n, p} ) \tilde{ \otimes}_{r_{j+1}}f_{n, p} \Vert_{H^{\otimes [(j+2)p-2r_1-\ldots-2r_{j+1}]}}^2 \\
&= \left\Vert
\left( f_{n, p} \tilde{\otimes}_{p/2}^{(j+1)}f_{n, p} \right) \tilde{\otimes}_{r_{j+1}} f_{n, p} 
 \right\Vert_{H^{\otimes [(j+2)p-2r_1-\ldots-2r_{j+1}]}}^2 \\☺
&\le \left\Vert f_{n, p}  \right\Vert_{H^{\otimes p} }^{2j+2} \, \Vert f_{n, p} \otimes_{p-r_{j+1}} f_{n, p}\Vert_{H^{\otimes( 2r_{j+1})}}.
\end{align*}
Since $1\le r_{j+1}\le p-1$ and $r_{j+1}\ne p/2$, we have that $1\le p-r_{j+1}\le p-1$ and $p-r_{j+1}\ne p/2$. Thus $\Vert f_{n, p} \otimes_{p-r_{j+1}} f_{n, p}\Vert_{H^{\otimes( 2r_{j+1})}}\to 0$, as $n\to \infty$, since $\sup_n \Vert f_{n, p}\Vert_{H^{\otimes p}}<\infty$, Eq.~\eqref{070416b} is proved.
\item With Eq.~\eqref{070416b}, we have for $i\ge 1$, as $n\to \infty$:
\begin{align*}
&\E\left[ 
\left(
\Gamma_i(I_p(f_{n, p}))- \E[\Gamma_i(I_p(f_{n, p}))] - c_p(p/2, \ldots, p/2) I_p\left( f_{n, p} \tilde{\otimes}_{p/2}^{(i+1)}f_{n, p} \right)
 \right)^2
\right]\\
&=\E\left[ 
\left(
\Gamma_i(I_p(f_{n, p}))- \E[\Gamma_i(I_p(f_{n, p}))] - c_p^i I_p\left( f_{n, p} \tilde{\otimes}_{p/2}^{(i+1)}f_{n, p} \right)
 \right)^2
\right]
\to 0,
\end{align*}
where we have used the definition of $c_p$ and \cite[Proposition 2.1]{azmoodeh}. 
For $i=0$, we have trivially $\E[(\Gamma_0(I_p(f_{n, p}))-\E[\Gamma_0(I_p(f_{n, p}))]-I_p(f_{n, p}))^2]=0$.
Thus with $M_{r-1}:=\Gamma_{r-1}(I_p(f_{n, p}))- \E[\Gamma_{r-1}(I_p(f_{n, p}))]$, as $n\to \infty$:
\begin{align}
\E\left[ 
\left( \sum_{r=1}^{{\textrm{deg}(P)}}\frac{P^{(r)}(0)}{r!2^{r-1}} \left(
M_{r-1} -c_p^{r-1} I_p\left( f_{n, p} \tilde{\otimes}_{p/2}^{(r)}f_{n, p} \right)\right)
 \right)^2
\right]\to 0. \label{070416c}
\end{align}
%\begin{align}
%\E\left[ 
%\left( \sum_{r=1}^{{\textrm{deg}(P)}}\frac{P^{(r)}(0)}{r!2^{r-1}} \left(
%\Gamma_{r-1}(I_p(f_{n, p}))- \E[\Gamma_{r-1}(I_p(f_{n, p}))] -c_p^{r-1} I_p\left( f_{n, p} \tilde{\otimes}_{p/2}^{(r)}f_{n, p} \right)\right)
% \right)^2
%\right]\to 0. \label{070416c}
%\end{align}
On the other hand, as $n\to \infty$:
\begin{align}
&\E\left[\left( \sum_{r=1}^{{\textrm{deg}(P)}}\frac{P^{(r)}(0)}{r!2^{r-1}}  
c_p^{r-1} I_p\left( f_{n, p} \tilde{\otimes}_{p/2}^{(r)}f_{n, p} \right)
 \right)^2 \right]\nonumber \\
 &= c_p^{-2}p! \left\Vert
 \sum_{r=1}^{\textrm{deg}(P)} \frac{P^{(r)}(0)}{r! 2^{r-1}} c_p^r f_{n, p} \tilde{\otimes}_{p/2}^{(r)} f_{n, p}
 \right\Vert^2_{H ^{\otimes p}} \to 0. \label{070416d}
\end{align}
  The result follows with Eq.~\eqref{070416c} and \eqref{070416d}, the reverse triangle inequality, Theorem \ref{230416a} and Remark \ref{150316a}.
  \end{enumerate}
\end{proof}

\begin{remark} \label{070416dd}
If $p=2$ and $\left\lbrace f_{n, 2}\right\rbrace_n\subset H^{\odot 2}$, it is known that  $I_2(f_{n,2})\stackrel{\textrm{Law}}{\to}X'$
 as $n\to \infty$ implies that $X'$ has a representation as in Eq.~\eqref{231116a} with $k_1\le \infty$ and $k_2=0$. 
In particular, if $k_1<\infty$, $k_2=a=0$ and $p=2$,   Theorem \ref{230416aa} holds and is the sufficient part of \cite[Proposition 3.1]{azmoodeh}. It is shown in \citep{azmoodeh} that the conditions of Theorem \ref{230416aa} are also necessary for the case $p=2$. The case $p=2$, $k_1<\infty$, $a\ne 0$ and $k_2= 0$ is not covered by \citep{azmoodeh} but it can be shown that the conditions given in Theorem \ref{230416aa} are also necessary conditions for this case.
This can be proved using \cite[Theorem 3.4]{ECP2023} and \cite[Eq. (2.7.17)]{Peccati}. 
\end{remark}

 {
We give now an example for the use of Theorem \ref{230416aa} in the case $k_2=0$  where the approximating sequence lives in a fixed Wiener chaos. We notice that it is so far unknown whether a sequence living in a fixed chaos can converge to a target variable as in Theorem \ref{230416aa} with $k_2\ne 0$, see Remark \ref{141018f}.
}

\begin{example}
Consider $p\ge 2 $ and sequences of non-zero functions $\left\lbrace f_{n, p} \right\rbrace_n $, $ \left\lbrace g_{n, p} \right\rbrace_n \subset H^{\odot p}$ and $\left\lbrace  h_{n, 2p}\right\rbrace_n \subset H^{\odot (2p)}$ such  that the sequences $\left\lbrace 
I_p(f_{n, p}) \right\rbrace_n$, $\left\lbrace 
I_p(g_{n, p}) \right\rbrace_n$ and $\left\lbrace 
I_{2p}(h_{n, 2p}) \right\rbrace_n$ converge in law to a standard normal variable. We also suppose that $\langle f_{n,p}, g_{n,p}\rangle_{H^{\otimes p}}\to 0$, as $n\to \infty$. We have then, see \cite[Theorem 1]{MR2126978}, as $n\to \infty$:
\begin{align*}
(I_p(f_{n,p}), I_p(g_{n,p}), I_{2p}(h_{n,2p}))^\top\stackrel{\textrm{Law}}{\to} (R_1, R_2, N)^\top,
\end{align*}
where $(R_1, R_2, N)^\top$ has a $3$-dimensional standard normal distribution. Using the continuous mapping theorem,  \cite[Theorem 1]{nualart2005} and proceeding as in \cite[Proposition 4.1]{nourdin2009b}, it is easy to see that, as $n\to \infty$:
\begin{align*}
I_{2p}(\varphi_{n,2p})\stackrel{\textrm{Law}}{\to}X= N+(R_1^2-1)- (R_2^2-1),
\end{align*} 
where $\varphi_{n,2p}:=f_{n, p}\tilde{\otimes} f_{n, p}- g_{n, p}\tilde{\otimes} g_{n, p}+ h_{n, 2p}$.
Alternatively the latter convergence in law can be proved using Theorem \ref{230416aa}. 
Define: \begin{align}
k_p:=\frac{1}{(p/2)! \begin{pmatrix}
p-1 \\ p/2-1
\end{pmatrix}^2} = \frac{4}{ (p/2)! \begin{pmatrix}
p \\ p/2 
\end{pmatrix}^2}~,~c_p:=2 (p/2)! \begin{pmatrix}
p-1 \\ p/2-1
\end{pmatrix}^2. \label{300516a}
\end{align}
%\begin{align}
%k_p:=\left[ (p/2)! \begin{pmatrix}
%p-1 \\ p/2-1
%\end{pmatrix}^2 \right]^{-1} = 4 \left[ (p/2)! \begin{pmatrix}
%p \\ p/2 
%\end{pmatrix}^2  \right]^{-1}~,~c_p:=2 (p/2)! \begin{pmatrix}
%p-1 \\ p/2-1
%\end{pmatrix}^2. \label{300516a}
%\end{align}
We shall need the following limits, as $n\to \infty$:
\begin{enumerate}[(a)]
\item $\langle f_{n, p} \tilde{\otimes}f_{n, p} , g_{n, p} \tilde{\otimes} g_{n, p}\rangle_{H^{\otimes (2p)}}\to 0$,
\item $ \langle f_{n, p} \tilde{\otimes} f_{n, p} , h_{n, 2p}\rangle_{H^{\otimes (2p)}}\to 0 $ and $ \langle g_{n, p} \tilde{\otimes} g_{n, p} , h_{n, 2p}\rangle_{H^{\otimes (2p)}}\to 0 $,
\item $\Vert (f_{n, p} \tilde{\otimes} f_{n, p})\tilde{\otimes}_p (g_{n, p} \tilde{\otimes} g_{n, p})  \Vert_{H^{\otimes(2p)}}\to 0$,
\item $\Vert (f_{n, p} \tilde{\otimes} f_{n, p}) \tilde{\otimes}_p h_{{n, 2p}} \Vert_{H^{\otimes (p)}}\to 0 $ and  $\Vert (g_{n, p} \tilde{\otimes} g_{n, p}) \tilde{\otimes}_p h_{{n, 2p}} \Vert_{H^{\otimes (p)}}\to 0 $,
\item $\Vert \varphi_{n, 2p}\tilde{\otimes}_p^{(r)} \varphi_{n, 2p} -  k_{2p}^{r-1} f_{n, p} \tilde{\otimes} f_{n, p} -k_{2p}^{r-1} g_{n, p} \tilde{\otimes} g_{n, p} \Vert_{H^{\otimes (2p)}}\to 0$ for $r=2, 3, 4$,
\item $\Vert \varphi_{n, 2p}\tilde{\otimes}_l \varphi_{n, 2p}  \Vert_{H^{\otimes (4p-2l)}}\to 0$ for $1\le l \le 2p-1$ and $l\ne p$.
\end{enumerate}
These statements (a)-(f). can be proved using \cite[Theorem 1]{nualart2005} and \cite[Theorem 1.2]{nourdin2009b}
We  check the conditions of Theorem \ref{230416aa}.
\begin{enumerate}[(1)]
\item We prove that $\lim_n \kappa_4(I_{2p}(\varphi_{n, 2p}))=\kappa_4(X)$. The convergence of the cumulants of order 2 and 3 can be seen similarly using (a)-(f) and  \cite[Eq. (3.4)]{nourdin2009b}. We use $\kappa_4(I_{2p}(\varphi_{n, 2p}))=\E[I_{2p}(\varphi_{n, 2p})^4]-3\E[I_{2p}(\varphi_{n,  2p})^2]^2$ and   \cite[Eq. (3.6)]{nourdin2009b}:
\begin{align*}
&\lim_n \kappa_4(I_{2p}(\varphi_{n,2p}))\\
&= \lim_n  \frac{3}{2p} \sum_{r=1}^{2p-1}(2p)^2 (r-1)! \begin{pmatrix}
2p-1 \\ r-1
\end{pmatrix}^2 r! \begin{pmatrix}
2p \\ r 
\end{pmatrix}^2 \, (4p-2r)! \\
&\rule{5mm}{0mm}\times\Vert \varphi_{n, 2p} \tilde{\otimes}_r \varphi_{n, 2p}\Vert^2_{H^{\otimes(4p-2r)}}\\
&= \lim_n  \frac{3}{2p} (2p)^2 (p-1)! \begin{pmatrix}
2p-1 \\ p-1
\end{pmatrix}^2 
p! \begin{pmatrix}
2p \\ p 
\end{pmatrix}^2 \, (2p)! \Vert \varphi_{n, 2p} \tilde{\otimes}_p \varphi_{n, 2p}\Vert^2_{H^{\otimes(2p)}}\\
&= \lim_n  \frac{3}{2p} (2p)^2 (p-1)! \begin{pmatrix}
2p-1 \\ p-1
\end{pmatrix}^2 
p! \begin{pmatrix}
2p \\ p 
\end{pmatrix}^2 \, (2p)! \\
&\rule{5mm}{0mm}\times\Vert 
k_{2p}f_{n, p}\tilde{\otimes} f_{n, p} + k_{2p}g_{n, p}\tilde{\otimes}g_{n, p}
\Vert^2_{H^{\otimes(2p)}}\\
&= \lim_n  \frac{3}{2p} (2p)^2 (p-1)! \begin{pmatrix}
2p-1 \\ p-1
\end{pmatrix}^2 
p! \begin{pmatrix}
2p \\ p 
\end{pmatrix}^2 \, (2p)! \\
&\rule{5mm}{0mm} \times \left( k_{2p}^2\Vert f_{n, p}\tilde{\otimes} f_{n, p}
\Vert^2_{H^{\otimes(2p)}}+k_{2p}^2\Vert g_{n, p}\tilde{\otimes} g_{n, p}
\Vert^2_{H^{\otimes(2p)}}\right)\\
&= \lim_n  \frac{3}{2p} (2p)^2 (p-1)! \begin{pmatrix}
2p-1 \\ p-1
\end{pmatrix}^2 
p! \begin{pmatrix}
2p \\ p 
\end{pmatrix}^2 \, (2p)! \\
&\rule{5mm}{0mm}\times\left( \Vert (f_{n, p}\tilde{\otimes} f_{n, p})\tilde{\otimes}_p (f_{n, p} \tilde{\otimes} f_{n, p})
\Vert^2_{H^{\otimes(2p)}}\right.\\
&\rule{5mm}{0mm}+\left.\Vert (g_{n, p}\tilde{\otimes} g_{n, p})\tilde{\otimes}_p (g_{n, p} \tilde{\otimes} g_{n, p})
\Vert^2_{H^{\otimes(2p)}}\right)\\
&= \kappa_4(R_1^2-1) + \kappa_4(R_2^2-1)=\kappa_4(X).
\end{align*}
The fifth equality follows from \cite[Theorem 1.2.(iv)]{nourdin2009b}. For the last equality we have used that, for independent random variables, the cumulant of the sum equals the sum of the cumulants and $\kappa_r(N)=0$ for every $r\ge 3$. We have thus proved Eq.~\eqref{230516b}. Notice that Eq.~\eqref{070416a} holds because of (f). 
We have $P(x)=x^2 (x-1)(x+1)$ and:\[
P'(0)=P'''(0)=0~,~ {P''(0)}/(2!2)=-1/2~, ~{P^{(4)}(0)}/(4!2^3)=1/8.
\]
We check now Eq.~\eqref{050316m}:
\begin{align*}
&\left\Vert \sum_{r=1}^4 \frac{P^{(r)}(0)}{r!2^{r-1}} \varphi_{n, 2p} \tilde{\otimes}_p^{(r)} \varphi_{n, 2p}c_{2p}^r \right\Vert_{H^{\otimes(2p)}} \\
&\le \left\Vert -\frac{c_{2p}^2}{2} \left( \varphi_{n, 2p} \tilde{\otimes}_p \varphi_{n, 2p}-k_{2p} f_{n, p}\tilde{\otimes} f_{n, p}-k_{2p}g_{n, p}\tilde{\otimes} g_{n, p} \right)\right\Vert_{H^{\otimes (2p)}}\\
&\rule{5mm}{0mm}+ \left\Vert \frac{c_{2p}^4}{8} \left( \varphi_{n, 2p} \tilde{\otimes}_p^{(4)} \varphi_{n, 2p}-k_{2p}^3 f_{n, p}\tilde{\otimes} f_{n, p}-k_{2p}^3g_{n, p}\tilde{\otimes} g_{n, p} \right)\right\Vert_{H^{\otimes (2p)}}\\
&\rule{5mm}{0mm}+ c_{2p}^2 k_{2p}\left( \left\Vert -\frac{1}{2} f_{n, p} \tilde{\otimes} f_{n, p} + \frac{c_{2p}^2k_{2p}^2}{8} f_{n, p} \tilde{\otimes} f_{n, p}\right\Vert_{H^{\otimes (2p)}}\right. \\
&\rule{5mm}{0mm} +\left.  \left\Vert -\frac{1}{2} g_{n, p} \tilde{\otimes} g_{n, p} + \frac{c_{2p}^2k_{2p}^2}{8} g_{n, p} \tilde{\otimes} g_{n, p}\right\Vert_{H^{\otimes (2p)}}\right)\to 0,
\end{align*}
as $n\to \infty$. Theorem \eqref{230416aa} yields now $I_{2p}(\varphi_{n, 2p})\stackrel{\textrm{Law}}{\to} X$, as $n\to \infty$.
\end{enumerate}
\end{example}

We consider now the convergence in law to random variable with a centered $\chi^2$ law and compare the new criterion of   Theorem \ref{230416aa} with the main result of \cite{nourdin2009b}. We shall see that in this case, both criteria are equivalent.

\begin{theorem} \label{240416c}
Consider $k_1> 0$, $a=k_2=0$ and 
$X= \sum_{i=1}^{k_1} (R_i^2-1)$ as defined in  Eq.~\eqref{210416a} and Eq.~\eqref{070416aaa}, where  $R_1, \ldots,  R_{k_1}\stackrel{\textrm{i.i.d.}}{\sim}\mathcal{N}(0, 1)$. 
Consider a sequence $\left\lbrace F_n\right\rbrace_n$ of non-zero random variables with $F_n=  I_p(f_{n, p})$ for $p\ge 2$ even fixed and $\left\lbrace f_{n, p}\right\rbrace_n\subset H^{\odot p}$. Define $ P(x)=x  (x-1)^{k_1}$ and suppose that $\kappa_2(F_n)\to 2k_1 $, as $n\to \infty$. The following conditions a., b. and c. are equivalent, as $n\to \infty$: 
\begin{enumerate}[a.]
\item $F_n \stackrel{\textrm{Law}}{\to} X$,% as $n\to \infty$,
\item \begin{enumerate}[1)]
\item $\displaystyle{\Vert f_{n, p} \tilde{\otimes}_l f_{n, p} \Vert_{H^{\otimes (2p-2l)}} \to 0,\quad \textrm{ for every } 1\le l \le p-1 \textrm{ with } l\ne p/2,}$,
\item $\Vert  f_{n, p} \tilde{\otimes}_{p/2} f_{n, p} - 2/c_p f_{n, p}\Vert_{H^{\otimes p}}\to 0$ 
\end{enumerate}
\item \begin{enumerate}[1)]
\item $\displaystyle{\Vert f_{n, p} \tilde{\otimes}_l f_{n, p} \Vert_{H^{\otimes (2p-2l)}} \to 0,\quad \textrm{ for every } 1\le l \le p-1 \textrm{ with } l\ne p/2,}$,
\item $\displaystyle{\left\Vert \sum_{r=1}^{\textrm{deg}(P)} \frac{P^{(r)}(0)}{r!2^{r-1}}  f_{n, p} \tilde{\otimes}_{p/2}^{(r)} f_{n, p} c_p^{r} \right\Vert_{H^{\otimes p}} \to 0}$ 
\item $\kappa_r(I_p(f_{n,p}))\to \kappa_r(X),\quad \textrm{for } r=1,\ldots, \textrm{deg } (P).$
\end{enumerate}
\end{enumerate}
\end{theorem}

\begin{proof}
The equivalence of a. and b. is proved in \cite{nourdin2009b}. We prove the equivalence of b. and c. for $k_1\ge 2$ since b. and c. are clearly equivalent if $k_1=1$.

 Suppose that b. holds, than a. holds as well and $\sup_n \E[F_n^2]<\infty$ together with the hypercontractivity property yields that $\sup_n \E[|F_n|^r]<\infty$ for every $r\ge 1$,   thus $\E[F_n^r]\to \E[X^r]$ for every $r\in \mathbb{N}$. Hence c.3) holds. 
  \begin{comment}
 Since b.2)  holds, we have, as $n\to \infty$:
\begin{align}
\Vert c_p f_{n, p} \tilde{\otimes}_{p/2} (c_p/2 f_{n, p}-1) \Vert_{H^{\otimes p}} =\left\Vert c_p^2/2 f_{n, p}\tilde{\otimes}_{p/2} f_{n, p} -c_p f_{n, p}\right\Vert_{H^{\otimes p}}\to 0, \label{070416f}
\end{align}
and this implies, since $\sup_n \Vert f_{n,p} \Vert <\infty$:
\begin{align*}
\Vert (c_p f_{n,p} \tilde{\otimes}_{p/2} (c_p/2f_{n,p}-1)) \tilde{\otimes}_{p/2}  \ldots ) \tilde{\otimes}_{p/2} (c_p/2f_{n,p}-1)\Vert_{H^{\otimes p}} \to 0,
\end{align*}
as $n\to \infty$, where we have considered the iterated contraction and $k$-times the function $c_p/2f_n-1$.  Expanding $(c_p f_{n,p} \tilde{\otimes}_{p/2} (c_p/2f_{n,p}-1)) \tilde{\otimes}_{p/2}  \ldots ) \tilde{\otimes}_{p/2} (c_p/2f_{n,p}-1)$ we find:
\begin{align*}
\sum_{l=0}^k \begin{pmatrix}
k\\l
\end{pmatrix} \frac{c_p^{l+1}}{2^l} f_{n,p}\tilde{\otimes}_{p/2}^{(l+1)}f_{n,p} (-1)^{k-l} &=\sum_{l=1}^{k+1} \begin{pmatrix}
k\\l-1
\end{pmatrix} \frac{c_p^{l}}{2^{l-1}} f_{n,p}\tilde{\otimes}_{p/2}^{(l)}f_{n,p} (-1)^{k-l+1}
= \sum_{l=1}^{k+1} \frac{P^{(l)}(0)}{l!2^{l-1}} c_p^l f_{n,p} \tilde{\otimes}_{p/2}^{(l)}f_{n,p},
\end{align*} \end{comment}
We have with $k=k_1$:
\begin{align*}
&\sum_{l=1}^{k+1}\frac{P^{(l)}(0)}{l!2^{l-1}}c_p^l f_{n, p} \tilde{\otimes}_{p/2}^{(l)} f_{n, p} = (-1)^k \sum_{l=1}^{k+1}
\begin{pmatrix}
k \\ l-1
\end{pmatrix} (-2)^{1-l}
c_p^l f_{n, p} \tilde{\otimes}_{p/2}^{(l)} f_{n, p} \\
&= (-1)^k \left[
\sum_{l=1}^k c_p^{l+1 } (-2)^{-l} \sum_{r=l}^k \begin{pmatrix}
k \\r
\end{pmatrix} (-1)^{r-l} \left(
f_{n, p} \tilde{\otimes}_{p/2}^{(l+1)} f_{n, p} - \frac{2}{c_p} f_{n, p} \tilde{\otimes}_{p/2}^{(l)} f_{n, p}
 \right) \right.
 \\&\rule{5mm}{0mm} - \left.c_p (-2)^{-1+1} \sum_{r=1}^k
\begin{pmatrix}
k \\r 
\end{pmatrix} (-1)^{r-1} f_{n, p} + \begin{pmatrix}
k \\ 0 
\end{pmatrix} (-2)^0 c_p f_{n, p}  \right]\\
&= (-1)^k \left[
\sum_{l=1}^k c_p^{l+1 } (-2)^{-l} \sum_{r=l}^k \begin{pmatrix}
k \\r
\end{pmatrix} (-1)^{r-l} \left(
f_{n, p} \tilde{\otimes}_{p/2}^{(l+1)} f_{n, p} - \frac{2}{c_p} f_{n, p} \tilde{\otimes}_{p/2}^{(l)} f_{n, p}
 \right) \right.\\
&\rule{5mm}{0mm} - \left. (-1)  c_p 
f_{n, p} \sum_{r=0}^k \begin{pmatrix}
k \\r 
\end{pmatrix} (-1)^r
 \right]\\
&= (-1)^k \left[
\sum_{l=1}^k c_p^{l+1 } (-2)^{-l} \sum_{r=l}^k \begin{pmatrix}
k \\r
\end{pmatrix} (-1)^{r-l} \left(
f_{n, p} \tilde{\otimes}_{p/2}^{(l+1)} f_{n, p} - \frac{2}{c_p} f_{n, p} \tilde{\otimes}_{p/2}^{(l)} f_{n, p}
 \right) \right],
\end{align*}
where we have used that $\sum_{r=0}^k \begin{pmatrix}
k \\ r 
\end{pmatrix} (-1)^r =(-1+1)^k=0$.
We have for $l\ge 2$, as $n\to \infty$:
\begin{align*}
&\left\Vert f_{n, p} \tilde{\otimes}_{p/2}^{(l+1)} f_{n, p} - 2/c_p f_{n, p} \tilde{\otimes}_{p/2}^{(l)} f_{n, p} \right \Vert_{H^{\otimes p}}\\
&= \left\Vert ( \ldots ( f_{n, p}\tilde{\otimes}_{p/2} f_{n, p} - 2/c_p f_{n, p} ) \tilde{\otimes}_{p/2} f_{n, p} ) \ldots ) \tilde{\otimes}_{p/2} f_{n, p}\right\Vert_{H^{\otimes p}}\\
&\le \left\Vert f_{n, p} \tilde{\otimes}_{p/2} f_{n, p} - 2/c_p f_{n, p} \right \Vert_{H^{\otimes p}} \, \Vert f_{n, p}\Vert_{H^{\otimes p}}^{l-1}\to 0,
\end{align*}
hence c.2) follows with the triangle inequality.
Suppose now that c. holds. We have to prove that $\Vert c_p^2 f_{n, p}\tilde{\otimes}_{p/2}^{(2)}f_{n, p}-2c_p f_{n, p}\Vert^2_{H^{\otimes p}}\to 0$, as $n\to \infty$, or equivalently: 
\begin{align}
&c_p^4 \Vert f_{n, p} \tilde{\otimes}_{p/2}^{(2)} f_{n, p}\Vert^2_{H^{\otimes p}}-4c_p^3  \langle f_{n, p} \tilde{\otimes}_{p/2}^{(2)} f_{n, p}, f_{n, p}\rangle_{H^{\otimes p}} + 4 c_p^2 \Vert f_{n, p}\Vert_{H^{\otimes p}}^2\nonumber\\
&=c_p^4 \langle f_{n, p} \tilde{\otimes}_{p/2}^{(3)} f_{n, p}, f_{n, p}\rangle_{H^{\otimes p}} + 4c_p^2 \Vert f_{n, p}\Vert_{H^{\otimes 2}}^2 \nonumber \\
&\rule{5mm}{0mm}- 4c_p^3  \langle f_{n, p} \tilde{\otimes}_{p/2} f_{n, p}, f_{n, p}\rangle_{H^{\otimes p}}\to 0. \label{240416a}
\end{align}  
We have used in the last step that: \[\langle f_{n, p}\tilde{\otimes}_{p/2}^{(k)} f_{n, p}, f_{n, p}\tilde{\otimes}_{p/2}^{(l)} f_{n, p}\rangle_{H^{\otimes p}}=\langle f_{n, p}\tilde{\otimes}_{p/2}^{(k+1)} f_{n, p}, f_{n, p}\tilde{\otimes}_{p/2}^{(l-1)} f_{n, p}\rangle_{H^{\otimes p}},\] for all integers $k, l$ with $k\ge 1$ and $l\ge 2$. This can be checked directly using the integral representation of the contractions. 
\begin{itemize}
\item If $k_1=2$, c.2) yields that, as $n\to \infty$:
\begin{align*}
\left\Vert c_p f_{n, p} - c_p^2 f_{n, p}\tilde{\otimes}_{p/2} f_{n, p} + c_p^3/4 f_{n, p} \tilde{\otimes}_{p/2}^{(3)} f_{n, p} \right\Vert_{H^{\otimes p}}\to 0,
\end{align*}
hence, since $\sup_n \Vert f_{n, p}\Vert_{H^{\otimes p}}<\infty$, as $n\to \infty$:
\begin{align*}
\langle 4c_p f_{n, p}~,~ c_p f_{n, p} - c_p^2 f_{n, p}\tilde{\otimes}_{p/2} f_{n, p} + c_p^3/4 f_{n, p} \tilde{\otimes}_{p/2}^{(3)} f_{n, p} \rangle_{H^{\otimes p}}\to 0.
\end{align*}
The linearity of the scalar product yields that  relation \eqref{240416a} holds.
\item If $k_1\ge 3$, the calculations made in the proof of Theorem \ref{230416aa}, together with \citep[Theorem 8.4.4]{Peccati} and c.3) show that: 
\begin{align*}
0 &=\lim_n\left[\kappa_i(F_n) - p! (i-1)! c_p^{i-2} \langle f_{n, p}\tilde{\otimes}_{p/2}^{(i-1)} f_{n, p}, f_{n, p}\rangle_{H^{\otimes p}}\right]\\
0 &= \lim_n \left[ \langle f_{n, p}\tilde{\otimes}_{p/2}^{(i-1)} f_{n, p}, f_{n, p}\rangle_{H^{\otimes p}} - \frac{\kappa_i(F_n)}{p!(i-1)! c_p^{i-2}} \right]\\
&= \lim_n \left[ \langle f_{n, p}\tilde{\otimes}_{p/2}^{(i-1)} f_{n, p}, f_{n, p}\rangle_{H^{\otimes p}}\right] - \frac{2^{i-1}k_1}{p! c_p^{i-2}},
\end{align*}
for $i=3,\ldots, k_1+1$. Thus:
\begin{align*}
&\lim_n \left[c_p^4 \langle f_{n, p} \tilde{\otimes}_{p/2}^{(3)} f_{n, p}, f_{n, p}\rangle_{H^{\otimes p}} + 4c_p^2 \Vert f_{n, p}\Vert_{H^{\otimes 2}}^2 \right.\\
&\rule{5mm}{0mm}\left.- 4c_p^3  \langle f_{n, p} \tilde{\otimes}_{p/2} f_{n, p}, f_{n, p}\rangle_{H^{\otimes p}}\right] =\frac{k_1c_p^2(8+8-16)}{p!}=0.
\end{align*}
\end{itemize}
This completes the proof.
\end{proof}

\begin{remark} \label{141018f}
\begin{enumerate}[(1)]
\item The equivalence of a. and b. above is the main result of \cite[Theorem 1.2]{nourdin2009b}.
Theorem \ref{240416c} shows that,  although the polynomials defined throughout this paper may not always  be of minimal degree if some of the coefficients are equal, the criterion of Theorem  \ref{230416aa} is   \textit{necessary} and sufficient    in some situations. 
It is  yet unknown  if the conditions of Theorem \ref{230416aa} are always   \textit{necessary} and sufficient.
\item Notice that the problem of characterising possible target variables amongst all random variables with a representation as in Eq.~\eqref{070416aaa} is far from being solved if the approximating sequence lives in a fixed Wiener chaos. Indeed further research on this topic is needed in order to make a comprehensive statement, but it seems doubtful that $c_1(P_1^2-1)+d_1P_1$ with $c_1, d_1\ne 0$ is a possible target for a sequence living in a chaos of fixed order. For  every  sequence living in a Wiener chaos of odd order,  all moments of odd order vanish which reduces the class of possible target variables $X=\sum_{i=1}^k   [c_i (P_i^2-1)+d_iP_i]$ with $P_1,\ldots, P_k\stackrel{\textrm{i.i.d.}}{\sim}\mathcal{N}(0,1)$ considerably. The following criterion allows to exclude such target variables by just considering the coefficients $c_i$ and $d_j$. Suppose without loss of generality that the coefficients satisfy the following conditions:
\begin{itemize}
\item $c_1, \ldots , c_m\in \mathbb{R}^*$,
\item $c_{m+1}=\ldots=c_k=0$,
\item the coefficients $c_1,\ldots, c_m$ are sorted in increasing order of their absolute values and:
\begin{align*}
|c_1|&=|c_2|=\ldots= |c_{k_1}|< |c_{k_1+1}|=\ldots= |c_{k_2}|\\
&< \ldots < |c_{k_{l-1}+1}|=\ldots =|c_{k_l}|,
\end{align*}
where $k_0:=0< k_1\le k_2 \le \ldots \le k_l=m$. 
\end{itemize}
Then $I_p(f_{n,p})\stackrel{\textrm{Law}}{\to} X$, as $n\to \infty$, for $p\ge 3$ odd, implies that:
\begin{enumerate}[(i)]
\item $m$ is even,
\item $k_i$ is even for every $i\in \left\lbrace 1 , \ldots, l\right\rbrace$ and, more precisely:
\[
|\left\lbrace k_{i-1}+1\le j\le k_i : c_j>0\right\rbrace|= |\left\lbrace k_{i-1}+1\le j\le k_i : c_j<0\right\rbrace|.
\]
\item we have for every $i\in \left\lbrace 1 ,\ldots , l\right\rbrace$:
\[
\sum_{i=k_{i-1}+1}^{k_i} d_i^2 \cdot c_i =0.
\]
\end{enumerate}
An analogue version holds for target variables in class (B) of Lemma \ref{230516c}.
\end{enumerate}
\end{remark}

\color{black}

\section{Stable convergence}\label{stable}
In this section we consider sequences of non-zero random variables living in a finite sum of Wiener chaoses. The sequences are supposed to converge in law to a non-zero target variable $X$ and we ask whether the sequence also converges stably. Our first result follows from \cite[Theorem 1.3]{rosinski2}. As before, we shall write $\varphi_Y $ for the characteristic function of a random variable $Y$.

\begin{theorem} \label{190916T1}
Consider $p\ge 3$ and the random variable  $X\stackrel{\textrm{Law}}{=}aN+\sum_{i=1}^{k_1}b_i(R_i^2-1)+\sum_{i=1}^{k_2}[c_i(P_i^2-1)+d_iP_i]$ %as defined in Eq.~\eqref{210416a} and \eqref{070416aaa}
 where $N, R_1, \ldots, R_{k_1}, P_1,\ldots, P_{k_2}\stackrel{\textrm{i.i.d.}}{\sim}\mathcal{N}(0, 1)$.
Assume that    $\left\lbrace g_{n, p}\right\rbrace_n \subset H^{\odot p} $ %with    $\sup_n \Vert g_{n, p}\Vert_{H^{\otimes p}}<\infty  $ 
and, as $n\to \infty$:
\begin{align}
%I_p(g_{n, p}) \stackrel{\textrm{Law}}{\to} X ,\nonumber\\
\Vert g_{n, p} \otimes_{p-1} g_{n, p}\Vert_{H\otimes H} \to 0
 \label{190616a}.
\end{align}
If $I_p(g_{n,p})\stackrel{\textrm{Law}}{\to}X$, as $n\to \infty$, then:
\begin{align}
I_p(g_{n,p}) \stackrel{\textrm{st.}}{\to } X,\label{190616b}
\end{align}
or, equivalently, for every $\mathcal{F}$-measurable random variable $Z$:
\begin{align*}
\left( I_p(g_{n,p}), Z \right)^\top 
\stackrel{\textrm{Law}}{\to} \left( X, Z \right)^\top,
\end{align*}
where $\mathcal{F}$ is introduced in Section \ref{wiener}, and $X$ is independent of the underlying Brownian motion.
\end{theorem}

\begin{proof}
Consider $f\in H$ with $\Vert f \Vert_H=1 $. We have to prove that, as $n\to \infty$:
\begin{align}
(
I_1(f), I_p(g_{n, p}))^\top\stackrel{\textrm{Law}}{\to} ( I_1(f) , X 
)^\top , \label{131116a}
\end{align}
see Section \ref{stableconvergence}, where $I_1(f)$ is independent of   $X$.
Since $I_p(g_{n,p})$ is a non-zero random variable, we have $\Vert g_{n,p}\Vert_{H^{\otimes p}}\ne 0$ for every $n\in \mathbb{N}$. Define $g'_{n,p}:=\frac{g_{n,p}}{\sqrt{p! \Vert g_{n,p}\Vert_{H^{\otimes p}}^2}}$ for every $n$, then $I_p(g'_{n,p})\stackrel{\textrm{Law}}{\to} X/\sqrt{A}$, as $n\to \infty$, where $A:=a^2+2\sum_{i=1}^{k_1}b_i^2 + \sum_{i=1}^{k_2} (2c_i^2+d_i^2) \ne 0$. We have $\E[I_p(g'_{n,p})^2]=1$ and with Eq.~\eqref{190616a}, as $n\to \infty$:
\begin{align*}
\Vert g'_{n,p} \otimes_1 f\Vert^2_{H^{\otimes (p-1)}}\le \Vert g_{n,p} \otimes_{p-1} g_{n,p} \Vert_{H^{\otimes 2}} \, \Vert f \Vert_H^2 \, \frac{1}{p! \Vert g_{n,p}\Vert^2_{H^{\otimes p}}}\to 0.
\end{align*}
\cite[Theorem 1.3]{rosinski2} yields with $a_n:=1/\sqrt{p! \Vert g_{n,p}\Vert^2_{H^{\otimes p}}} \to a :={1}/{\sqrt{A}}$, as $n\to \infty$:
\begin{align*}
&\varphi_{I_1(f)}(t_1)\,\varphi_{X}(t_2/\sqrt{A}) = \varphi_{I_1(f)}(t_1)\,\varphi_{X/\sqrt{A}}(t_2)\\
&= \lim_n \E\left[ \exp\left(it_1 I_1(f) + it_2  I_p(g'_{n,p})    \right)\right] \\
&=\lim_n \E\left[ \exp\left(it_1 I_1(f) + it_2 a_n I_p(g_{n,p})    \right)\right] \\
&= \lim_n \left( \E\left[ \exp\left(it_1 I_1(f) + it_2 a  I_p(g_{n,p})    \right)\right] \right.\\
&\rule{5mm}{0mm}+\left.
\E\left[ \exp\left(it_1 I_1(f) + it_2 a  I_p(g_{n,p})    \right) \, \left( \exp(i t_2 (a_n-a) I_p(g_{n,p})) -1 \right) \right] 
 \right).
\end{align*}
We have with $|e^{ix}-1|\le |x|$ for every $x\in \mathbb{R}$:
\begin{align*}
&\left|  
\E\left[ \exp\left(it_1 I_1(f) + it_2 a  I_p(g_{n,p})    \right) \, \left( \exp(i t_2 (a_n-a) I_p(g_{n,p})) -1 \right) \right] 
\right| \\
&\le \E\left[ \left| \exp(i t_2 (a_n-a) I_p(g_{n,p})) -1  \right| \right] \le \E\left[ \left| t_2 (a_n-a) I_p(g_{n,p})  \right|\right] \\
&\le |t_2 (a_n-a) | \E[I_p(g_{n,p})^2]^{1/2} \le |t_2(a_n-a)| \, \sqrt{p! \sup_n \Vert g_{n,p} \Vert_{H^{\otimes p}}^2} \to 0,
 \end{align*}
 as $n\to \infty$. Hence:
 \begin{align*}
\varphi_{I_1(f)}(t_1)\,\varphi_{X}(t_2/\sqrt{A}) &=\lim_n  \E\left[ \exp\left(it_1 I_1(f) + it_2 a  I_p(g_{n,p})    \right)\right],
 \end{align*}
and with $t'_2:=t_2/\sqrt{A}$\,:
 \begin{align*}
\varphi_{I_1(f)}(t_1)\,\varphi_{X}(t'_2) &=\lim_n  \E\left[ \exp\left(it_1 I_1(f) + it'_2   I_p(g_{n,p})    \right)\right].
 \end{align*}
%
%Since Eq.~\eqref{190616a} holds, we have, as $n\to \infty$:
%\begin{align*}
%\Vert g_{n, p}\otimes_1 f\Vert_{H^{\otimes (p-1)}}^2 = \langle g_{n, p}\otimes_{p-1} g_{n, p}, f\otimes f\rangle_{H\otimes H}\le \Vert g_{n, p}\otimes_{p-1} g_{n, p}\Vert_{H\otimes H}\to 0,  
%\end{align*}
%and  \cite[Theorem 1.3]{rosinski2} yields: 
%\[
%\E\left[ \exp\left( it_1 I_1(f)+it_2 I_p(g_{n, p}) \right) \right]\to \varphi_{I_1(f)}(t_1)\varphi_X(t_2),
%\]
%as $n\to \infty$. 
Since $(t_1,t'_2)\mapsto \varphi_{I_1(f)}(t_1)\varphi_X(t'_2)$ is the characteristic function of $(I_1(f), X)^\top$ where $I_1(f)$ is independent of $X$, we have that Eq.~\eqref{131116a}, or equivalently Eq.~\eqref{190616b} holds. Since  $f\in  H$ with $\Vert f \Vert_H=1$ is arbitrary, the statement follows with the remarks of Section \ref{stableconvergence}.  
\end{proof}

\begin{remark}
The proof of the previous theorem was straightforward since the converging sequence of random variables lives in a fixed Wiener chaos. Under these assumptions, \cite[Theorem 1.3]{rosinski2} yields the desired stable convergence. If the sequence of random variables is allowed to live in a finite sum of Wiener chaoses or if assumption \eqref{190616a} does not hold, the conditions ensuring stable convergence involve $\Gamma$-operators.
\end{remark}

For the next results, the target variables are (again) supposed to have the form:
\begin{align}
X\stackrel{\textrm{Law}}{=} a N + \sum_{i=1}^{k_1}b_i (R_i^2-1)+\sum_{i=1}^{k_2}\left[ c_i(P_i^2-1)+d_iP_i \right],\label{261116aa}
\end{align}
where $N, R_1, \ldots, R_{k_1}, P_1,  \ldots, P_{k_2}$ are independent standard normal variables and $b_i\ne 0 $ for $1\le i \le k_1$, as well as $c_jd_j\ne 0$ for $1\le j\le k_2$. We suppose that at least one of the parameters $k_1, k_2$ is positive. We shall use the convention $0^0:=1$.

\begin{theorem}\label{191016ta}
Consider $p\ge 3$, $k_1, k_2\ge 0$  and $X$ as defined in Eq.~\eqref{261116aa}.
Suppose that $\left\lbrace g_{n, l}\right\rbrace_n\subset H^{\odot l}$  %with $\sup_n\Vert g_{n, l}\Vert_{H^{\otimes l}}<\infty$ 
for $1\le l \le p$ and:
\begin{align*}
F_n=\sum_{l=1}^p I_l(g_{n, l}),%\quad \textrm{ with } \Vert g_{n, p }\Vert_{H^{\otimes p}}\ne 0.
\end{align*}
% thus F_n is not in the sum of the first two wiener chaoses.
Define $P(x)=x^{1+1_{[a\ne 0]}}\prod_{i=1}^{k_1} (x-b_i)\prod_{i=j}^{k_2} (x-c_j)^2$ and suppose that, as $n\to \infty$:
\begin{align}
&\kappa_r(F_n)\to k(r):=1_{[r=2]}a^2 +\sum_{i=1}^{k_1}2^{r-1}(r-1)!b_i^r \nonumber\\
&\rule{5mm}{0mm} +\sum_{j=1}^{k_2}2^{r-1}(r-1)!\left( c_j^r    +\frac{r c_j^{r-2} d_j^2}{4} \right),\quad \textrm{ for } 2\le r\le \textrm{deg}(P), \label{191016a}\\
&\sum_{r=1}^{\textrm{deg}(P)} \frac{P^{(r)}(0)}{r!2^{r-1}} \left( \Gamma_{r-1}(F_n) - \E[\Gamma_{r-1}(F_n)] \right)\stackrel{{L^1}}{\to} 0.\label{191016b}
\end{align}
If the following two conditions hold, as $n\to \infty$, for every $f\in H$ with $\Vert f \Vert_H=1$ and $k:=2k_2+k_1$:
\begin{align}
&\E\left[ \exp\left(i(t_1I_1(f)+t_2F_n)\right) \langle DI_1(f), DF_n\rangle_H \right]\to 0 ,\quad \textrm{ for every }(t_1, t_2)\in \mathbb{R}^2, \label{191016c}\\
& it_1 \E\Bigl[ \exp\left( i(it_1 I_1(f)+t_2F_n) \right)\,\langle DI_1(f) , \sum_{r=1}^{\textrm{deg}(P)} \frac{P^{(r)}(0)}{r!2^{r-1}} \, \sum_{\alpha+\beta=r-1 \atop \alpha\ge 1,\beta\ge 0}(it_2)^{k+1_{[a\ne 0]}-\alpha}\Bigr. \nonumber\\
&\rule{5mm}{0mm} \times \Bigl.\left( -DL^{-1}\Gamma_\beta (F_n) \right)\rangle_H \Bigr]\to 0 ,\quad \textrm{ for every }(t_1,t_2)\in \mathbb{R}^2, \label{191016d}
\end{align} 
then $F_n\stackrel{\textrm{st.}}{\to} X$, as $n\to \infty$, and  $X$ is independent of the underlying Brownian motion.
\end{theorem}

\begin{proof}
Theorem \ref{040416a}, Eq.~\eqref{191016a} and Eq.~\eqref{191016b} imply that $F_n\stackrel{\textrm{Law}}{\to} X$, as $n\to \infty$. Define:
\begin{align*}
B_0 &:= \E\left[ \exp\left( i(t_1I_1(f)+t_2F_n) \right) \langle DI_1(f), DF_n\rangle_H   \right],\\
A_{m}&:=\E\left[ \exp\left( i (t_1I_1(f) + t_2 F_n) \right) \langle DI_1(f), -DL^{-1}\Gamma_m(F_n)\rangle_H \right]1_{[m\ge 0]},\\
\varphi_n(t_1, t_2)&:=\E\left[  \exp\left( i (t_1I_1(f) + t_2 F_n) \right) \right].
\end{align*}
The obvious dependence  of $(t_1,t_2)$ is dropped in the first equalities. The proof is divided in three steps:
\begin{itemize}
\item we derive two equations involving derivatives of $\varphi_n$,
\item we prove that the sequence $\left\lbrace(I_1(f), F_n)^\top\right\rbrace_n$ is tight and that we have $(I_1(f), F_n)^\top \stackrel{\textrm{Law}}{\to}(I_1(f), X)^\top$, as $n\to \infty$, where $X$ is independent of $I_1(f)$,
\item we conclude that $F_n\stackrel{\textrm{st.}}{\to }X$, as $n\to \infty$, and  $X$ is independent of the underlying Brownian motion. 
\end{itemize}
\begin{enumerate}[(1)]
\item Consider  $(t_1, t_2)\in \mathbb{R}\times \mathbb{R}^*$. Then for $r\ge 2 $:% with $\Gamma_s(F_n):=0$ for $s<0$:
\begin{align*}
&\E\left[ \exp\left( i (t_1I_1(f) + t_2 F_n) \right) \Gamma_{r-1}(F_n)   \right]\\
&= \E\left[ \exp\left( i (t_1I_1(f) + t_2 F_n) \right) \langle DF_n, -DL^{-1}\Gamma_{r-2}(F_n)\rangle_H  \right]\\
&= \frac{1}{it_2}\E\left[ \exp\left( i (t_1I_1(f) + t_2 F_n) \right) \langle it_2 DF_n+it_1 DI_1(f), -DL^{-1}\Gamma_{r-2}(F_n)\rangle_H\right]\\
&\rule{5mm}{0mm}-\frac{it_1}{it_2} \E\left[ \exp\left( i (t_1I_1(f) + t_2 F_n) \right) \langle DI_1(f), -DL^{-1}\Gamma_{r-2}(F_n)\rangle_H \right]\\
&= \frac{1}{it_2}\E\left[ \exp\left( i (t_1I_1(f) + t_2 F_n) \right) \Gamma_{r-2}(F_n) \right] - \frac{1}{it_2} \varphi_n(t_1, t_2)\E[\Gamma_{r-2}(F_n)]\\
&\rule{5mm}{0mm}- \frac{it_1}{it_2}A_{r-2}.
\end{align*}
We have used the integration by parts rule and $-\delta DL^{-1}F=F-\E[F]$ for every $F\in L^2(\Omega)$. Repeating this first calculations  once again if $r-2\ge 1$, we find:
\begin{align*}
&\E\left[ \exp\left( i (t_1I_1(f) + t_2 F_n) \right) \Gamma_{r-1}(F_n)   \right]\\
&=\frac{1}{it_2}\left( 
\frac{1}{it_2}\E\left[ \exp\left( i (t_1I_1(f) + t_2 F_n) \right)\Gamma_{r-3}(F_n)  \right]-\frac{1}{it_2}\varphi_n(t_1, t_2)\E\left[ \Gamma_{r-3}(F_n) \right]\right.\\
&\rule{5mm}{0mm}-\left.\frac{it_1}{it_2}A_{r-3}\right)- \frac{1}{it_2}\varphi_n(t_1, t_2)\E[\Gamma_{r-2}(F_n)]-\frac{it_1}{it_2}A_{r-2},
\end{align*}
%
%\begin{align*}
%&\E\left[ \exp\left( i (t_1I_1(f) + t_2 F_n) \right) \Gamma_{r-1}(F_n)   \right]\\
%&=\frac{1}{it_2}\left( 
%\frac{1}{it_2}\E\left[ \exp\left( i (t_1I_1(f) + t_2 %F_n) \right)\Gamma_{r-3}(F_n)  \right]\right.\\
%&\rule{5mm}{0mm}-\left.\frac{1}{it_2}\varphi_n(t_1, t_2)\E\left[ \Gamma_{r-3}(F_n) \right]-\frac{it_1}{it_2}A_{r-3}\right)- \frac{1}{it_2}\varphi_n(t_1, t_2)\E[\Gamma_{r-2}(F_n)]-\frac{it_1}{it_2}A_{r-2},
%\end{align*}
%
and iteration yields finally for $r\ge 1$:
\begin{align*}
&\E\left[ \exp\left( i (t_1I_1(f) + t_2 F_n) \right) \Gamma_{r-1}(F_n)   \right]\\
&= \frac{1}{(it_2)^{r-1}} \E\left[ \exp\left( i (t_1I_1(f) + t_2 F_n) \right) \Gamma_0(F_n) \right]\\
&\rule{5mm}{0mm}-  \varphi_n(t_1, t_2)\sum_{\alpha+\beta=r-1 \atop \alpha,\beta\ge 1} \frac{1}{(it_2)^\alpha} \E[\Gamma_\beta(F_n)]  - it_1 \sum_{\alpha+\beta=r-1 \atop \alpha\ge 1,\beta\ge 0} \frac{1}{(it_2)^\alpha} A_{\beta}.
\end{align*}
With $\frac{\partial \varphi_n}{\partial t_2} (t_1, t_2)=i\E\left[    \exp\left( i (t_1I_1(f) + t_2 F_n) \right) F_n \right]$, we find:
\begin{align*}
&\E\left[  \exp\left( i (t_1I_1(f) + t_2 F_n) \right) \Gamma_{r-1}(F_n) \right]\\
&= \frac{-i}{(it_2)^{r-1}}\frac{\partial \varphi_n}{\partial t_2} (t_1, t_2)-\varphi_n(t_1, t_2)\sum_{\alpha+\beta=r-1 \atop \alpha,\beta\ge 1} \frac{1}{(it_2)^\alpha}\E\left[ \Gamma_\beta(F_n) \right]\\
&\rule{5mm}{0mm}-it_1 \sum_{\alpha+\beta=r-1 \atop \alpha\ge 1 ,\beta\ge 0} \frac{1}{(i t_2)^\alpha} A_\beta.
\end{align*}
Hence, for all $(t_1, t_2)\in \mathbb{R}\times \mathbb{R}^*$:
\begin{align}
&\E\left[  \exp\left( i (t_1I_1(f) + t_2 F_n) \right) \sum_{r=1}^{\textrm{deg}(P)}\frac{P^{(r)}(0)}{r!2^{r-1}}\left( \Gamma_{r-1}(F_n) -\E\left[  \Gamma_{r-1}(F_n)\right] \right) \right]\nonumber\\
&=-\sum_{r=1}^{\textrm{deg}(P)} \frac{P^{(r)}(0)}{r!2^{r-1}}\frac{i}{(it_2)^{r-1}} \frac{\partial \varphi_n}{\partial t_2} (t_1, t_2)\nonumber \\
&\rule{5mm}{0mm}- \varphi_n(t_1, t_2) \sum_{r=1}^{\textrm{deg}(P)} \frac{P^{(r)}(0)}{r!2^{r-1}} \sum_{\alpha+\beta=r-1 \atop \alpha,\beta \ge 1 }\frac{1}{(it_2)^\alpha} \,\frac{\kappa_{\beta+1}(F_n)}{\beta!}\nonumber\\
&\rule{5mm}{0mm}- \varphi_n(t_1, t_2) \sum_{r=1}^{\textrm{deg}(P)}\frac{P^{(r)}(0)}{r!2^{r-1}}\frac{\kappa_{r}(F_n)}{(r-1)!}- it_1 \sum_{r=1}^{\textrm{deg}(P)}\frac{P^{(r)}(0)}{r!2^{r-1}} \sum_{\alpha+\beta=r-1 \atop \alpha\ge 1,\beta\ge 0} \frac{1}{(it_2)^\alpha} A_\beta .\label{191016ga}
\end{align}

With $\Gamma_1(I_1(f))-1=0$, we find for $(t_1, t_2)\in \mathbb{R}^*\times \mathbb{R}$:
\begin{align}
0&= \E\left[ \exp\left( i(t_1 I_1(f) + t_2F_n) \right) \,\left( \Gamma_1(I_1(f))-1 \right) \right]\nonumber\\
&=\E\left[\exp\left( i(t_1 I_1(f) + t_2F_n) \right) \langle DI_1(f),-DL^{-1}I_1(f)\rangle_H \right]- \varphi_n(t_1, t_2)\nonumber\\
&= \frac{1}{it_1}\E\left[ \exp\left( i(t_1 I_1(f) + t_2F_n) \right) \langle it_1DI_1(f)+it_2 DF_n,-DL^{-1}I_1(f)\rangle_H \right]\nonumber\\
&\rule{5mm}{0mm}-\frac{it_2}{it_1}\E\left[ \exp\left( i(t_1 I_1(f) + t_2F_n) \right) \langle DF_n,-DL^{-1}I_1(f)\rangle_H\right] -\varphi_{n}(t_1, t_2)\nonumber\\
&= \frac{1}{i^2t_1} \E\left[ \exp\left( i(t_1 I_1(f) + t_2F_n) \right)iI_1(f) \right]-\frac{it_2}{it_1}\E\left[\exp\left( i(t_1 I_1(f) + t_2F_n) \right)\right. \nonumber\\
&\rule{5mm}{0mm}\times \left.
\langle DF_n,DI_1(f)\rangle_H \right]
-\varphi_n(t_1, t_2)\nonumber\\
&= -\frac{1}{t_1}\left( \frac{\partial \varphi_n}{\partial t_1} (t_1, t_2)+t_1\varphi_n(t_1, t_2)\right) - \frac{it_2}{it_1} B_0. \label{191016gb}
\end{align}

\item We prove that the sequence $\left\lbrace(I_1(f), F_n)^\top\right\rbrace_n$ is tight. We have with the Markov inequality, for $K\to \infty$:
\begin{align*}
\p\left( I_1(f)^2+F_n^2\ge K \right) &\le K^{-1} \E[I_1(f)^2+ F_n^2]=K^{-1} \left( 1+\E[F_n^2] \right)\\
&\le K^{-1}\left( 1+\sup_n \E[F_n^2]\right)\to 0.
\end{align*}
Consider a subsequence   $\left\lbrace(I_1(f), F_{n_l})^\top\right\rbrace_l$ which converges in law to a random vector $V=(V_1, V_2)^\top$. If we have that $V$ is the same for every converging subsequence and $V \stackrel{\textrm{Law}}{=} (I_1(f), X)^\top $, where $I_1(f)$ and %$X'\stackrel{\textrm{Law}}{=}X$ 
$X$ are independent, then we have that $\left\lbrace(I_1(f), F_n)^\top\right\rbrace_n$ converges in law to the same limit.
We have that $\varphi_{n_l}(t_1, t_2)= \E\left[ \exp\left(i (t_1 I_1(f)+t_2F_{n_l})\right) \right]\to 
\varphi_{(V_1, V_2)^\top}(t_1, t_2)=: \varphi^*(t_1, t_2)$, as $l\to \infty$. Since we have Eq.~\eqref{191016b} and Eq.~\eqref{191016d}, we have for $(t_1, t_2)\in \mathbb{R}\times \mathbb{R}^*$ and $l\to \infty$ with Eq.~\eqref{191016ga}:
\begin{align}
&-\sum_{r=1}^{\textrm{deg}(P)} \frac{P^{(r)}(0)}{r!2^{r-1}}\frac{i}{(it_2)^{r-1}} \frac{\partial \varphi^*}{\partial t_2} (t_1, t_2) \nonumber \\
&\rule{5mm}{0mm}- \varphi^*(t_1, t_2) \sum_{r=1}^{\textrm{deg}(P)} \frac{P^{(r)}(0)}{r!2^{r-1}} \sum_{\alpha+\beta=r-1 \atop \alpha,\beta \ge 1 }\frac{1}{(it_2)^\alpha} \frac{k(\beta+1)}{\beta!}\nonumber\\
&\rule{5mm}{0mm}- \varphi^*(t_1, t_2) \sum_{r=1}^{\textrm{deg}(P)}\frac{P^{(r)}(0)}{r!2^{r-1}}\frac{k(r)}{(r-1)!}=0.\label{191016gc}
\end{align}
Notice that, as $l\to \infty$, we have $\frac{\partial \varphi_{n_l}}{\partial t_2} \to \frac{\partial \varphi^*}{\partial t_2}$ by the continuous mapping theorem and $\kappa_{\beta+1}(F_n)/\beta!\to k(\beta+1)/\beta!$, as $n\to \infty$,  since Eq.~\eqref{191016a} holds. 
We have:
\begin{align*}
\frac{\partial \varphi^*}{\partial t_2}(t_1, t_2)=\lim_l\frac{\partial \varphi_{n_l}}{\partial t_2}(t_1, t_2)=\lim_l \E\left[ \exp\left( i(t_1I_1(f) + t_2F_{n_l}) \right)i F_{n_l} \right],
\end{align*}
hence:
\begin{align*}
&\frac{\partial \varphi^*}{\partial t_2}(t_1, 0)=\lim_l \E\left[ \exp\left( it_1I_1(f) \right)i F_{n_l} \right]\\
&=-
\lim_l \E\left[  \exp\left( it_1 I_1(f) \right) \langle t_1 DI_1(f) , -DL^{-1}F_{n_l} \rangle_p \right]
=0.
\end{align*}
To see this, apply  Eq.~\eqref{191016d}   with $t_2=0$.
We find similarly
$\frac{\partial \varphi^*}{\partial t_1}(0,t_2)=0$.
Multiplying Eq.~\eqref{191016gc} by $(it_2)^{k+1_{[a\ne 0]}}$, we find   the following   equation for $(t_1,t_2)\in \mathbb{R}^2$:
\begin{align}
&-\sum_{r=1}^{\textrm{deg}(P)} \frac{P^{(r)}(0)}{r!2^{r-1}}{i}{(it_2)^{k+1_{[a\ne 0]}-r+1}} \frac{\partial \varphi^*}{\partial t_2} (t_1, t_2) \nonumber\\
&\rule{5mm}{0mm}- \varphi^*(t_1, t_2) \sum_{r=1}^{\textrm{deg}(P)} \frac{P^{(r)}(0)}{r!2^{r-1}} \sum_{\alpha+\beta=r-1 \atop \alpha,\beta \ge 1 }{(it_2)^{k+1_{[a\ne 0]}-\alpha}} \frac{k(\beta+1)}{\beta!}\nonumber\\
&\rule{5mm}{0mm}- (it_2)^{k+1_{[a\ne 0]}}\varphi^*(t_1, t_2) \sum_{r=1}^{\textrm{deg}(P)}\frac{P^{(r)}(0)}{r!2^{r-1}}\frac{k(r)}{(r-1)!}=0, \label{191016gh}
\end{align}
with $\varphi^*(0, 0)=1$ and $\frac{\partial \varphi^*}{\partial t_1}(0, t_2)=\frac{\partial \varphi^*}{\partial t_2}(t_1, 0)=0$.
From Eq.~\eqref{191016gb},  {we find with $n$ replaced by $n_l$ and $l\to \infty$, for $(t_1, t_2)\in \mathbb{R}^*\times \mathbb{R}$:
\begin{align}
 \frac{\partial \varphi^*}{\partial t_1} (t_1, t_2)+t_1\varphi^*(t_1, t_2)=0.\label{191016ge}
\end{align}
We have used that condition \eqref{191016c} implies that $B_0$ converges to 0 in Eq.~\eqref{191016gb}.} Since $\frac{\partial \varphi^*}{\partial t_1}(0, t_2)=0$,
the differential equation \eqref{191016ge} holds for every $(t_1,   t_2)\in \mathbb{R}^2$.
We have thus a system of partial differential equations, given by Eq.~\eqref{191016gh} and Eq.~\eqref{191016ge}, with the conditions $\varphi^*(0, 0)=1$ and $\frac{\partial \varphi^*}{\partial t_1}(0, t_2)=\frac{\partial \varphi^*}{\partial t_2}(t_1, 0)=0$. 
We %write $\varphi_Z$ for the characteristic function of a random variable $Z$ and 
notice that the calculations in the proof of Theorem \ref{030416aa}, in particular Step 1, Eq.~\eqref{040416a} and Eq.~\eqref{030416c} show that the  function  $(t_1, t_2)\mapsto\varphi_X(t_2)$ satisfies Eq.~\eqref{191016gh} and we have $\varphi_X(0)=1$ as well as  $\varphi'_X(0)=0$. For a standard normal random variable $M$, the  function $(t_1, t_2)\mapsto \varphi_M(t_1)$ satisfies Eq.~\eqref{191016ge}  and we have $\varphi_M(0)=1$ as well as $\varphi'_M(0)=0$. A solution  of the system is   given by the function $(t_1, t_2)\mapsto \varphi_M(t_1)\varphi_X(t_2)$. Suppose that $\tilde{\varphi}$ is another solution of the system. Define $\Psi(t_1, t_2):=\tilde{\varphi}(t_1, t_2)/(\varphi_M (t_1)\varphi_X(t_2))$. 
Notice that $\varphi_M (t_1)\varphi_X(t_2)\ne 0 $ for $t_1, t_2\in \mathbb{R}$.
For $t_2\ne 0$, Eq.~\eqref{191016gh} yields an expression of the form $\frac{\partial \tilde{\varphi}}{\partial t_2}(t_1, t_2)= \tilde{\varphi}(t_1, t_2) \, Q(t_2)$, where $Q$ is a function which depends only on $t_2$. Since Eq.~\eqref{040416a} and Eq.~\eqref{030416c} hold, we have with the same function $Q$ the representation $\frac{\partial {\varphi_X} }{\partial t_2}(t_2)={\varphi_X}(t_2) \, Q(t_2)$. Hence:
\begin{align*}
\frac{\partial \Psi }{\partial t_2}(t_1, t_2)&= \frac{\tilde{\varphi}(t_1, t_2)Q(t_2)\varphi_M(t_1)\varphi_X(t_2)-\tilde{\varphi}(t_1, t_2)\varphi_X(t_2)Q(t_2)\varphi_M(t_1)}{\left[ \varphi_M(t_1)\varphi_X(t_2) \right]^2}=0.
\end{align*}
We find  $\frac{\partial \Psi }{\partial t_2}(t_1, 0)=0$ since $\frac{\partial\tilde{\varphi}}{\partial t_2}(t_1, 0)=\varphi'_X(0)=0$. A similar calculation shows that $
\frac{\partial \Psi }{\partial t_1}(t_1, t_2)=0$, hence $\nabla \Psi=0$ on $\mathbb{R}^2$.
We conclude that $\Psi$ is constant and since $\Psi(0, 0)=1$, we have that:
\[
\tilde{\varphi}(t_1, t_2)=\varphi_N(t_1) \varphi_X(t_2).
\]

\item We have finally that every subsequence which converges in law, has the same limit $ (I_1(f), X)^\top$ with $X$ independent of $I_1(f)$. Since the $\p$-completion of the $\sigma$-field generated by 
$\left\lbrace I_1(f) : f\in H , \Vert f\Vert_H=1 \right\rbrace$
is the $\sigma$-field $\mathcal{F}$ of the Brownian motion, an application of \cite[Lemma 2.3.]{2013arXiv1305.3899N}
 concludes the proof, see Section \ref{stableconvergence}.
\end{enumerate}
\end{proof}

The following Corollary \ref{081116fff} is a special case of the implication $(2)\Rightarrow (3)$ in \cite[Theorem 1.3]{rosinski2}.

\begin{corollary} \label{081116fff}
 Consider $p\ge 3$, $k_1, k_2\ge 0$   and $X$ as defined in Eq.~\eqref{261116aa}.
  Suppose that $\left\lbrace g_{n, p}\right\rbrace_n\subset H^{\odot p}$, % with $\sup_n\Vert g_{n, p}\Vert_{H^{\otimes p}}<\infty$.
define $P(x)=x^{1+1_{[a\ne 0]}}\prod_{i=1}^{k_1} (x-b_i)\prod_{j=1}^{k_2}(x-c_j)^2$ and suppose that, as $n\to \infty$:
%\[\kappa_r(F_n)\to k(r):=1_{[r=2]}a^2 +\sum_{i=1}^{k_1}2^{r-1}(r-1)!b_i^2  +\sum_{j=1}^{k_2}2^{r-1}(r-1)!\left( c_j^2    +\frac{r}{4}\,\frac{d_j^2}{c_j^2} \right) \textrm{ for } 2\le r\le \textrm{deg}(P), \]
%\[\sum_{r=1}^{\textrm{deg}(P)} \frac{P^{(r)}(0)}{r!2^{r-1}} \left( \Gamma_{r-1}(F_n) - \E[\Gamma_{r-1}(F_n)] \right)\stackrel{{L^1}}{\to} 0.
%\]
%
\begin{align}
&\kappa_r(F_n)\to k(r):=1_{[r=2]}a^2 +\sum_{i=1}^{k_1}2^{r-1}(r-1)!b_i^r \nonumber\\
&\rule{5mm}{0mm} +\sum_{j=1}^{k_2}2^{r-1}(r-1)!\left( c_j^r    +\frac{r c_j^{r-2} d_j^2}{4} \right),\quad \textrm{ for } 2\le r\le \textrm{deg}(P), \nonumber\\
&\sum_{r=1}^{\textrm{deg}(P)} \frac{P^{(r)}(0)}{r!2^{r-1}} \left( \Gamma_{r-1}(F_n) - \E[\Gamma_{r-1}(F_n)] \right)\stackrel{{L^1}}{\to} 0.\nonumber
\end{align}
If the following condition holds, as $n\to \infty$, for every $f\in H$ with $\Vert f \Vert_H=1$:
\begin{align}
\Vert g_{n, p} \otimes_1 f \Vert_{H^{\otimes (p-1)}} \to 0,\label{061116dd}
\end{align} 
then $I_p(g_{n,p})\stackrel{\textrm{st.}}{\to} X$, as $n\to \infty$, and $X$ is independent of the underlying Brownian motion.
\end{corollary}

\begin{proof}
We prove that Eq.~\eqref{191016c} and Eq.~\eqref{191016d} hold, Theorem \ref{191016ta} yields then the statement.
\begin{enumerate}[(1)]
\item We have, as $n\to \infty$:
\begin{align*}
\E\left[ \left| \langle DI_1(f), DI_p(g_{n,p})  \rangle_H\right|\right]^2 &\le \E\left[ 
\left| p I_{p-1}(g_{n,p} \otimes_1 f) \right|^2
\right] \\
&= p^2 (p-1)! \Vert g_{n,p}\otimes_1 f \Vert_{H^{\otimes (p-1)}}^2 \to 0.
\end{align*}
We have used the stochastic Fubini theorem for multiple Wiener integrals, and Eq.~\eqref{191016c} holds since $
\left|\exp\left( i(t_1 I_1(f) + t_2I_p(g_{n,p})) \right) \right|=1$.
\item To prove that Eq.~\eqref{191016d} holds, we notice that $\Gamma_b(I_p(g_{n,p}))$ has a representation as a finite linear combination of random variables of the following form with $R:=\sum_{i=1}^b r_i$:
\begin{align*}
I_{(b+1)p-2R} \left( 
\ldots ( g_{n, p} \tilde{\otimes}_{r_1} g_{n, p} ) \tilde{\otimes}_{r_2} g_{n, p}) \ldots ) \tilde{\otimes}_{r_b} g_{n, p}
 \right),
\end{align*}
see \citep[Proposition 2.1]{azmoodeh} for details and the set containing $(r_1,\ldots, r_b)$ in particular.  We can represent $\langle DI_1(f), -DL^{-1}\Gamma_b (I_p(g_{n,p}))\rangle_H$  as finite linear combination of terms having the following form:
\begin{align}
\langle f(\cdot) , I_{(b+1)p-2R-1} \left( 
\ldots ( g_{n, p} \tilde{\otimes}_{r_1} g_{n, p} )  \ldots ) \tilde{\otimes}_{r_b} g_{n, p}) (\, ,\cdot)
 \right) \rangle_H\nonumber \\
 = I_{(b+1)p-2R-1} \left(  \int_0^T (
\ldots ( g_{n, p} \tilde{\otimes}_{r_1} g_{n, p} )   \ldots ) \tilde{\otimes}_{r_b} g_{n, p}) (\, , t) f(t) dt
 \right). \label{061116v}
\end{align}
\begin{itemize}
\item Let $A:=(b+1)p-2(r_1+\ldots + r_b) >1$, then we can represent \[(\ldots(g_{n, p}\tilde{\otimes}_{r_1}g_{n, p}) \ldots ) \tilde{\otimes}_{r_b} g_{n, p} (x_1,\ldots, x_{A})\] as  finite linear combination of integrals of the following form:
%
%\item Let $A:=(b+1)p-2(r_1+\ldots + r_p)-1>0$, then $(\ldots(g_{n, p}\tilde{\otimes}_{r_1}g_{n, p}) \ldots ) \tilde{\otimes}_{r_b} g_{n, p} (x_1,\ldots, x_{A})$ has a representation as finite linear combination of integrals of the following form:
%
\begin{align}
\int_{[0,T]^{R}}g_{n,p}(Y_{1,l}, X_{1,l}) \ldots  g_{n,p}(Y_{b+1,l}, X_{b+1,l})dY_{1,l}\,\ldots\,dY_{b+1,l}, \label{060516TT1}
\end{align}
where $l$ is a summation index, $Y_{i,l}$ and $X_{j,l}$ represent collections of variables such that:
\begin{itemize}
\item $ \cup_{i=1}^{b+1} Y_{i,l}  $ contains  $R$ elements  and every element is in exactly two of the (non-empty) sets $Y_{1,l},\ldots, Y_{b+1,l}$,
\item $\left(\cup_{i=1}^{b+1} Y_{i,l} \right) \cap \left\lbrace x_1,\ldots,  x_{A}\right\rbrace=\emptyset$,
\item $\dot\cup_{i=1}^{b+1} X_{i,l}= \left\lbrace x_1,\ldots, x_{A}\right\rbrace$ .
\end{itemize}
We have then with $x_m\in \left\lbrace x_1, \ldots, x_A\right\rbrace$ and \cite[Lemma 2.3]{rosinski1}, as $n\to \infty$:
\begin{align*}
&\int_{[0,T ]^{A-1}}\left( \int_{[0,T]^{R+1}} \prod_{i=1}^{b+1} g_{n,p}(Y_{i,l}, X_{i,l}) f(x_m)dY_{1,l}\,\ldots\,dY_{b+1,l}dx_m\right)^2\\
&\rule{5mm}{0mm} \times \frac{dx_1 \ldots dx_A}{dx_m} \le \Vert g_{n,p} {\otimes}_1 f \Vert_{H^{\otimes (p-1)}}^2 \, \Vert g_{n,p} \Vert_{H^{\otimes p}}^{2b} \to 0. 
\end{align*}
\item Let $A:=(b+1)p-2(r_1+\ldots + r_b)=1$, then we can represent \[ (\ldots(g_{n, p}\tilde{\otimes}_{r_1}g_{n, p})\tilde{\otimes}_{r_2}g_{n, p}) \ldots ) \tilde{\otimes}_{r_b} g_{n, p}(x_0)  \]  as linear combination of integrals as in Eq.~\eqref{060516TT1} and all but one set $X_{i, l}$ is empty. Assume without loss of generality that $X_{1, l}$ is the non-empty set, containing the integration variable, say $x_0$. Hence, with \cite[Lemma 2.3]{rosinski1}, as $n\to \infty$:
\begin{align*}
&\left|\int_{[0,T]^{R+1}}f(x_0)g_{n,p}(Y_{1, l}, x_0)g_{n, p}(Y_{2,  l}) \ldots  g_{n, p}(Y_{b+1, l})dY_{1, l}\,\ldots\,dY_{b+1, l}dx_0\right|^2 \\
&=\left|\int_{[0,T]^{R}}(g_{n,p}\otimes_1 f)(Y_{1, l})
\prod_{i=2}^{b+1}  g_{n, p}(Y_{i, l}) dY_{1, l}\,\ldots\,dY_{b+1, l}\right|^2 \\
&\le \Vert g_{n, p} \otimes_1 f \Vert_{H^{\otimes (p-1)}}^2 \, \Vert g_{n, p}\Vert^{2b}_{H^{\otimes p}}\to 0.
\end{align*}
\end{itemize}
Combining these results, we find with $(y_1+\ldots + y_q)^2\le q \left( y_1^2 + \ldots + y_q^2 \right)$ that:
\begin{align*}
\E\left[ 
I_{(b+1)p-2R-1} \left(  \int_0^T (
\ldots ( g_{n, p} \tilde{\otimes}_{r_1} g_{n, p} ) \tilde{\otimes}_{r_2} g_{n, p}) \ldots ) \tilde{\otimes}_{r_b} g_{n, p} (\, , t) f(t) dt
 \right) ^2 \right] 
\end{align*}
converges to 0, as $n\to \infty$, hence:
\begin{align*}
\E\left[ \langle  DI_1(f), -DL^{-1}\Gamma_b({I_p(g_{n,p})}) \rangle_H^2 \right]\to 0,
\end{align*}
since $\langle DI_1(f), -DL^{-1}\Gamma_b({I_p(g_{n,p})}) \rangle_H$ can be represented as linear combination of integrals as in Eq.~\eqref{061116v}. Eq.~\eqref{191016d} follows since $\left| \exp(i(t_1I_1(f)+t_2I_p(g_{n,p})))\right|=1$. 
\end{enumerate}
\end{proof}

\begin{remark}
Eq.~\eqref{061116dd} holds if $\Vert g_{n,p} \tilde{\otimes}_{p-1} g_{n,p}\Vert_{H\otimes H} \to 0 $,  as $n\to \infty$. This follows directly from Lemma \ref{070316a}.
\end{remark}

We have the following converse of Theorem \ref{191016ta}.

\begin{theorem} \label{191016tb}
With the notations of Theorem \ref{191016ta}, consider $p\ge 3$, $k_1, k_2\ge 0$   and $X$ as defined in Eq.~\eqref{261116aa}. Suppose that $\left\lbrace g_{n, l}\right\rbrace_n\subset H^{\odot l}$   %  with $ \sup_n\Vert g_{n, l}\Vert_{H^{\otimes l}}<\infty$ 
for $1\le l \le p$, and: 
\begin{align*}
F_n=\sum_{l=1}^p I_l(g_{n, l}),%\quad \textrm{ with } \Vert g_{n,p }\Vert_{H^{\otimes p}}\ne 0.
\end{align*}
If condition \eqref{191016b} holds and
 $F_n\stackrel{\textrm{st.}}{\to} X$, as $n\to \infty$, where $X$ is independent of the underlying Brownian motion, then for every $f\in H$ with $\Vert f \Vert_H=1 $:
\begin{enumerate}[(1)]
\item Eq.~\eqref{191016c} and 
 Eq.~\eqref{191016d} hold,
 \item $\lim_n \left( \sum_{l=2}^{p-1} (l+1)! \langle g_{n,l-1} \tilde{\otimes}_{l-1} g_{n,l+1} , f\tilde{\otimes} f \rangle_{H\otimes H} + \E\left[ \langle DI_1(f) , DF_n \rangle_H^2  \right] \right)=0$.
\end{enumerate}
\end{theorem}

\begin{proof}
\begin{enumerate}[(1)]
\item Consider an arbitrary element $f\in H$ with $\Vert f \Vert_H=1$.
The stable convergence and the independence property imply that
$(I_1(f), F_n)^\top \stackrel{\textrm{Law}}{\to}(I_1(f), X)^\top$, as $n\to \infty$, where $X$ is independent of $I_1(f)$. Hence:
\[ \lim_n \varphi_n(t_1, t_2)=
\lim_n \E[\exp(i(t_1I_1(f)+t_2F_n))]= \varphi_{I_1(f)}(t_1) \varphi_X(t_2),
\] 
and Eq.~\eqref{191016ga}  {together with Eq.~\eqref{030416f}} implies that Eq.~\eqref{191016d} holds for every $(t_1, t_2)\in\mathbb{R}\times \mathbb{R}^*$. With Eq.~\eqref{191016gb}, we have that Eq~\eqref{191016c} holds  {for every $t_1, t_2 \ne 0$}. To see that Eq.~\eqref{191016c} holds for $t_1=0, t_2 \ne 0$, calculate the derivative with respect to $t_1$:
\begin{align*}
\frac{\partial \varphi_n}{\partial t_1}(t_1, t_2)&=i \E\left[\exp\left( i(t_1I_1(f)+t_2F_n )\right)I_1(f)\right]
\end{align*}
hence:
\begin{align*}
\frac{\partial \varphi_n}{\partial t_1}(0, t_2)&=
i\E\left[ \exp(it_2F_n) I_1(f) \right]=i\E\left[ \exp(it_2F_n) (-\delta D L^{-1} I_1(f)) \right]\\
&=
i \E\left[  \exp\left( i t_2 F_n \right) \langle it_2DF_n , DI_1(f)\rangle_H \right].
\end{align*}
By the continuous mapping theorem, we have: \[\left( \exp(it_2F_n) , I_1(f) \right)^\top\stackrel{\textrm{Law}}{\to} \left( \exp(it_2X) , I_1(f) \right)^\top,\] as $n\to \infty$, and $X$ is independent of $I_1(f)$. Since  $\sup_n \E\left[ \left| \exp(it_2F_n)I_1(f) \right|^2 \right] = \sup_n \E[1 \, I_1(f)^2]=1 <\infty$, we have:
\[
\lim_n \E \left[ \exp(it_2 F_n) I_1(f) \right]= \varphi_X(t_2) \E[I_1(f)]=0,
\]
thus, as $n\to \infty$:
\begin{align*}
0&= i\varphi_X(t_2)\E[I_1(f)]=\lim_n \frac{\partial \varphi_n}{\partial t_1}(0, t_2)\\
&=-\lim_n t_2 \E\left[  \exp\left( i t_2 F_n  \right) \langle D F_n ,DI_1(f)\rangle_H \right].
\end{align*}
 {
We prove that Eq.~\eqref{191016c} holds for $t_2=0, t_1\ne 0$:
\begin{align*}
&\E[\exp(it_1 I_1(f))\, \langle DI_1(f), DF_n\rangle_H]\\
&= (it_1)^{-1}\E[ \langle it_1 f \exp(it_1 I_1(f)), DF_n\rangle_H ]\\
&=(it_1)^{-1} \E[\delta(it_1 f \exp(it_1 I_1(f))) \, F_n]\\
&= (it_1)^{-1} \E[(it_1 I_1(f) + t_1^2)\, \exp(it_1 I_1(f)) \, F_n]\to 0,
\end{align*}
as $n\to \infty$.We have used Eq.~\eqref{int_parts_0} in the last step, the convergence follows as above and Eq.~\eqref{191016c} holds for $t_2=0, t_1\ne 0$. In the case $t_1=t_2=0$, we find with the independence property and Eq.~\eqref{int_parts} that Eq.~\eqref{191016c} holds. 
}
We see similarly that:
\begin{align}
0&= \lim_n \frac{\partial \varphi_n}{\partial t_2}(t_1, 0)=-\lim_n t_1 \E\left[  \exp\left( i t_1 I_1(f ) \right) \langle DI_1(f) ,-DL^{-1} F_n\rangle_H \right]. \label{071116a}
\end{align}

On the other hand, Eq.~\eqref{191016d} yields for $t_2=0$ the following condition:
\begin{align}
 \frac{it_1P^{k+1+1_{[a\ne 0]}}(0)}{(k+1+1_{[a\ne 0]})! 2^{k+1_{[ a\ne 0]}}}\E\left[ 
\exp\left( it_1 I_1(f)\right)\, \langle DI_1(f), -DL^{-1}F_n\rangle_H 
\right] \to 0,\label{071116b}
\end{align}
as $n\to \infty$. Since $P^{k+1+1_{[\alpha_0\ne 0]}}(0)\ne 0$, the convergence in Eq.~\eqref{071116b} follows clearly from Eq.~\eqref{071116a}. We conclude that Eq.~\eqref{191016c} and \eqref{191016d} hold for every $(t_1, t_2) \in \mathbb{R}\times \mathbb{R}$.

\item Consider an arbitrary element $f\in H$ with $\Vert f\Vert_H^2=1$. 
Since $(I_1(f), F_n)^\top\stackrel{\textrm{Law}}{\to} (I_1(f), X)^\top$, as $n\to \infty$,  where $I_1(f)\sim \mathcal{N}(0,1)$ is independent of $X$, we have that  $(I_1(f)^2, F_n^2)^\top\stackrel{\textrm{Law}}{\to} (I_1(f)^2,  X^2)$, as $n\to \infty$, by the continuous mapping theorem, hence $I_1(f)^2 F_n^2\stackrel{\textrm{Law}}{\to}I_1(f)^2X^2$, as $n\to \infty$. We have that:
\begin{align}
\sup_n \E\left[ \left(  F_n^2 I_1(f)^2\right)^2 \right]<\infty. \label{071116c}
\end{align}
 {this follows from the hypercontractivity property and the Cauchy-Schwarz inequality. Hence:}
%then we can infer that:
\begin{align}
0&=   \lim_n \left(  \E[F_n^2 I_1(f)^2]  - \E[F_n^2] \E[ I_1(f)^2] \right), \label{071116d} 
\end{align}
since $I_1(f)$ is independent of $X$. \begin{comment}
We prove now Eq.~\eqref{071116c} using the hypercontractivity property of random variables   living in a  finite sum of Wiener chaoses. 
\begin{align*}
&\sup_n \E \left[ \left( F_n^2 I_1(f)^2\right)^2 \right] = \sup_n \E \left[ \left( F_n I_1(f)\right)^4 \right] \le \sup_n c(p) \E[(F_n I_1(f))^2]^2 \\
&\le \sup_n c(p) \E \left[ \left( \sum_{l=1}^p I_{l+1}(g_{n, l}\tilde{\otimes}f) + \sum_{l=1}^p l I_{l-1} (g_{n, l} \tilde{\otimes}_1 f) \right)^2 \right]^2 \\
&\le   c(p)\sup_n \left[ 2\sum_{l=1}^p (l+1)! \Vert g_{n, l}\tilde{\otimes} f \Vert_{H^{\otimes (l+1)}}^2 + 2\sum_{l=1}^p l^2 (l-1)! \Vert g_{n, l} \tilde{\otimes}_1 f \Vert_{H^{\otimes (l-1)}}^2 \right]^2\\
&\le   c(p)\sup_n \left[ 2\sum_{l=1}^p (l+1)! \Vert g_{n, l} \Vert_{H^{\otimes (l)}}^2 +2 \sum_{l=1}^p l^2 (l-1)! \Vert g_{n, l} \Vert_{H^{\otimes (l)}}^2 \right]^2 < \infty.
\end{align*}
We have used Lemma \ref{070316a} and $\sup_l \sup_n \Vert g_{n, l} \Vert_{H^{\otimes l}} <\infty $.\end{comment}
 With Eq.~\eqref{071116d}:
\begin{align}
0&= \lim_n \left(  \E[F_n^2 I_1(f)^2 ] - \E[F_n^2]\E[I_1(f)^2] \right)\nonumber\\
&= \lim_n \left( 
\E\left[ F_n^2 I_1(f) (-\delta D L^{-1} I_1(f)) \right] -\E[F_n^2]
\right)\nonumber\\
&= \lim_n \left( \E \left[ \langle D(F_n^2 I_1(f)) , -DL^{-1}I_1(f) \rangle_H  \right]-\E[F_n^2] \right)\nonumber\\
&= \lim_n \left( 
2\E \left[
F_n I_1(f) \langle DF_n , DI_1(f)\rangle_H\right] + \E \left[ F_n^2 \Vert DI_1(f)\Vert_H^2 \right] - \E\left[ F_n^2
\right]
\right)\nonumber\\
&= 2\lim_n  
\E \left[
F_n I_1(f) \langle DF_n , DI_1(f)\rangle_H\right], \label{261116ab}
 \end{align}
and with the stochastic Fubini theorem for multiple Wiener integrals:
\begin{align*}
\langle DI_1(f) , DF_n\rangle_H&= \sum_{l=1}^p l I_{l-1} (g_{n,l}\tilde{\otimes}_1 f),\\
F_n I_1(f)&= \sum_{l=1}^p I_{l+1}(g_{n,l}\tilde{\otimes}f) + \sum_{l=1}^p l I_{l-1}(g_{n,l}\tilde{\otimes}_1 f)\\
&=\sum_{l=1}^p I_{l+1}(g_{n,l}\tilde{\otimes}f) + \langle DI_1(f),  DF_n\rangle_H.
\end{align*}
Hence with Eq.~\eqref{261116ab}: {
\begin{align*}
0&= \lim_n \left( 
\E\left[ \sum_{l=1}^p I_{l+1}(g_{n, l}\tilde{\otimes }f)  \, \sum_{l=1}^{p}l I_{l-1} (g_{n, l}\tilde{\otimes}_1 f) \right]
+\E[\langle DI_1(f) , DF_n \rangle_H^2]
 \right) \\
 &= \lim_n \left( 
\E\left[ \sum_{l=2}^{p+1} I_{l}(g_{n, l-1}\tilde{\otimes }f)  \, \sum_{l=0}^{p-1}(l+1) I_{l} (g_{n, l+1}\tilde{\otimes}_1 f) \right]\right.\\
&\rule{5mm}{0mm}\left.+\E[\langle DI_1(f) , DF_n \rangle_H^2]
 \right) \\
&= \lim_n \left( 
\sum_{l=2}^{p-1}(l+1)!\langle g_{n, l-1} \tilde{\otimes} f , g_{n, l+1}\tilde{\otimes}_1 f\rangle_{H^{\otimes l}} + \E[\langle DI_1(f) , DF_n \rangle_H^2]
 \right) \\
&= \lim_n \left( 
\sum_{l=2}^{p-1}(l+1)!\langle g_{n, l-1} \tilde{\otimes}_{l-1} g_{n, l+1}, f\tilde{\otimes} f\rangle_{H\otimes H} + \E[\langle DI_1(f) , DF_n \rangle_H^2]
 \right).
\end{align*}}
The last equality can be checked directly using the definition of contractions.
\end{enumerate}
\end{proof}

\begin{corollary}
Let the notations and assumptions of Theorem \ref{191016tb} prevail and suppose that $F_n=\sum_{l=1}^p I_l(g_{n, l})\stackrel{\textrm{st.}}{\to} X$, as $n\to \infty$, where  $X$ is independent of the underlying Brownian motion. Consider $f\in H$ arbitrary.
\begin{enumerate}
\item If $F_n=I_p(g_{n, p})$, then we have  $\lim_n \Vert g_{n,p}\tilde{\otimes}_1 f \Vert_{H^{\otimes (p-1)}}= 0$.%, as $n\to \infty$.%, for every $f\in H$.
\item If
$\lim_n \Vert g_{n, l+1}\tilde{\otimes}_{l-1} g_{n, l-1}\Vert_{H\otimes H}= 0$, %as $n\to \infty$, 
for every $2\le l\le p-1$,
then  we have $
\lim_n \Vert g_{n,l} \tilde{\otimes}_1 f\Vert_{H^{\otimes (l-1)}}= 0 
$, for every $1\le l \le p$.
\item If $F_n=I_{p-1}(g_{n, p-1})+ I_p(g_{n, p})$, then we have $\lim_n \Vert g_{n, p-1}\tilde{\otimes}_1 f \Vert_{H^{\otimes (p-2)}}=\lim_n \Vert g_{n, p}\tilde{\otimes}_1 f \Vert_{H^{\otimes (p-1)}}=0$.
\end{enumerate} 
\end{corollary}

We can recover \cite[Proposition 4.2. and Remark 4.3.]{nourdin2009b}.

\begin{proposition}
Consider $p\ge 4$ even and  a sequence $\left\lbrace g_{n, p}\right\rbrace_n \subset H^{\otimes p}$ such that %$\sup_n \Vert g_{n, p}\Vert_{H^{\otimes p}}<\infty$ and 
 $I_p(g_{n, p})\stackrel{\textrm{Law}}{\to}  N^2-1$, as $n\to \infty$, where $N$ is a standard normal variable. Then $I_p(g_{n, p})\stackrel{\textrm{st.}}{\to} N^2-1$, as $n\to \infty$, and $N^2-1$ is independent of the underlying Brownian motion.
\end{proposition}

\begin{proof}
\cite[Theorem 1.2.]{nourdin2009b} implies with $P(x)=x(x-1)$, as $n\to \infty$:
\begin{align}
&\sum_{r=1}^2 \frac{P^{(r)}(0)}{r!2^{r-1}}\left(  \Gamma_{r-1}(I_{g_{n, p}})- \E[\Gamma_{r-1}(I_{g_{n, p}})]  \right)\nonumber\\
 &= I_p(g_{n, p}) - \frac{1}{2} \Gamma_1(I_p(g_{n, p}))+ \frac{1}{2}\E[\Gamma_1(I_p(g_{n, p}))]\nonumber\\
&= I_p(g_{n, p}) - \frac{1}{2} \Gamma_1(I_p(g_{n, p}))+ \frac{1}{2}\E[I_p(g_{n, p})^2]\stackrel{{L^2}}{\to} 0. \nonumber %\label{191016g}. 
\end{align}
Moreover the convergence in law implies that $\Vert g_{n, p}{\otimes}_{p-1}g_{n, p}\Vert_{H\otimes H}\to 0$, as $n\to \infty$.
Notice that in condition \eqref{191016d} in Theorem \ref{191016ta} we must have $(\alpha, \beta, r)=(1, 0, 2)$ and condition \eqref{191016d}  { is then satisfied if condition \eqref{191016c} holds}. 
We check now condition \eqref{191016c}. 
\begin{align*}
&\left|
\E\left[ \exp\left( i(t_1I_1(f)+t_2I_p(g_{n, p})) \right)\langle DI_1(f), DI_p(g_{n, p})\rangle_H \right]\right|\\
&\le \E\left[ \langle DI_1(f), DI_p(g_{n, p})\rangle_H^2 \right]^{1/2}\\
&= \E[p^2 I_{p-1}(g_{n, p}\otimes_1 f )^2] = p^2 (p-1)! \Vert g_{n, p}\otimes_1 f\Vert_{H^{\otimes (p-1)}}^2\\ &\le  p^2 (p-1)! \Vert g_{n, p}{\otimes }_{p-1}g_{n, p}\Vert_{H^{\otimes 2}} \Vert f\Vert_H^2\to 0 ,
\end{align*}
as $n\to \infty$.
\end{proof}

\section{Appendix}

\begin{definition} \label{040316d}
\begin{enumerate}[(1)]
\item Consider $k_1\ge 0$. If $k_1>0$, we suppose that $b_j\ne 0$ for every $j=1,\ldots, k_1$.  {We set}:
\[
T_j = \begin{cases}  1 & \textnormal{\textit{for }} j=0,\\ \displaystyle{\sum_{1\le i_1 < \ldots < i_j\le k_1} b_{i_1}}\times \ldots \times b_{i_j}&\textnormal{\textit{for} } 1 \le j \le k_1,  \\
0&\textnormal{\textit{for} } j\notin \left\lbrace 0 , 1 , \ldots, k_1\right\rbrace , 
\end{cases}
\]
and  { for $l=1,\ldots, k_1$}:
\[
T_j^{(l)} = \begin{cases} 1 & \textnormal{\textit{for }} j=0,\\ \displaystyle{\sum_{1\le i_1 < \ldots < i_j\le k_1 \atop i_1, \ldots, i_j \ne l} b_{i_1}\times \ldots \times b_{i_j}}&\textnormal{\textit{for} } 1\le j \le  k_1-1  ,\\
0&\textnormal{\textit{for} } j\notin \left\lbrace 0 , 1 ,\ldots, k_1-1\right\rbrace .
\end{cases}
\]
If $k_1=0$, we set $b_j=0$ for every $j$, $T_0=1$ and $T_j=T_j^{(l)}=0$ for all other values of $j$ and $l$. 

\item Consider $k_2\ge 0$. If $k_2>0$, we suppose that $c_j d_j\ne 0$ for every $j=1, \ldots, k_2$.   {We set}:
\[
S_j = \begin{cases}  1 &\textnormal{\textit{for} }j=0, \\ \displaystyle{\sum_{1\le i_1 < \ldots < i_j\le k_2} c_{i_1}}\times \ldots \times c_{i_j}&\textnormal{\textit{for} } 1\le j \le k_2, \\
0&\textnormal{\textit{for} } j\notin \left\lbrace 0 , 1 , \ldots, k_2\right\rbrace  ;
\end{cases}
\]
and  { for $l=1, \ldots, k_2$}:
\[
S_j^{(l)} = \begin{cases} 1 &\textnormal{\textit{for }} j=0, \\ \displaystyle{\sum_{1\le i_1 < \ldots < i_j\le k_2 \atop i_1, \ldots, i_j \ne l} c_{i_1}\times \ldots \times c_{i_j}}&\textnormal{\textit{for} } 1 \le j \le k_2-1  ,\\
0&\textnormal{\textit{for} } j\notin \left\lbrace 0 , 1 , \ldots, k_2-1\right\rbrace .
\end{cases}
\]
If $k_2=0$, we set $c_j=d_j=0$ for every $j$, $S_0=1$ and $S_j=S_j^{(l)}=0$ for all other values of $j$ and $l$.
\end{enumerate}
\end{definition}

 \newpage

\begin{proposition}\label{291216d}
\textbf{Proof of Eq.~\eqref{040416dd}. }\end{proposition}
\begin{proof}
We prove that Eq.~\eqref{040416dd} holds for $x\ne 0$. Suppose that $x\varphi_Y(x)\ne 0 $   and divide Eq.~\eqref{040416dd} by $i\varphi_Y(x)$:
\begin{align}
&\sum_{l=1}^{k} (2ix)^{k+1}   (ix)^{-l} \sum_{r=l}^{k} \frac{P^{(r+2)}(0)}{(r+2)!2^{r+1}}  \frac{\kappa_{r-l+2}(X)}{(r-l+1)!} + (2ix)^{k+1}\frac{ (-1)^{k_1}a^2}{2}\prod_{j=1}^{k_2} c_j^2\, \prod_{j=1}^{k_1} b_j \nonumber \\
&=  ix a^2 G_1(x)G_2(x) + 2ix G_2(x)  \sum_{l=1}^{k_1} b_l^2 \prod_{j\ne l}(1-2ixb_j)\nonumber\\
&\rule{5mm}{0mm}+   G_1(x) \sum_{l=1}^{k_2} \left(\prod_{j\ne l} (1-2ixc_j)^2 \right) \left[ ix \Delta_l (1-2ixc_l) + (ix)^2 c_l\Delta_l  - 2ix c_l^2 \right]. \label{040416d}
 \end{align}
We determine an alternative representation for $P(x)$ and calculate $P^{(l)}(0)/l!$. We have with $k=2k_2+k_1$
\begin{align*}
P(x)&= x^2 \prod_{j=1}^{k_1}(x-b_j) \prod_{j=1}^{k_2}(x-c_j)^2 = x^2 x^k \prod_{j=1}^{k_1}(1-b_j/x) \prod_{j=1}^{k_2}(1-c_j/x)^2\\
&=x^{k+2} \sum_{j_1=0}^{k_1}\,\sum_{j_2,j_3=0}^{k_2} (-1)^{j_1} T_{j_1} (-1)^{j_2} S_{j_2} (-1)^{j_3} S_{j_3}x^{-j_1-j_2-j_3}\\
&= x^{k+2} \sum_{p=0}^k x^{-p} (-1)^p \sum_{j_1+j_2+j_3=p} T_{j_1}S_{j_2}S_{j_3} \\
&= \sum_{l=2}^{k+2} x^l (-1)^{k+2-l} \left( \sum_{j_1+j_2+j_3=k+2-l} T_{j_1}S_{j_2} S_{j_3} \right).
\end{align*}
If for instance $k_1=0$, we have $T_j= 1_{[j=0]}$ and  the formula above holds also for the cases $k_1=0$, $k_2=0$ and $k_1=k_2=0$. Hence:
\begin{align*}
\frac{P^{(l)}(0)}{l!} = (-1)^{k+2-l} \sum_{j_1+j_2+j_3=k+2-l} T_{j_1}S_{j_2} S_{j_3}.
\end{align*}
We compare now the terms in $x$, on the left-hand side of Eq.~\eqref{040416d} we have:
\begin{align*}
2^{k+1}(ix)^{k+1-k} \frac{P^{(k+2)}(0)}{(k+2)!2^{k+1}} \kappa_2(X)=  ix \left(a^2+ 2\sum_{j=1}^{k_1}b_j^2 + \sum_{j=1}^{k_2}(2c_j^2+d_j^2) \right),
\end{align*}
and for the right-hand side of Eq.~\eqref{040416d}:
\begin{align*}
ixa^2 + 2ix \sum_{l=1}^{k_1}b_l^2 +  \sum_{l=1}^{k_2} [ix\Delta_l - 2ixc_l^2],
\end{align*}
the desired equality for the terms in $x$ follows.
The term in $x^{k+1}$ on the left-hand side of Eq.~\eqref{040416d} is 
$  2^k(ix)^{k+1} (-1)^{k_1}a^2 \prod_{j=1}^{k_2} c_j^2\prod_{j=1}^{k_1}b_j$,
 and on the right-hand side of Eq.~\eqref{040416d} we have:
\begin{align*}
 ix a^2 \prod_{j=1}^{k_2}(-2xic_j)^2 \prod_{j=1}^{k_1}(-2ixb_j) &= (ix)^{k+1}   (-2)^k a^2\prod_{j=1}^{k_2} c_j^2 \prod_{j=1}^{k_1}b_j\\
 &=(ix)^{k+1}   (-1)^{k_1} 2^k a^2\prod_{j=1}^{k_2} c_j^2 \prod_{j=1}^{k_1}b_j.
\end{align*}
We have for the term in $x^m$ and $1<m<k+1$ on the left-hand side of Eq.~\eqref{040416d}:
\begin{align} 
&2^{k+1} (ix)^{k+1-(k+1-m)}\left(
\frac{P^{(k+1-m+2)}(0)\kappa_2(X)}{(k+1-m+2)! 2^{k-m+2}}  + \ldots + \frac{P^{(k+2)}(0)}{(k+2)!2^{k+1}} \frac{\kappa_{m+1}(X)}{m!} \right)\nonumber\\
&=2^{k+1}(ix)^m \left[ (-1)^{m-1} \sum_{j_1+j_2+j_3\atop=m-1}T_{j_1}S_{j_2}S_{j_3} \, \frac{2^{m-1}}{2^{k+1}}\kappa_2(X) \right.\nonumber\\
&\rule{5mm}{0mm} \left.+ \ldots +  (-1)^0 \sum_{j_1+j_2+j_3\atop=0}T_{j_1}S_{j_2}S_{j_3} \, \frac{2^0}{2^{k+1}}\frac{\kappa_{m+1}(X)}{m!} \right]\nonumber\\
&=  (ix)^m \left[ (-2)^{m-1}\sum_{j_1+j_2+j_3\atop=m-1}T_{j_1}S_{j_2}S_{j_3} \, \left( a^2 + 2\sum_{j=1}^{k_2}c_j^2 +2\sum_{j=1}^{k_1}b_j^2 + \sum_{j=1}^{k_2}d_j^2 \right) \right.\nonumber\\
&\rule{5mm}{0mm}+ \ldots + \left.(-2)^{0}\sum_{j_1+j_2+j_3\atop=0}T_{j_1}S_{j_2}S_{j_3} \, \left(  2^m\sum_{j=1}^{k_2}c_j^{m+1} +2^m\sum_{j=1}^{k_1}b_j^{m+1} \right.\right.\nonumber \\
&\rule{5mm}{0mm}+ \left.\left. 2^{m-2}(m+1)\sum_{j=1}^{k_2}c_j^{m-1}d_j^2\right] \right) . \label{070416*}
\end{align}We use  the relations $S_j=S_j^{(l)}+ c_l S_{j-1}^{(l)}$  and $T_j=T_j^{(l)}+ b_l T_{j-1}^{(l)}$ if $k_1, k_2>0$,
and set $Z_j^{(l)}:= \sum_{j_2+j_3=j}S_{j_2}^{(l)} S_{j_3}^{(l)}$.
We consider the powers of $b_j$ in Eq.~\eqref{070416*} and suppose  that $k_1\ne 0$ for the next calculation.

\begin{align*}
&\sum_{l=1}^{k_1}\left[ (-2)^{m-1}  \left( \sum_{j_1+j_2+j_3\atop=m-1} T_{j_1}S_{j_2}S_{j_3} \right)2 b_l^2 + \ldots + (-2)^0 \left( \sum_{j_1+j_2+j_3\atop=0} T_{j_1}S_{j_2}S_{j_3}\right) 2^m b_l^{m+1 } \right]\\
&=2^m\sum_{l=1}^{k_1}\left[ (-1)^{m-1} \left( \sum_{j_1+j_2+j_3\atop=m-1} (T_{j_1}^{(l)} + b_l T_{j_1-1}^{(l)})S_{j_2}S_{j_3} \right) b_l^2 \right. \\
&\rule{5mm}{0mm} \left.  +  \ldots +   \left( \sum_{j_1+j_2+j_3\atop=0} (T_{j_1}^{(l)}+ b_l T_{j_1-1}^{(l)})S_{j_2}S_{j_3}\right)  b_l^{m+1 } \right] \\
&=2^m \sum_{l=1}^{k_1}\left[    (-1)^{m-1}\left( \sum_{j_1+j_2+j_3\atop=m-1}T_{j_1}^{(l)} S_{j_2}S_{j_3} \right)b_l^2 +   (-1)^{m-1}\left( \sum_{j_1+j_2+j_3\atop=m-2}T_{j_1}^{(l)} S_{j_2}S_{j_3} \right)b_l^3  \right.\\
&\rule{5mm}{0mm}+    (-1)^{m-2} \left( \sum_{j_1+j_2+j_3\atop=m-2}T_{j_1}^{(l)} S_{j_2}S_{j_3} \right)b_l^3 +   (-1)^{m-2} \left( \sum_{j_1+j_2+j_3\atop=m-3}T_{j_1}^{(l)} S_{j_2}S_{j_3} \right)b_l^4 \\
&\rule{5mm}{0mm}+ \left.\ldots +   (-1)^0 \left( \sum_{j_1+j_2+j_3\atop=0}T_{j_1}^{(l)} S_{j_2}S_{j_3} \right) b_l^{m+1} +   (-1)^0 \left( \sum_{j_1+j_2+j_3\atop=-1}T_{j_1}^{(l)} S_{j_2}S_{j_3} \right)b_l^{m+2} \right]\\
&=  2^m(-1)^{m-1} \sum_{l=1}^{k_1}\left[\left( \sum_{j_1+j_2+j_3\atop=m-1}T_{j_1}^{(l)} S_{j_2}S_{j_3} \right) b_l^2\right],
\end{align*}
hence:
\begin{align*}
&\sum_{l=1}^{k_1}\left[ (-2)^{m-1} \left( \sum_{j_1+j_2+j_3\atop =m-1} T_{j_1}S_{j_2}S_{j_3} \right)2 b_l^2 + \ldots +   \left( \sum_{j_1+j_2+j_3\atop=0} T_{j_1}S_{j_2}S_{j_3}\right) 2^m b_l^{m+1 } \right]\\
&=(-2)^{m-1} \sum_{l=1}^{k_1}\left[\left( \sum_{j_1+j_2+j_3=m-1}T_{j_1}^{(l)} S_{j_2}S_{j_3} \right) 2b_l^2\right],
\end{align*}
Notice that, with our settings, the previous equality holds also if $k_1=0$ or $k_2=0$
We consider now the powers of $c_j$ in Eq.~\eqref{070416*} and suppose that $k_2>0$ for the next calculation.
\begin{align}
&\sum_{j=1}^{k_2}   \left[ (-2)^{m-1} \left( \sum_{j_1+j_2+j_3\atop =m-1}T_{j_1}S_{j_2}S_{j_3}  \right) 2c_j^2 \right. \nonumber \\
&\rule{5mm}{0mm}+ \left. \ldots+ (-2)^0 \left( \sum_{j_1+j_2+j_3\atop=0}T_{j_1}S_{j_2}S_{j_3}\right) 2^m  c_j^{m+1} \right]\label{251116a}\\
&=2^m \sum_{j=1}^{k_2} \left[ 
(-1)^{m-1} \left( \sum_{j_1+j_2+j_3\atop =m-1}T_{j_1}S_{j_2}S_{j_3}\right) c_j^2 + \ldots +   \left(\sum_{j_1+j_2+j_3\atop =0}T_{j_1}S_{j_2}S_{j_3}\right) c_j^{m+1}
 \right]\nonumber\\  
 &= 2^m \sum_{j=1}^{k_2}  \sum_{j_1=0}^{k_1} T_{j_1} \sum_{r=0}^{m-1} (-1)^{m-1-r} \left( \sum_{j_2+j_3\atop =m-1-r-j_1 }S_{j_2} S_{j_3}\right)c_j^{r+2}\nonumber\\
 &= 2^m  \sum_{j=1}^{k_2} \sum_{j_1=0}^{k_1}T_{j_1} \sum_{r=0}^{m-1} (-1)^{m-1-r} \left[  \sum_{j_2+j_3\atop =m-1-r-j_1 }(S_{j_2}^{(j)}+c_j S_{j_2-1}^{(j)} )\,( S_{j_3}^{(j)}+c_j S_{j_3-1}^{(j)}) \right]c_j^{r+2}\nonumber\\
 &= 2^m \sum_{j=1}^{k_2}\sum_{j_1=0}^{k_1} T_{j_1}  \sum_{r=0}^{m-1} (-1)^{m-1-r}\left[   Z_{m-1-r-j_1}^{(j)} c_j^{r+2} + 2   Z_{m-2-r-j_1}^{(j)} c_j^{r+3}\right. \nonumber\\
 &\rule{5mm}{0mm} \left.+   Z_{m-3-r-j_1}^{(j)} c_j^{r+4}   \right]\nonumber\\
 &= 2^m \sum_{j=1}^{k_2} \sum_{j_1=0}^{k_1}T_{j_1}(-1)^{m-1} \left[   Z_{m-1-j_1}^{(j)} c_j^2 + 2   Z_{m-2-j_1}^{(j)} c_j^3 +  Z_{m-3-j_1}^{(j)}c_j^4 \right]\nonumber\\
&\rule{5mm}{0mm}+ 2^m \sum_{j=1}^{k_2}\sum_{j_1=0}^{k_1} T_{j_1}(-1)^{m-2}\left[  Z_{m-2-j_1}^{(j)} c_j^3 + 2   Z_{m-3-j_1}^{(j)} c_j^4 +   Z_{m-4-j_1}^{(j)}c_j^5 \right] \nonumber\\
&\rule{5mm}{0mm}+\ldots +2^m \sum_{j=1}^{k_2}\sum_{j_1=0}^{k_1} T_{j_1} (-1)^1\left[   Z_{1-j_1}^{(j)} c_j^{m} + 2   Z_{-j_1}^{(j)} c_j^{m+1} +   Z_{-1-j_1}^{(j)}c_j^{m+2} \right] \nonumber\\
&\rule{5mm}{0mm}+ 2^m \sum_{j=1}^{k_2} c_j^{m+1} .\nonumber
\end{align}
Hence the expression in Eq.~\eqref{251116a} equals:
\begin{align*}
&2^m (-1)^{m-1} \sum_{j=1}^{k_2}\left( \sum_{j_1+j_2+j_3\atop=m-1}T_{j_1}S_{j_2}^{(j)}S_{j_3}^{(j)} \right) c_j^2+ 2^m \sum_{j=1}^{k_2}  \left( \sum_{j_1+j_2+j_3\atop=m-2}T_{j_1}S_{j_2}^{(j)}S_{j_3}^{(j)} \right) \\
&\rule{5mm}{0mm}\times c_j^3 \left[ 2(-1)^{m-1} + (-1)^{m-2} \right] + 2^m \sum_{r=2}^{m-1}\sum_{j=1}^{k_2}\left( \sum_{j_1+j_2+j_3\atop=m-1-r}T_{j_1}S_{j_2}^{(j)}S_{j_3}^{(j)} \right)\\
&\rule{5mm}{0mm}\times\left[ (-1)^{m-1-r} + 2(-1)^{m-1-(r-1)}  +(-1)^{m+1-r} \right]c_j^{r+2}\\
&=(-2)^{m-1}\sum_{j=1}^{k_2}\left(\sum_{j_1+j_2+j_3\atop =m-1} T_{j_1}S_{j_2}^{(j)}S_{j_3}^{(j)}\right)2c_j^2  -(-2)^{m-2}\sum_{j=1}^{k_2}\left(\sum_{j_1+j_2+j_3\atop =m-2} T_{j_1}S_{j_2}^{(j)}S_{j_3}^{(j)}\right)4c_j^3.
\end{align*}
The resulting equality  holds also if $k_1=0$ or $k_2=0$. In the latter case, we have $S_{j_2}^{(l)}=c_j=d_j=0$ and an empty sum in \eqref{251116a} equals 0.
For the terms of the form $c_l^{r}d_l^{2}$ in Eq.~\eqref{070416*}, we can make a similar calculation. For the calculation below, we suppose that $k_2\ne 0$.
\begin{align*}
&\sum_{j=1}^{k_2} \left[ (-2)^{m-1} \left( \sum_{j_1+j_2+j_3\atop = m-1}T_{j_1}S_{j_2}S_{j_3} \right) d_j^2 + (-2)^{m-2} \left( \sum_{j_1+j_2+j_3\atop =m-2} T_{j_1}S_{j_2}S_{j_3}  \right) 3 d_j^2c_j \right.\\
&\rule{5mm}{0mm}+ \left.\ldots + (-2)^0 \left(\sum_{j_1+j_2+j_3\atop = 0} T_{j_1}S_{j_2} S_{j_3} \right)2^{m-2}(m+1)d_j^2c_j^{m-1}\right]\\
&=\sum_{j=1}^{k_2} 2^{m-1}(-1)^{m-1} \left(\sum_{j_1+j_2+j_3\atop=m-1}T_{j_1}S_{j_2}^{(j)}S_{j_3}^{(j)}\right) d_j^2 \\
&\rule{5mm}{0mm}- \sum_{j=1}^{k_2} 2^{m-2} (-1)^{m-2}  
\left(\sum_{j_1+j_2+j_3\atop=m-2}T_{j_1}S_{j_2}^{(j)}S_{j_3}^{(j)}\right)
  c_j d_j^2 \\
  &\rule{5mm}{0mm}+ 2^{m-2}\sum_{j=1}^{k_2} 
\left(\sum_{j_1+j_2+j_3\atop=-1}T_{j_1}S_{j_2}^{(j)}S_{j_3}^{(j)}\right)  
  (m+2) c_j^m d_j^2+ \sum_{j=1}^{k_2} \sum_{l=2}^{m-1} \left[ (-1)^{m-l+1}l c_j^l d_j^2  \right.\\
  &\rule{5mm}{0mm}+ 2(l+1)(-1)^{m-l}c_j^l d_j^2 + \left.(l+2)(-1)^{m-l-1}c_j^l d_j^2 \right]\, 
\left(\sum_{j_1+j_2+j_3\atop =m-l-1}T_{j_1}S_{j_2}^{(j)}S_{j_3}^{(j)}\right) \, 2^{m-2} \\
&=(-2)^{m-1}  \sum_{j=1}^{k_2} 
\left(\sum_{j_1+j_2+j_3\atop=m-1}T_{j_1}S_{j_2}^{(j)}S_{j_3}^{(j)}\right)
d_j^2 - (-2)^{m-2}\sum_{j=1}^{k_2} 
\left(\sum_{j_1+j_2+j_3\atop=m-2}T_{j_1}S_{j_2}^{(j)}S_{j_3}^{(j)}\right)
 c_jd_j^2.
\end{align*}

%\begin{align*}
%&\sum_{j=1}^{k_2} \left[ (-2)^{m-1} \left( \sum_{j_1+j_2+j_3\atop = m-1}T_{j_1}S_{j_2}S_{j_3} \right) d_j^2 + (-2)^{m-2} \left( \sum_{j_1+j_2+j_3\atop =m-2} T_{j_1}S_{j_2}S_{j_3}  \right) 3 d_j^2c_j \right.\\
%&\rule{5mm}{0mm}+ \left.\ldots + (-2)^0 \left(\sum_{j_1+j_2+j_3\atop = 0} T_{j_1}S_{j_2} S_{j_3} \right)2^{m-2}(m+1)d_j^2c_j^{m-1}\right]\\
%&=\sum_{j=1}^{k_2} 2^{m-1}(-1)^{m-1} \left(\sum_{j_1+j_2+j_3\atop=m-1}T_{j_1}S_{j_2}^{(j)}S_{j_3}^{(j)}\right) d_j^2 \\
%&\rule{5mm}{0mm}- \sum_{j=1}^{k_2} 2^{m-2} (-1)^{m-2}  
%\left(\sum_{j_1+j_2+j_3\atop=m-2}T_{j_1}S_{j_2}^{(j)}S_{j_3}^{(j)}\right)
%  c_j d_j^2 \\
%  &\rule{5mm}{0mm}+ 2^{m-2}\sum_{j=1}^{k_2} 
%\left(\sum_{j_1+j_2+j_3\atop=-1}T_{j_1}S_{j_2}^{(j)}S_{j_3}^{(j)}\right)  
%  (m+2) c_j^m d_m^2+ \sum_{j=1}^{k_2} \sum_{l=2}^{m-1} \left[ (-1)^{m-l+1}l c_j^l d_j^2 + 2(l+1)(-1)^{m-l}c_j^l d_j^2 \right.\\
%  &\rule{5mm}{0mm}+ \left.(l+2)(-1)^{m-l-1}c_j^l d_j^2 %\right]\, 
%\left(\sum_{j_1+j_2+j_3=m-l-1}T_{j_1}S_{j_2}^{(j)}S_{j_3}^{(j)}\right) \, 2^{m-2} \\
%&=(-2)^{m-1}  \sum_{j=1}^{k_2} 
%\left(\sum_{j_1+j_2+j_3=m-1}T_{j_1}S_{j_2}^{(j)}S_{j_3}^{(j)}\right)
%d_j^2 - (-2)^{m-2}\sum_{j=1}^{k_2} 
%\left(\sum_{j_1+j_2+j_3=m-2}T_{j_1}S_{j_2}^{(j)}S_{j_3}^{(j)}\right)
% c_jd_j^2.
%\end{align*}
The resulting equality holds also if $k_1=0$ or $k_2=0$.
We have finally for the term in $x^m$ and $1<m<k+1$ on the left-hand side of Eq.~\eqref{070416*}:
\begin{align}
&  2^{k+1} (ix)^{m}\left(
\frac{P^{(k+1-m+2)}(0)}{(k+1-m+2)! 2^{k-m+2}} \E[\Gamma_1(X)] + \ldots + \frac{P^{(k+2)}(0)}{(k+2)!2^{k+1}}\E[\Gamma_m(X)]\right)\nonumber \\
&= (ix)^m (-2)^{m-1} \sum_{j=1}^{k_2} \left(
\sum_{j_1+j_2+j_3\atop=m-1}T_{j_1}S_{j_2}^{(j)}S_{j_3}^{(j)} \right) (2c_j^2+d_j^2 )\nonumber \\
&\rule{5mm}{0mm}+ (ix)^m  (-2)^{m-2} \sum_{j=1}^{k_2} \left(
\sum_{j_1+j_2+j_3\atop=m-2}T_{j_1}S_{j_2}^{(j)}S_{j_3}^{(j)} \right) (-4c_j^3-c_jd_j^2)\nonumber \\
&\rule{5mm}{0mm}+  (ix)^m (-2)^{m-1}  \sum_{j=1}^{k_1} 
\left(
\sum_{j_1+j_2+j_3\atop=m-1}T_{j_1}^{(j)}S_{j_2}S_{j_3}  \right)2b_j^2 \nonumber \\
&\rule{5mm}{0mm}+ (ix)^m(-2)^{m-1} a^2 \left(
\sum_{j_1+j_2+j_3\atop=m-1}T_{j_1}S_{j_2}S_{j_3} \right) . \label{030416a}
\end{align}
Consider now the term in $x^m$  for $1<m<k+1$ on the right-hand side of Eq.~\eqref{040416d}. Define:
\[
Q_m\left( \sum_{i=0}^\infty c_i x^i \right) = c_m x^m.
\] 
$Q_m$ is thus the projection of a series (or a polynomial) on its term of degree $m$. The terms in $x^m$ for $1<m<k+1$ on the right-hand side of Eq.~\eqref{040416d} are given by:
\begin{align}
& ixa^2 Q_{m-1}(G_1(x)G_2(x))+  2ix Q_{m-1}\left(G_2(x)\sum_{l=1}^{k_1}b_l^2 \prod_{j\ne l}(1-2ixb_j)\right)\nonumber \\
&\rule{5mm}{0mm}+ ix Q_{m-1}\left( G_1(x) \sum_{l=1}^{k_2} \Delta_l (1-2ixc_l) \prod_{j\ne l}(1-2ixc_j)^2\right)\nonumber \\
&\rule{5mm}{0mm}+  (ix)^2 Q_{m-2}\left( G_1(x) \sum_{l=1}^{k_2} c_l \Delta_l  \prod_{j\ne l} (1-2ixc_j)^2\right)\nonumber \\
&\rule{5mm}{0mm}-2  ix Q_{m-1} \left( G_1(x) \sum_{l=1}^{k_2}c_l^2 \prod_{j\ne l} (1-2ixc_j)^2 \right).\label{060416a}
\end{align}
We calculate now the projections appearing in the expression above.
\begin{align*}
&G_1(x)G_2(x)= \sum_{j_1=0}^{k_1}(-2ix)^{j_1}T_{j_1} \sum_{j_2=0}^{k_2}(-2ix)^{j_2}S_{j_2} \sum_{j_3=0}^{k_2} (-2ix)^{j_3} S_{j_3},\\
&Q_{m-1}(G_1(x)G_2(x))= (-2ix)^{m-1} \left(\sum_{j_1+j_2+j_3\atop=m-1}T_{j_1}S_{j_2}S_{j_3}\right),\\
&G_2(x) \sum_{l=1}^{k_1} b_l^2 \prod_{j\ne l}(1-2ixb_l) = \sum_{l=1}^{k_1} b_l^2  \sum_{j_1=0}^{k_1-1} (-2ix)^{j_1} T_{j_1}^{(l)} \sum_{j_2=0}^{k_2}(-2ix)^{j_2}S_{j_2} \sum_{j_3=0}^{k_2}(-2ix)^{j_3} S_{j_3},\\
&Q_{m-1}\left( G_2(x) \sum_{l=1}^{k_1} b_l^2 \prod_{j\ne l}(1-2ixb_l)  \right)= (-2ix)^{m-1}\sum_{l=1}^{k_1} b_l^2  \left( \sum_{j_1+j_2+j_3\atop=m-1} T_{j_1}^{(l)} S_{j_2}S_{j_3} \right)  .
\end{align*}
Notice that these equalities    hold if $k_1\ge 0$ and $k_2\ge 0$. The following equalities can be derived as above:
\begin{align*}
&Q_{m-1}\left( G_1(x)\sum_{l=1}^{k_2}\Delta_l (1-2ixc_l) \prod_{j\ne l} (1-2ixc_j)^2 \right)\\
&= (-2ix)^{m-1}\left[ \sum_{l=1}^{k_2} \Delta_l
\left(
\sum_{j_1+j_2+j_3\atop =m-1} T_{j_1}S_{j_2}^{(l)}S_{j_3}^{(l)} \right) +   \sum_{l=1}^{k_2}c_l \Delta_l \left( \sum_{j_1+j_2+j_3\atop=m-2}T_{j_1} S_{j_2}^{(l)} S_{j_3}^{(l)} \right)\right], \\
&Q_{m-2}\left( G_1(x) \sum_{l=1}^{k_2} c_l\Delta_l \prod_{j\ne l}(1-2ixc_j)^2\right) = (-2ix)^{m-2}\sum_{l=1}^{k_2}c_l\Delta_l \left(
\sum_{j_1+j_2+j_3\atop=m-2} T_{j_1}S_{j_2}^{(l)}S_{j_3}^{(l)} \right),\\
&Q_{m-1}\left( G_1(x)\sum_{l=1}^{k_2}c_l^2 \prod_{j\ne l}(1-2ixc_j)^2 \right) = (-2ix)^{m-1}\sum_{l=1}^{k_2} c_l^2 \left( 
\sum_{j_1+j_2+j_3\atop=m-1} T_{j_1} S_{j_2}^{(l)}S_{j_3}^{(l)}
\right).
\end{align*}
%
\begin{comment}
\begin{align*}
Q_{m-1}\left( G_1(x)\sum_{l=1}^{k_2}\Delta_l (1-2ixc_l) \prod_{j\ne l} (1-2ixc_j)^2 \right)&= (-2ix)^{m-1} \sum_{l=1}^{k_2} \Delta_l
\left(
\sum_{j_1+j_2+j_3=m-1} T_{j_1}S_{j_2}^{(l)}S_{j_3}^{(l)} \right) \\
&\rule{5mm}{0mm}+ (-2ix)^{m-1} \sum_{l=1}^{k_2}c_l \Delta_l \left( \sum_{j_1+j_2+j_3=m-2}T_{j_1} S_{j_2}^{(l)} S_{j_3}^{(l)} \right), \\
Q_{m-2}\left( G_1(x) \sum_{l=1}^{k_2} c_l\Delta_l \prod_{j\ne l}(1-2ixc_j)^2\right) &= (-2ix)^{m-2}\sum_{l=1}^{k_2}c_l\Delta_l \left(
\sum_{j_1+j_2+j_3=m-2} T_{j_1}S_{j_2}^{(l)}S_{j_3}^{(l)} \right),\\
Q_{m-1}\left( G_1(x)\sum_{l=1}^{k_2}c_l^2 \prod_{j\ne l}(1-2ixc_j)^2 \right) &= (-2ix)^{m-1}\sum_{l=1}^{k_2} c_l^2 \left( 
\sum_{j_1+j_2+j_3=m-1} T_{j_1} S_{j_2}^{(l)}S_{j_3}^{(l)}
\right).
\end{align*}
\end{comment}
%
We have thus for   expression \eqref{060416a}:
\begin{align}
&  ix a^2 (-2ix)^{m-1}\left( \sum_{j_1+j_2+j_3\atop=m-1}T_{j_1}S_{j_2}S_{j_2} \right) + 2  ix(-2ix)^{m-1}\sum_{l=1}^{k_1}b_l^2 \left( \sum_{j_1+j_2+j_3\atop=m-1} T_{j_1}^{(l)}S_{j_2}S_{j_3} \right)\nonumber \\
&\rule{5mm}{0mm}+   ix(-2ix)^{m-1}\sum_{l=1}^{k_2} \Delta_l \left( \sum_{j_1+j_2+j_3\atop=m-1} T_{j_1}S_{j_2}^{(l)}S_{j_3}^{(l)}\right) \nonumber \\
&\rule{5mm}{0mm}+  ix(-2ix)^{m-1} \sum_{l=1}^{k_2}c_l\Delta_l  \left( \sum_{j_1+j_2+j_3\atop=m-2}T_{j_1}S_{j_2}^{(l)}S_{j_3}^{(l)} \right)\nonumber\\
&\rule{5mm}{0mm}+ (ix)^2(-2ix)^{m-2}\sum_{l=1}^{k_2}c_l\Delta_l\left( \sum_{j_1+j_2+j_3\atop=m-2}T_{j_1}S_{j_2}^{(l)}S_{j_3}^{(l)} \right)\nonumber \\
&\rule{5mm}{0mm}-2  ix(-2ix)^{m-1}\sum_{l=1}^{k_2}c_l^2 \left( \sum_{j_1+j_2+j_3\atop=m-1} T_{j_1}S_{j_2}^{(l)}S_{j_3}^{(l)} \right)\nonumber \\
&=  (ix)^m \left\lbrace(-2)^{m-1}\left( \sum_{j_1+j_2+j_3\atop=m-1}T_{j_1}S_{j_2}S_{j_3}\right)a^2+2   (-2)^{m-1} \sum_{l=1}^{k_1}\left( \sum_{j_1+j_2+j_3\atop=m-1}T_{j_1}^{(l)}S_{j_2}S_{j_3}\right)b_l^2\right.\nonumber\\
&\rule{5mm}{0mm}+ \left.(-2)^{m-1}  \sum_{l=1}^{k_2}(2c_l^2+d_l^2) \left( \sum_{j_1+j_2+j_3\atop=m-1}T_{j_1}S_{j_2}^{(l)}S_{j_3}^{(l)}\right) \right.\nonumber \\
&\rule{5mm}{0mm}- \left.  (-2)^{m-2} \sum_{l=1}^{k_2}c_l\Delta_l \left( \sum_{j_1+j_2+j_3\atop=m-2}T_{j_1}S_{j_2}^{(l)}S_{j_3}^{(l)}\right)\right\rbrace.\label{030416b}
\end{align}
This equality  holds (again) for $k_1\ge 0$ and $k_2\ge 0$.
Comparing Eq.~\eqref{030416a} and Eq.~\eqref{030416b}, we find the desired equality of the terms in $x^m$ on the left-hand and the right-hand side of Eq.~\eqref{040416d}.% Considering the previous remarks, this concludes the proof.
\end{proof}

\section*{Acknowledgements} I would like to thank two anonymous
referees for their comments and remarks, as well as 
Prof. Dr. G.
Peccati and Prof. Dr. A. Thalmaier for their time and many interesting
discussions and suggestions.

  \bibliographystyle{plain}
\bibliography{references_MR_14}

\end{document}